\documentclass[12pt,reqno]{amsart}
\usepackage{amsmath} 

\usepackage{amsthm}
\usepackage{amssymb}
\usepackage{graphics}
\usepackage{latexsym}

\numberwithin{equation}{section}
\newtheorem{thm}{Theorem}[section]
\newtheorem{prop}[thm]{Proposition}
\newtheorem{lem}[thm]{Lemma}
\newtheorem{cor}[thm]{Corollary}

\theoremstyle{remark}
\newtheorem{rem}[thm]{Remark}
\newtheorem{defn}{Definition}

\newcommand{\BBB}{\mathbb}
\newcommand{\R}{{\BBB R}}
\newcommand{\Z}{{\BBB Z}}
\newcommand{\T}{{\BBB T}}

\newcommand{\N}{{\BBB N}}
\newcommand{\C}{{\BBB C}}
\newcommand{\LR}[1]{{\langle {#1} \rangle }}
\newcommand{\lec}{{\ \lesssim \ }}
\newcommand{\gec}{{\ \gtrsim \ }}

\newcommand{\al}{\alpha}
\newcommand{\be}{\beta}
\newcommand{\ga}{\gamma}

\newcommand{\vp}{\varphi}

\newcommand{\e}{\varepsilon}

\newcommand{\p}{\partial}

\newcommand{\La}{\Lambda}

\newcommand{\de}{\delta}

\newcommand{\supp}{\operatorname{supp}}

\newcommand{\dis}{\displaystyle}
\newcommand{\mT}{\mathcal{T}}

\newcommand{\EQ}[1]{\begin{equation} \begin{split} #1
 \end{split} \end{equation}}
\newcommand{\EQS}[1]{\begin{align} #1 \end{align}}
\newcommand{\EQQS}[1]{\begin{align*} #1 \end{align*}}
\newcommand{\EQQ}[1]{\begin{equation*} \begin{split} #1
 \end{split} \end{equation*}}

\newcommand{\ti}{\widetilde}
\newcommand{\ha}{\widehat}

\title[LWP of fifth order KdV type equations]
{Cancellation properties and unconditional well-posedness for the fifth order KdV type equations with periodic boundary condition
}
\author[T. K.  Kato]{Takamori Kato}
\author[K. Tsugawa]{Kotaro Tsugawa}
\address[T.K. Kato]{Mathematical Science Course, Faculty of Science and Engineering, Saga University, Saga, 840-8502, 
Japan}
\address[K. Tsugawa]{Department of Mathematics, Faculty of Science and Engineering, Chuo University,
Bunkyo-ku, Tokyo, 112-8551, Japan}
\email[T.K. Kato]{tkkato@cc.saga-u.ac.jp}
\email[K. Tsugawa]{tsugawa@math.chuo-u.ac.jp}

\begin{document}

\begin{abstract}
We consider the fifth order KdV type equations and prove the unconditional well-posedness in $H^s(\T)$ for $s \ge 1$.
It is optimal in the sense that the nonlinear terms can not be defined in the space-time distribution framework for $s<1$.
The main idea is to employ the normal form reduction and a kinds of cancellation properties to deal with the derivative losses. 
\end{abstract}

\maketitle
\setcounter{page}{001}

\section{Introduction}
We consider the Cauchy problem of the fifth order KdV type equations as follows:
\begin{align} \label{5KdV1}
& \p_t u +\p_x^5 u + \alpha \p_x (\p_x u)^2 + \beta \p_x (u \p_x^2 u) + 10 \ga \p_x(u^3) =0, \\
& u(0,x)= \vp (x) \label{ini}
\end{align}
where $\alpha, \be, \ga \in \R$, $u(t,x): \R \times \T \to \R$ and $\vp(x): \T \to \R$. 
There are many results for the KdV equation:
\begin{align} \label{KdV}
\p_t u+ \p_x^3 u = \p_x (u^2).
\end{align} 
However, known results for the fifth order KdV type equations are few. 
This is because the strong singularity in the nonlinear terms makes the problem difficult.
The case of $x\in \R$ was studied in \cite{GKK}, \cite{KP}, \cite{Kwo} and \cite{Ponce}.
When $x \in \T$, the linear part does not have any smoothing effects and that makes the problem more difficult.
The case of $x\in \T$ was studied in \cite{TKK} and \cite{Kwa} under the assumption $\al=\be/2$ and studied in \cite{Mc} and \cite{T} for general $\al$, $\be$ and $\ga$.
Note that the cancellation properties, that is Propositions \ref{prop_res1} and \ref{prop_res2} were found in \cite{TKK}, which is written by the first author (see Propositions 5.4 and 5.5 in \cite{TKK}).

In \cite{T}, the second author has proven the local well-posedness of fifth order dispersive equations with sufficiently smooth initial data.  
The result includes the local well-posedness of (\ref{5KdV1}) with (\ref{ini}) for sufficiently large $s$. 
In the present paper, we are especially interested in the low regularity problem. 
Our main result is as below.

\begin{thm} \label{thm_main}
Let $s \ge 1$ and $\vp \in H^s(\T)$. Then, there exist 
$T=T(\| \vp \|_{H^s})>0$ and a unique solution $u \in C([-T, T] :H^s(\T))$ to \eqref{5KdV1} with \eqref{ini}. 
Moreover, the solution map $H^s(\T) \ni \vp \mapsto  u \in C([-T, T] :H^s(\T) )$ is continuous. 
\end{thm}

\begin{rem}\label{rem_thm_main}
Let $u \in C([-T, T] :H^s(\T))$ be a solution to (\ref{5KdV1}).
When $s \ge 1$, $\p_x(u^3)$, $\p_x (\p_x u)^2$ and $\p_x (u \p_x^2 u)=\p^3_x (u^2)/2-\p_x (\p_x u)^2$ are in  
$C([-T, T] :H^{s-3}(\T))$.
Thus, (\ref{5KdV1}) make sense in $C([-T, T] :H^{s-5}(\T))$.
On the other hand,
when $s<1$, $\p_x (\p_x u)^2$ and $\p_x (u \p_x^2 u)$ can not be defined even in the space-time distribution framework. 
Therefore, the unconditional uniqueness in Theorem \ref{thm_main} for $s=1$ is optimal except the case $\al=\be$.

When $s \ge 1/6$, $\p_x(u^3)$ and $\p_x (\p_x u)^2 + \p_x (u \p_x^2 u)=\p^3_x (u^2)/2$ are in $C([-T, T] :H^{s-10/3}(\T))$.
Thus, (\ref{5KdV1}) make sense in $C([-T, T] :H^{s-5}(\T))$ if $\al =\be$.
Therefore, the unconditional uniqueness is likely to hold even when $s<1$ if $\al =\be$. 
\end{rem}

When $u \in C([-T, T] :H^1(\T))$ is a solution to (\ref{5KdV1}) with (\ref{ini}), the following conservation law holds:
\begin{align} \label{cons1}
E_0(u)(t):= \int_{\T} u(t, x) \, dx = \int_{\T} \vp(x) \, dx \hspace{0.5cm} (t \in [-T, T] )
\end{align}
where
\begin{align*}
\int_{\T} f(x) \, dx := \frac{1}{ 2 \pi} \int_0^{2\pi} f(x) \, dx
\end{align*}
for a $2 \pi$-periodic function $f$. Moreover, when $\alpha= \beta/2$, 
\begin{align*}
H(u)(t):= \frac{1}{2} \int_{\T} \Big\{ (\p_x^2 u (t,x))^2+ 5 \ga u^4(t,x)- \beta u(t,x) (\p_x u(t,x))^2  \Big\} \, dx
\end{align*} 
is conserved. In this case, (\ref{5KdV1}) can be regarded as the Hamiltonian PDE $\p_t u = -\p_x  H'(u)$. 
In the case $(\alpha, \beta, \gamma)= (5,  10, 1)$ or $(-5, -10, 1)$, (\ref{5KdV1}) is called the fifth order KdV equation, 
which is completely integrable and enjoys infinitely conserved quantities. 
Well-posedness for this case was studied in \cite{BKV} on real line and in \cite{KM} on the torus.
See \cite{KT} and \cite{KV} for related results on the KdV equation.

When $\alpha=\be/2$, by the conserved quantity $H(u)$, we have the following global results as a corollary of Theorem~\ref{thm_main}. 

\begin{cor} 
Let $\al= \be/2$ and $\vp \in H^2(\T)$. 
Then, the solution obtained in Theorem~\ref{thm_main} can be extended to the solution on $t \in (-\infty,\infty)$.  
\end{cor}

In our problem, main difficulty comes from the derivative losses.
We will use the normal form reduction to deal with them. 
The normal form reduction was introduced by Shatah in \cite{Sh} to study the existence of the global solution of quadratic nonlinear Klein-Gordon Equations for small initial data.
It was used to recover derivative losses by Takaoka-Tsutusmi in \cite{TaTs} in the study of the well-posedness of the mKdV equation 
and 
by Babin-Ilyin-Titi in \cite{BIT} in the study of the well-posedness of the KdV equation. 
For the studies of unconditional uniqueness by the normal form reduction, 
see \cite{GKO}, \cite{Kn1}, \cite{Kn2}, \cite{Kn3}, \cite{Kn4}, \cite{Kn5}, \cite{KOY} and \cite{KO}.

Since the normal form method employs the oscillation effect, it does not work for resonant parts. 
Therefore, we need to remove the resonant parts with derivative losses. 
We will cancel them by the conserved quantity $E_0$ and changing the variables $x$ to $x+ c(t)$.
This is the main idea in this paper and we explain it below.

In \cite{Bo}, Bourgain separated the nonlinear terms of \eqref{KdV} into the resonant parts with derivative losses and the others as follows: 
For a solution $u$ of (\ref{KdV}), we put $v:= e^{t \p_x^3} u$. Then, $\ha{v}(t,k)$ satisfies 
\begin{equation*}
\p_t \ha{v}(t, k) = \sum_{ \substack{ k=k_1+k_2 \\ k_1 k_2 (k_1+k_2) \neq 0}} e^{-3itk_1 k_2 (k_1+k_2)} \, i (k_1+k_2) \prod_{l=1}^2 \ha{v}(t, k_l) 
+ 2 \ha{v}(t, 0) \, ik \ha{v}(t, k)
\end{equation*}
for each $k \in \Z$.
We can recover the derivative loss of the first term of the right-hand side 
by using the advantage of the oscillation in $e^{-3it k_1 k_2 (k_1+k_2)}$.
The second term is troublesome since it has one derivative loss and is resonant part, that is the interaction with $k_1k_2 (k_1+k_2)=0$.
However, since $\ha{v}(t, 0)=\ha{u}(t,0)= \int_{\T} u \, dx$ is conserved 
for solutions of \eqref{KdV}, we can regard the second term as a linear term.
From this point of view, we rewrite \eqref{KdV} into
\begin{align*}
\p_t u + \p_x^3 u  - 2 \int_{\T} u \, dx \, \p_x u = 2 \Big( u - \int_{\T} u\, dx \Big) \p_x u.
\end{align*}
Then we can treat the right-hand side as a perturbation and solve it.
Also, in \cite{Bo}, Bourgain separated the nonlinear terms of the mKdV equation
\begin{align} \label{mKdV}
\p_t u +\p_x^3 u = \p_x(u^3)
\end{align}
into the resonant parts with derivative losses and the others as follows: 
For a solution $u$ of (\ref{mKdV}), we put $v:= e^{t \p_x^3} u$. 
Then, $\ha{v}(t,k)$ satisfies 
\begin{align} \label{mKdV2}
\p_t \ha{v}(t,k) = 
& \sum_{ \substack{ k=k_1+k_2+k_3 \\ (k_1+k_2)(k_2+k_3)(k_1+k_3) \neq 0 }} 
e^{-3it (k_1+k_2) (k_2+k_3) (k_1+k_3)} \, i(k_1+k_2+k_3) \prod_{l=1}^3 \ha{v}(t, k_l) \notag \\
&\hspace{0.5cm} -3ik |\ha{v}(t, k)|^2 \ha{v}(t,k)
 + 3 \sum_{k_1 \in \Z} |\ha{v}(t, k_1)|^2 \, ik \ha{v}(t,k)
\end{align}
for each $k \in \Z$.
We can recover the derivative loss in the first term of the right-hand side by using the advantage of 
the oscillation in $e^{-3it (k_1+k_2) (k_2+k_3) (k_1+k_3)}$.
The second term does not have the derivative loss since
\[
\LR{k}^s(|k| |\ha{v}(t, k)|^3) \le (\LR{k}^s|\ha{v}(t, k)|)^3
\]
for $s\ge 1/2$.
The third term is troublesome since it has one derivative loss and is resonant part, that is the interaction with $(k_1+k_2) (k_2+k_3) (k_1+k_3) =0$. 
However, since
\[
\sum_{k_1 \in \Z} |\ha{v}(t, k_1)|^2= \sum_{k_1 \in \Z} |\ha{u}(t, k_1)|^2 =  \int_{\T} u^2 \, dx
\]
is conserved for solutions of \eqref{mKdV}, we can regard the third term as a linear term.
From this point of view, we rewrite \eqref{mKdV}
into
\begin{align} \label{mKdV3}
\p_t u +\p_x^3 u - 3 \int_{\T} u^2 \, dx \, \p_x u= 3 \Big( u^2 - \int_{\T} u^2 \, dx  \Big) \p_x u .
\end{align}
Then we can treat the right-hand side as a perturbation and solve it.

Now we apply the same argument to our problem.
From the same points of view and \eqref{cons1}, we rewrite (\ref{5KdV1}) into
\begin{equation}\label{5KdV1.5}
\begin{split}
& \p_t u + \p_x^5 u + \be E_0(\vp) \p_x^3 u + 30 \ga \int_{\T} u^2 \, dx \, \p_x u \\
&=- \al \p_x (\p_x u)^2 - \be \p_x \Big\{ \Big(u -\int_{\T} u\, dx  \Big)  \p_x^2 u \Big\} -30 \ga \Big( u^2 - \int_{\T} u^2 \, dx  \Big) \p_x u.
\end{split}
\end{equation}
Then, the right-hand side include only non-resonant parts or no derivative loss parts.
Therefore, they look harmless.
In fact, however, we can not treat them as perturbations.
Roughly speaking, Lemmas \ref{Le1}, \ref{Le2} and \ref{Le3} mean that
only one derivative loss can be recovered by applying the normal form reduction once. 
Since the nonlinear terms include the third derivative, we need to apply it three times.
After we apply the normal form reduction to the first and the second terms of the right-hand side once, cubic resonant parts with derivative losses appear, which cause trouble.
To deal with them, we add $J_2(u)$ below on the both sides of \eqref{5KdV1.5} and use cancellation properties (see Propositions \ref{prop_res1}, \ref{prop_res2} and \ref{prop_res3}).
That is, we rewrite \eqref{5KdV1} into
\begin{align} \label{5KdV2}
 \p_t u + \p_x^5 u + \beta E_0(\vp) \p_x^3 u +K(u) \p_x u 
 = J_1(u)+ J_2(u) +J_3(u).
\end{align}
where 
\begin{align*}
& J_1(u)= -\alpha \p_x (\p_x u)^2 - \beta \p_x \Big\{ \Big(u - \int_{\T} u\, dx \Big) \p_x^2 u \Big\}, \\
& J_2(u)=  - \frac{\be^2}{5} \Big\{ \int_{\T} u^2 \, dx- \Big( \int_{\T} u \, dx \Big)^2  \Big\} \p_x u,  \\
& J_3(u)= -30 \gamma \Big( u^2 - \int_{\T} u^2 \, dx  \Big) \p_x u,  \\
& K(u)=\big(30 \gamma - \frac{ \be^2 }{5} \big) \int_{\T} u^2 \, dx + \frac{\be^2}{5} \Big( \int_{\T} u \, dx \Big)^2. 
\end{align*}
Note that $K(u)$ does not depend on $x$. Thus by the changing variables 
\begin{align} \label{change1}
u(t, x) \mapsto u \Big( t, x + \int_0^t [K(u)](t') \, dt'  \Big), 
\end{align}
(\ref{5KdV2}) is rewritten into 
\begin{align} \label{5KdV3}
\p_t u + \p_x^5 u + \beta E_0(\vp) \p_x^3  u = J_1(u)+ J_2(u)+J_3(u).
\end{align}
We only need to show the local well-posedness of \eqref{5KdV3} with \eqref{ini} (see Propositions~\ref{prop_equiv0} and \ref{prop_equiv}). 
Since $J_3(u)$ have only one derivative loss, we only need to apply the normal form reduction once.
After we apply the normal form reduction to $J_1(u)$ the first time, cubic parts with derivative losses appear.
For the resonant parts of them, which correspond to the symbols $M_{4,\vp}^{(3)}$ and $M_{6, \vp}^{(3)}$ defined in Proposition \ref{prop_NF2}, we use a kind of cancellation properties (see Propositions \ref{prop_res1} and \ref{prop_res3}).
Here, $J_2(u)$ corresponds to the symbol $M_{2,\vp}^{(3)}$ and plays an important role in this cancellation. 
Since the non-resonant parts of them still have two derivative losses, we apply the normal form reduction the second time.
Then quartic parts with one derivative loss appear.
For the resonant part of them, which corresponds to the symbol $M_{5,\vp}^{(4)}$, we use a kind of cancellation properties (see Propositions \ref{prop_res2}).
For the non-resonant parts of them, we apply the normal form reduction the third time.
Then, all derivative losses can be recovered.
This is our main idea in this paper. 

Recently, we have known the paper \cite{Mc} by R. McConnell.
In this paper, he also consider the same problem.
By combining the Fourier restriction norm method and the normal form reduction, he proved the ``conditional'' local well-posedness of \eqref{5KdV1} in $H^{s}(\T)$ for $s >35/64$.
He also proved the ``unconditional'' local well-posedness of \eqref{5KdV1} in $H^s(\T)$ for $s>1$.
The latter result is almost same as our main result.
However, there are the following differences:
(i) The sharp case $s=1$ was not proved in his paper. 
(ii) He consider only the case $E_0(\vp)=0$.
In our paper, we will consider the problem without the assumption $E_0(\vp)=0$.
This makes our proof more complicated since the phase function $\Phi_\vp^{(N)}$
depends on $E_0(\vp)$.
For instance, we prepare \eqref{ff2}, \eqref{ff3}, \eqref{2eq6} and \eqref{C42} to estimate $M_{9,f}^{(3)}$ and $M_{10,f}^{(4)}$ in Lemma \ref{lem_pwb2} and to prove Lemma \ref{rem_pwb3}. 
Note that the constant $C_2=0$ when $E_0(\vp_1)= E_0(\vp_2)$ in Corollary \ref{cor_mainest}.
Therefore, we can show the solution map in Proposition \ref{prop_LWP} is locally Lipschitz continuous under the assumption $E_0(\vp)=0$.
However, we have only the continuous of the solution map in Proposition \ref{prop_LWP}
since we do not assume  $E_0(\vp)=0$.

Finally, we give some notations. 
We write $k_{i, i+1, \dots, j}$ to mean $k_i+ k_{i+1}+ \dots+ k_j$ for integers $i$ and $j$ satisfying $i<j$.
$k_{\max}$ is denoted by $k_{\max}:=\max_{1\le j\le N} \{|k_j|\}$.
We will use $A \lesssim B$ to denote an estimate of $A \le C B$ for some positive constant $C$ and 
write $A \sim B$ to mean $A \lesssim B$ and $B \lesssim A$. 
$\| \cdot \|_{L_T^{\infty} X}$ is denoted by $\| \cdot \|_{L_T^{\infty}X}:= \sup_{t \in [-T, T]} \|  \cdot  \|_{X}$ for a Banach space $X$.
For $s \in \R$, $l_s^2$ is denoted by
$l_s^2:=\big\{ f : \Z \to \C: \| f \|_{l_s^2}:= \| \langle \cdot \rangle^s f \|_{l^2} < \infty   \big\}$.

\section*{Acknowledgement}
We would like to thank Professor Nobu Kishimoto for his valuable comments. 
The first author was supported by JSPS KAKAEHI Grant Number JP26800070 and JP18K13442 
and the second author was supported by JSPS KAKENHI Grant Number JP17K05316.


\section{notations and preliminary lemmas}

In this section, we prepare some lemmas to prove main theorem. 
First, we give some notations. 
For a $2 \pi$-periodic function $f$ and a function $g$ on $\Z$, we define the Fourier transform and the inverse Fourier transform by 
\begin{align*}
(\mathcal{F}_x  f )(k):= \ha{f}(k):= \int_{\T} e^{-ixk} f(x) \, dx, \hspace{0.5cm}
(\mathcal{F}^{-1} g) (x):= \sum_{k \in \Z} e^{ixk} g(k). 
\end{align*} 
Then we have 
\begin{align*}
f= \mathcal{F}^{-1} (\mathcal{F}_x f), \hspace{0.5cm} 
\| f \|_{L^2}:= \Big( \int_{\T} |f(x)|^2 \, dx \Big)^{1/2}= \Big( \sum_{k \in \Z} |\ha{f} (k)|^2 \Big)^{1/2}=\|\ha{f}\|_{l^2}. 
\end{align*}

We give some estimates on the phase function $\Phi_{\vp}^{(N)}$ defined as 
\begin{align*}
&\Phi_{\vp}^{(N)}:= \Phi_{\vp}^{(N)} (k_1, \dots, k_N):= \phi_{\vp}(k_{1,\dots, N})- \sum_{j=1}^N \phi_{\vp} (k_j),\\
&\phi_{\vp} (k):= -ik^5 + i \beta  E_0(\vp) k^3 = \mathcal{F}_x(-\p_x^5 - \beta E_0(\vp) \p_x^3)
\end{align*}
for $N \in \N$, which plays an important role to recover some derivatives when we estimate non-resonant parts of nonlinear terms.
A simple calculation yields that $\Phi_{f}^{(1)}=0$, 
\begin{align}
&\Phi_{f}^{(2)}= -\frac{5}{2} i k_{1} k_{2} k_{1,2} \big( k_1^2 + k_2^2 +k_{1,2}^2- \frac{6}{5} \be E_0(f) \big), \label{2eq1} \\
& \Phi_{f}^{(3)}= -\frac{5}{2} i k_{1,2} k_{2,3} k_{1,3} \big( k_{1,2}^2 + k_{2,3}^2 +k_{1,3}^2-  \frac{6}{5} \be E_0(f) \big) \label{2eq2}
\end{align}
for $f \in L^2(\T)$. We easily prove the following lemmas by the factorization formula \eqref{2eq1}--\eqref{2eq2}.

\begin{lem} \label{Le1}
Assume that $f, g \in L^2(\T)$, $k_{\max} > 4 \max\{ |\be E_0(f)|, |\be E_0(g)| \}$ and $k_1k_2 k_{1,2} \neq 0$.
Then it follows that
\begin{align}
& |\Phi_f^{(2)}|  \gtrsim \min\{|k_1|, |k_2|, |k_{1,2} |  \} k_{\max}^4, \label{ff1}\\
& \Big| \frac{1}{\Phi_{f}^{(2)}}- \frac{1}{\Phi_g^{(2)}} \Big|   \lesssim \frac{|\be| |E_0(f) -E_0(g) |}{ k_{\max}^2 } 
\min \left\{ \frac{1}{|\Phi_f^{(2)}|}, \frac{1}{|\Phi_g^{(2)}|} \right\}. \label{ff2}
\end{align}
\end{lem}

\begin{lem} \label{Le2}
Assume that $f,g \in L^2(\T)$, $k_{\max}> 16 \max\{ | \be E_0(f) |, |\be E_0(g)| \}$ and $k_{1,2} k_{2,3} k_{1,3} \neq 0$. If 
\begin{align} \label{rel1}
|k_1| \sim |k_2| \sim |k_3|,
\end{align}
then it follows that 
\begin{align}
& |\Phi_f^{(3)}|  \gtrsim \min\{ |k_{1,2} | |k_{1,3}|, |k_{1,2}| |k_{2,3}|, |k_{1,3}| |k_{2,3}|  \}  \, k_{\max}^3,\notag \\
& \Big| \frac{1}{\Phi_{f}^{(3)}}- \frac{1}{\Phi_g^{(3)}} \Big|   \lesssim \frac{|\be| |E_0(f) -E_0(g) |}{ k_{\max}^2} 
\min \left\{  \frac{1}{|\Phi_f^{(3)}|}, \frac{1}{|\Phi_g^{(3)}|} \right\}.\label{ff3}
\end{align}
If \eqref{rel1} does not hold, then it follows that
\begin{align} \label{ff4}
|\Phi_f^{(3)}|  \gtrsim   \min \{  |k_{1,2}|, |k_{1,3}|, |k_{2,3}|  \} \, k_{\max}^4
\end{align}
and \eqref{ff3}.
\end{lem}

We need estimates similar to \eqref{ff1}--\eqref{ff2} and \eqref{ff3}--\eqref{ff4} for $N \ge 4$ to recover derivative losses. 
However, no factorization formula are known for $N \ge 4$ and the following proposition means that
it is impossible to gain $k_{\max}^4$ by the phase function $\Phi_{f}^{(4)}$
under the assumption $k_{1,2,3} \neq 0$ and 
\begin{equation}
|k_4| \gg \max \{  |k_1|, |k_2|, |k_3| \}.\label{rel2}
\end{equation}
\begin{prop}
\[
\inf \left\{ |\Phi_f^{(4)}||k_4|^{-4} \, \Big| \,  |k_4| \gg \max \{  |k_1|, |k_2|, |k_3| \}, |k_4| \gec |\be E_0(f)|, k_{1,2,3}\neq 0\right\}=0
\]
\end{prop}
\begin{proof}
Since $k_{1,2,3,4}^3- \sum_{j=1}^4 k_j^3$ is equal to
\[
(k_{1,2,3,4}^3 -k_{1,2,3}^3- k_4^3)+ \big( k_{1,2,3}^3- \sum_{j=1}^3 k_j^3  \big) =3 (k_4 k_{1,2,3} k_{1,2,3,4} +  k_{1,2} k_{2,3} k_{1,3}),
\]
if follows that
\[
\Phi_f^{(4)}=-i \big\{ k_{1,2,3,4}^5- \sum_{i=1}^4 k_i^5 \big\} + 3 i \be E_0(f) (k_4 k_{1,2,3}  k_{1,2,3,4} +k_{1,2} k_{2,3} k_{1,3} ). 
\]
Here, we take $k_4=n^{12}$, $k_3=n^{11}+n^{4}$, $k_2= - n^{11}$, $k_1=-n^4+1$ and $n\to \infty$. Then,
\begin{align*}
|\Phi_f^{(4)}| & \lesssim |(n^{12}+1)^5- (n^{12})^5 - (n^{11}+n^4)^5- (-n^{11})^5 - (-n^4+1)^5  | \\
&+ |\be E_0(f)|(n^{12}(n^{12}+1)+(n^{11}+n^{4}-1)n^4(n^{11}+1)) \lesssim n^{41}.
\end{align*}
\end{proof}
Therefore, we use slightly stronger assumption \eqref{rel3}, \eqref{rel4} or \eqref{rel30} instead of \eqref{rel2} in the following lemma. 

\begin{lem} \label{Le3}
Assume $f, g \in L^2(\T)$, $|k_4| > 16 \max \{1, |\be E_0(f)|, | \be E_0(g)|\}$, $k_{1,2, 3} \neq 0$. 
If at least one of 
\begin{align} 
& |k_4| > 4 |k_3| >4^3 |k_2| \ge 4^3 |k_1| , \label{rel3} \\
& |k_4|^{4/5} > 4^3 |k_3| \ge 4^3 |k_2| \ge 4^3 |k_1| \label{rel4}
\end{align}
or
\begin{align} \label{rel30}
|k_4| > 4^3 |k_3| \ge 4^3 |k_2| \ge 4^3 |k_1| \hspace{0.3cm} \text{and} \hspace{0.3cm}  |k_{1,2,3}| > 16|k_1|
\end{align}
holds, then it follows that 
\begin{align}
& |\Phi_f^{(4)} | > |k_{1,2,3}| |k_4|^4, \label{2eq5}\\
& \Big| \frac{1}{\Phi_f^{(4)}}- \frac{1}{ \Phi_{g}^{(4)}} \Big| \lesssim \frac{ |\be| |E_0(f) -E_0(g) |  }{ |k_{1,2,3}| |k_4| } 
\min \left\{ \frac{1}{|\Phi_f^{(4)}|}, \frac{1}{ |\Phi_g^{(4)}|}  \right\}. \label{2eq6}
\end{align}
\end{lem}

\begin{proof}
A simple calculation yields that 
\begin{align} \label{L1}
 &\Phi_f^{(4)} (k_1, k_2, k_3, k_4) = \Phi_{f}^{(2)} (k_{1,2,3}, k_4) + \Phi_{f}^{(3)} (k_1, k_2, k_3) \nonumber \\
& = \Phi_{0}^{(2)} (k_{1,2,3}, k_4) + \Phi_{0}^{(3)} (k_1, k_2, k_3) + \be E_0(f) R(k_1, k_2, k_3, k_4)
\end{align}
where  
\begin{align*}
R(k_1, k_2, k_3, k_4): = 3 i (k_4 k_{1,2,3} k_{1,2,3,4} + k_{1,2} k_{2,3} k_{1,3}  ).
\end{align*}
By \eqref{2eq1} and \eqref{2eq2}, we have
\begin{align*}
&\Phi_0^{(2)} (k_{1,2,3}, k_4) =- \frac{5}{2} i k_4 k_{1,2,3} k_{1,2,3,4} (k_4^2 + k_{1,2,3}^2 + k_{1,2,3,4}^2), \\
& \Phi_0^{(3)} (k_1, k_2, k_3) = - \frac{5}{2} i k_{1,2} k_{2,3} k_{1,3} (k_{1,2}^2 + k_{2,3}^2 + k_{1,3}^2).
\end{align*}
We can easily check
\begin{align*}
&|\Phi_0^{(2)} (k_{1,2,3}, k_4  )| > \frac{3}{2} |k_{1,2,3}| |k_4|^4, \hspace{0.3cm}
|\be E_0(f) \, R(k_1, k_2, k_3, k_4)| < \frac{1}{4} |k_{1,2,3}| |k_4|^4
\end{align*}
when \eqref{rel3}, \eqref{rel4} or \eqref{rel30} holds.
When \eqref{rel3} holds, we have
\[
|\Phi_0^{(3)} (k_1, k_2, k_3) | < |k_{1,2}| |k_4|^4 < \frac{1}{4} |k_{1,2,3}| |k_4|^4.
\]
When \eqref{rel4} holds, we have
\[
|\Phi_0^{(3)} (k_1, k_2, k_3) | < 4^4 |k_3|^5 < 4^{-11} |k_4|^4 \le 4^{-11} | k_{1,2,3} | |k_4|^4.
\]
When \eqref{rel30} holds, we have
\[
|\Phi_0^{(3)} (k_1, k_2, k_3) | < 120 |k_{2,3}| |k_3|^4 < 4^4 |k_{1,2,3}| |k_3|^4 \le 4^{-11} | k_{1,2,3} | |k_4|^4.
\]
Therefore, we obtain \eqref{2eq5} by \eqref{L1}.
Moreover, we obtain \eqref{2eq6}
by
\[
| \Phi_f^{(4)}-\Phi_g^{(4)} | = | \beta| |E_0(f) - E_0(g)| | R(k_1,k_2,k_3,k_4) | \lec | \beta| |E_0(f) - E_0(g)| |k_4|^3,
\]
and
\[
\Big| \frac{1}{\Phi_f^{(4)}}- \frac{1}{ \Phi_{g}^{(4)}} \Big| =\frac{| \Phi_f^{(4)}-\Phi_g^{(4)} |}{|\Phi_f^{(4)}| |\Phi_{g}^{(4)}|} \lec\frac{| \Phi_f^{(4)}-\Phi_g^{(4)} |}{|k_{1,2,3}||k_4|^4} \min\left\{ \frac{1}{|\Phi_f^{(4)}|}, \frac{1}{|\Phi_{g}^{(4)}|}\right\}.
\]
\end{proof}

\begin{lem} \label{Le4}
Assume that $f,g \in  L^2 (\T)$ and $|k_4|> 8 \max \{  1, |\be E_0(f)|, |\be E_0(g)| \}$. If 
\begin{align} \label{L12}
4|k_3| \ge |k_4| \ge |k_3|, \hspace{0.5cm} |k_{3,4}| > 4^2 |k_2| \ge 4^2 |k_1|, 
\end{align}
then, it follows that 
\begin{align}
& |\Phi_f^{(4)}| \gtrsim |k_{1,2,3,4}| |k_4|^4 \label{C43}, \\
& \left| \frac{1}{\Phi_f^{(4)}}- \frac{1}{\Phi_g^{(4)}} \right| \lesssim \frac{|\be| |E_0(f) -E_0(g)| }{|k_4|^2} 
\min \left\{ \frac{1}{|\Phi_f^{(4)}|}, \frac{1}{|\Phi_g^{(4)}|} \right\}. \label{C42} 
\end{align}
\end{lem}

\begin{proof}
A direct computation yields that 
\begin{align} \label{L13}
&\Phi_f^{(4)}(k_1, k_2, k_3, k_4) = \Phi_f^{(3)} (k_{1,2} , k_3, k_4)+ \Phi_f^{(2)} (k_1, k_2)  \nonumber \\
& = \Phi_0^{(3)} (k_{1,2}, k_3, k_4)+ \Phi_0^{(2)} (k_1, k_2) + \be E_0(f) R(k_1, k_2, k_3, k_4).
\end{align}
where
\begin{align*}
R(k_1, k_2, k_3, k_4):= 3i (k_{1,2,3} k_{1,2,4} k_{3,4}+k_1 k_2 k_{1,2}). 
\end{align*}
By (\ref{2eq1}) and (\ref{2eq2}), we have
\begin{align*}
& \Phi_{0}^{(3)} (k_{1,2}, k_3, k_4 )= -\frac{5}{2} i k_{1,2,3} k_{1,2,4} k_{3,4} (k_{1,2,3}^2 + k_{1,2,4}^2 +k_{3,4}^2), \\
& \Phi_0^{(2)} (k_1, k_2)=- \frac{5}{2} i k_1 k_2 k_{1,2} (k_1^2 + k_2^2 +k_{1,2}^2). 
\end{align*}
By \eqref{L12} and $|k_4| > 8 \max \{ 1, |\be E_0(f)| \}$, it follows that 
\begin{align*}
& |\Phi_0^{(2)} (k_1, k_2) | \le 30 |k_2|^5 < 4^{-4} |k_{3,4}| |k_3| |k_4|^3, \\
& |\be E_0(f)  R(k_1, k_2, k_3, k_4)| \le 8 |\be E_0(f)| |k_{3,4}| |k_3| |k_4| \le \frac{1}{8} |k_{3,4}| |k_3| |k_4|^3, \\
& |\Phi_0^{(3)} (k_{1,2}, k_3, k_4)| > \frac{3}{8} |k_{3,4}| |k_3| |k_4|^3.
\end{align*}
Thus, by \eqref{L12} and \eqref{L13}, we obtain 
\begin{align*}
|\Phi_f^{(4)}(k_1,k_2,k_3,k_4)| \gec |k_{3,4}| |k_3| |k_4|^3 \gec  |k_{1,2,3,4}| |k_4|^4. 
\end{align*} 
Since
\begin{align*} 
|\Phi_f^{(4)}- \Phi_g^{(4)}| & = |\be| |E_0(f)-E_0(g)| |R(k_1, k_2, k_3, k_4)| \\ 
&  \lec |\be| |E_0 (f) - E_0(g)| |k_{1,2,3,4}| |k_4|^2,
\end{align*}
we also have (\ref{C42}) by \eqref{C43}. 
\end{proof}

\begin{defn}
For $f \in L^2(\T)$ and an $N$-multiplier $m^{(N)} (k_1, \dots, k_N)$, we define $N$-linear functional $\La_f^{(N)}$ by 
\begin{align*}
\La_f^{(N)} (m^{(N)}, \ha{v}_1, \dots, \ha{v}_N) (t,k)
= \sum_{k=k_{1,\dots, N}} e^{-t \Phi_{f}^{(N)} } m^{(N)} (k_1, \dots, k_N) \prod_{l=1}^N \ha{v}_l (k_l) 
\end{align*}
where $\ha{v}_1, \dots, \ha{v}_N$ are functions on $\Z$. 
$\La_f^{(N)} (m^{(N)}, \ha{v}, \dots, \ha{v} )$ may simply be written $\La_f^{(N)} (m^{(N)}, \ha{v}) $. 
\end{defn}

\begin{defn}
We say an $N$-multiplier $m^{(N)}$ is symmetric if 
\begin{equation*}
m^{(N)}(k_1, \dots, k_N)= m^{(N)} (k_{\sigma(1)}, \dots, k_{ \sigma(N) })
\end{equation*}
for all $\sigma \in S_N$, 
the group of all permutations on $N$ objects. 
The symmetrization of an $N$-multiplier $m^{(N)}$ is the multiplier  
\begin{align*}
[m^{(N)} ]_{sym}^{(N)} (k_1, \dots, k_N):= \frac{1}{N!} \sum_{\sigma \in S_N} m^{(N)} (k_{\sigma(1)}, \dots, k_{ \sigma(N) }).
\end{align*}
We also use $\ti{m}^{(N)}:=[m^{(N)} ]_{sym}^{(N)}$ for short.
\end{defn}
\begin{defn}
We say an $N$-multiplier $m^{(N)}$ is symmetric with $(k_i,k_j)$ if
\begin{equation*}
m^{(N)}(k_1, \ldots,k_i,\ldots,k_j,\ldots,  k_N)= m^{(N)} (k_1, \ldots,k_j,\ldots,k_i,\ldots,  k_N)
\end{equation*}
holds for any $(k_1,\ldots,k_N)\in \Z^N$.
\end{defn}
\begin{defn}
We define $(N+j)$-extension operators of an $N$-multiplier $m^{(N)}$ with $j \in \N$ by 
\begin{align*}
[ m^{(N)}  ]_{ext1}^{(N+j)} (k_1, \dots, k_{N+j})& = m^{(N)} (k_1, \dots, k_{N-1}, k_{N, \dots, N+j}), \\
[ m^{(N)}  ]_{ext2}^{(N+j)} (k_1, \dots, k_{N+j})& = m^{(N)} (k_{j+1}, \dots, k_{j+N}). 
\end{align*}
\end{defn}
\begin{rem}\label{rem_sym}
For any multipliers $m_1, m_2$, it follows that
$$[m_1m_2]^{(N)}_{ext1}=[m_1]^{(N)}_{ext1}[m_2]^{(N)}_{ext1}, \ \ \ \ 
[m_1m_2]^{(N)}_{ext2}=[m_1]^{(N)}_{ext2}[m_2]^{(N)}_{ext2}.$$
\end{rem}

For $N \in \{ 2,3,4\}$ and $L>0$, we define some multipliers to restrict summation regions in the Fourier space as follows:
\begin{align*}
& \chi_{\le L}^{(N)}:=
\begin{cases}
\dis 1, \hspace{0.2cm} \text{when} \hspace{0.2cm} \max_{1 \le j \le N}  \{  |k_j| \} \le L \\
0, \hspace{0.2cm} \text{otherwise}
\end{cases}, 
\hspace{0.3cm} 
\chi_{> L}^{(N)}:=
\begin{cases}
\dis 1, \hspace{0.2cm} \text{when} \hspace{0.2cm} \max_{1 \le j \le N}  \{  |k_j| \} > L \\
0, \hspace{0.2cm} \text{otherwise}
\end{cases},
\end{align*}
which are used to have the smallness of $F_{\varphi,L}(\ha{v})$ of \eqref{NF21} by taking $L$ large,
\begin{align*}
& \chi_{NR1}^{(2)}:=
\begin{cases}
\dis 1, \hspace{0.2cm} \text{when} \hspace{0.2cm} k_1 k_2 k_{1,2} \neq 0 \\
0, \hspace{0.2cm} \text{otherwise}
\end{cases},
\hspace{0.5cm}
\chi_{NR1}^{(3)}:=
\begin{cases}
\dis 1, \hspace{0.2cm} \text{when} \hspace{0.2cm} k_{1,2} k_{2,3} k_{1,3} \neq 0 \\
0, \hspace{0.2cm} \text{otherwise}
\end{cases}, 
\end{align*}
which are used to extract a part of non-resonance part (see Proposition \ref{prop_req1} and its proof),
\begin{align*}
& \chi_{H1}^{(N)}:=
\begin{cases}
\dis 1, \hspace{0.2cm} \text{when} \hspace{0.2cm} 4^{N-1} \max_{1 \le l \le N-1} \{ |k_l|   \} < |k_N| \\
0, \hspace{0.2cm} \text{otherwise}
\end{cases}, \\
& \chi_{H2}^{(2)}:=
\begin{cases}
\dis 1, \hspace{0.2cm} \text{when} \hspace{0.2cm} |k_1|/ 4 \le |k_2| \le 4 |k_1| \\
0, \hspace{0.2cm} \text{otherwise}
\end{cases}, \\
& \chi_{R1}^{(N)}:=
\begin{cases}
\dis 1, \hspace{0.2cm} \text{when} \hspace{0.2cm} k_{1,2,\dots, N-1}= 0 \\
0, \hspace{0.2cm} \text{otherwise}
\end{cases},
\hspace{0.5cm}
\chi_{R2}^{(3)}:=
\begin{cases}
\dis 1, \hspace{0.2cm} \text{when} \hspace{0.2cm} k_1 k_2= 0 \\
0, \hspace{0.2cm} \text{otherwise}
\end{cases}, \\
& \chi_{R3}^{(3)}:=
\begin{cases}
\dis 1, \hspace{0.2cm} \text{when} \hspace{0.2cm} k_1=-k_2=k_3 \\
0, \hspace{0.2cm} \text{otherwise}
\end{cases},
\hspace{0.3cm}
\chi_{R4}^{(N)}:=
\begin{cases}
\dis 1, \hspace{0.2cm} \text{when} \hspace{0.2cm} k_1=\dots=k_{N}= 0 \\
0, \hspace{0.2cm} \text{otherwise}
\end{cases}, \\
& \chi_{R5}^{(4)}:=
\begin{cases}
1, \hspace{0.2cm} \text{when} \hspace{0.2cm}  |k_4|^{4/5} \le 64  \max \{|k_1|, |k_2|, |k_3| \},  \hspace{0.3cm} |k_{1,2,3}| \le 16 \min\{ |k_1|, |k_2|, |k_3| \} \\
\hspace{2cm}   \max \{ |k_1|, |k_2|, |k_3| \} \le 16 \, \text{med} \{|k_1|,|k_2|, |k_3| \} \\
0, \hspace{0.2cm} \text{otherwise}
\end{cases} 
\end{align*}
and
\begin{align*} 
& \chi_{A1}^{(3)}:=
\begin{cases}
1, \hspace{0.2cm} \text{when} \hspace{0.2cm} 4 |k_1| < |k_2| \\
0, \hspace{0.2cm} \text{otherwise}
\end{cases}, \hspace{0.2cm}
 \chi_{A2}^{(3)}:=
\begin{cases}
1, \hspace{0.2cm} \text{when} \hspace{0.2cm} 4 \max \{ |k_2|, |k_3| \} < |k_1| \\
0, \hspace{0.2cm} \text{otherwise}
\end{cases}, \\
& \chi_{A3}^{(3)}:=
\begin{cases}
1, \hspace{0.2cm} \text{when} \hspace{0.2cm}  16 |k_1| < |k_{2,3}|  \\
0, \hspace{0.2cm} \text{otherwise}
\end{cases}, \hspace{0.2cm}
 \chi_{A4}^{(4)}:=
\begin{cases}
1, \hspace{0.2cm} \text{when} \hspace{0.2cm} 16 \max\{ |k_1|, |k_2| \} < |k_3|  \\
0, \hspace{0.2cm} \text{otherwise}
\end{cases}. 
\end{align*}
Moreover, we define
\begin{align*}
& \chi_{NR2}^{(3)} := [ \chi_{NR1}^{(2)} ]_{ext1}^{(3)} [\chi_{NR1}^{(2)}]_{ext2}^{(3)}, \hspace{0.3cm}
\chi_{NR3}^{(3)}:= \chi_{NR1}^{(3)} \chi_{NR2}^{(3)},
\hspace{0.3cm} \chi_{NR1}^{(4)}:= [ \chi_{NR3}^{(3)} ]_{ext1}^{(4)} [ \chi_{NR1}^{(2)} ]_{ext2}^{(4)}, \\ 
& \chi_{NR(1,1)}^{(3)}:= [ \chi_{H1}^{(2)}]_{ext1}^{(3)} [ \chi_{H1}^{(2)} ]_{ext2}^{(3)}, 
\hspace{0.3cm}
\chi_{NR(1,2)}^{(3)}:= \chi_{H1}^{(2)} (k_{2,3}, k_1) \chi_{H1}^{(2)} (k_2, k_3), \\
& \chi_{NR(2,1)}^{(3)}:= [ \chi_{H1}^{(2)}]_{ext1}^{(3)} [ \chi_{H2}^{(2)} ]_{ext2}^{(3)}, 
 \hspace{0.3cm}
\chi_{NR(2,2)}^{(3)}:= \chi_{H1}^{(2)} (k_{2,3}, k_1) \chi_{H2}^{(2)} (k_2, k_3), \\
& \chi_{NR(1,1)}^{(4)}:= [ \chi_{H1}^{(3)}]_{ext1}^{(4)} [ \chi_{H1}^{(2)} ]_{ext2}^{(4)}, 
\hspace{0.3cm}
\chi_{NR(1,2)}^{(4)}:= \chi_{H1}^{(3)} (k_{3,4}, k_2, k_1) \chi_{H1}^{(2)} (k_3, k_4), \\
& \chi_{NR(2,1)}^{(4)}:=[ \chi_{H1}^{(3)}]_{ext1}^{(4)} [ \chi_{H2}^{(2)} ]_{ext2}^{(4)}, 
\hspace{0.3cm}
\chi_{NR(2,2)}^{(4)}:= \chi_{H1}^{(3)} (k_{3,4}, k_2, k_1) \chi_{H2}^{(2)} (k_3, k_4).
\end{align*}

\begin{rem} \label{rem_mNR}
Since $(k_1, k_2, k_3) \in  \supp \chi_{NR3}^{(3)}$ means that
\begin{align*}
k_1k_2 k_3 k_{1,2,3} \neq 0, \hspace{0.5cm} k_{1,2} k_{2,3} k_{1,3} \neq 0,
\end{align*}
the multiplier $\chi_{NR3}^{(3)}$ is symmetric.
\end{rem}

\begin{lem} \label{Le6} 
Let a $2$-multiplier $m_2^{(2)}$ and a $3$-multiplier $m_1^{(3)}$ be symmetric and 
a $3$-multiplier $m_3^{(3)}$ is symmetric with $(k_1, k_2)$. 
For any $2$-multiplier $m_1^{(2)}$, it follows that 
\begin{align}
&\big[ [ m_1^{(2)} [2 \chi_{H1}^{(2)}]_{sym}^{(2)} ]_{ext1}^{(3)} [m_2^{(2)} ]_{ext2}^{(3)}  \big]_{sym}^{(3)}  \nonumber \\
 & = 2 \big[ [ m_1^{(2)} ]_{ext1}^{(3)} [m_2^{(2)}]_{ext2}^{(3)} \chi_{NR(1,1)}^{(3)} \big]_{sym}^{(3)} 
+  2 \big[ [ m_1^{(2)} ]_{ext1}^{(3)} [m_2^{(2)}]_{ext2}^{(3)}  \chi_{NR(1,2)}^{(3)} \big]_{sym}^{(3)} \nonumber \\
& \hspace{0.3cm} +  \big[ [ m_1^{(2)} ]_{ext1}^{(3)} [m_2^{(2)}]_{ext2}^{(3)} \chi_{NR(2,1)}^{(3)} \big]_{sym}^{(3)} 
+   \big[ [ m_1^{(2)} ]_{ext1}^{(3)} [m_2^{(2)}]_{ext2}^{(3)} \chi_{NR(2,2)}^{(3)} \big]_{sym}^{(3)} 
\label{le211} \\
&\big[ [ m_1^{(3)}  [3  m_3^{(3)} \chi_{H1}^{(3)}   ]_{sym}^{(3)} ]_{ext1}^{(4)} [m_2^{(2)}  ]_{ext2}^{(4)}  \big]_{sym}^{(4)}
 \nonumber \\
& =2 \big[ [ m_1^{(3)} m_3^{(3)} ]_{ext1}^{(4)} [m_2^{(2)}]_{ext2}^{(4)} \chi_{NR(1,1)}^{(4)}  \big]_{sym}^{(4)}  \nonumber \\
& + 4 \big[ [ m_1^{(3)}]_{ext1}^{(4)} [m_2^{(2)}]_{ext2}^{(4)}  [ [ 3 m_{3}^{(3)} \chi_{H1}^{(3)}]_{sym}^{(3)} ]_{ext1}^{(4)} 
 \chi_{NR(1,2)}^{(4)}  \big]_{sym}^{(4)} \nonumber \\
& + \big[ [ m_1^{(3)} m_3^{(3)} ]_{ext1}^{(4)} [m_2^{(2)}]_{ext2}^{(4)} \chi_{NR(2,1)}^{(4)}  \big]_{sym}^{(4)}  \nonumber \\
 & + 2 \big[ [ m_1^{(3)}]_{ext1}^{(4)} [m_2^{(2)}]_{ext2}^{(4)}  [ [ 3 m_{3}^{(3)} \chi_{H1}^{(3)}]_{sym}^{(3)} ]_{ext1}^{(4)} 
 \chi_{NR(2,2)}^{(4)}  \big]_{sym}^{(4)}
\label{le212} 
\end{align}
\end{lem}

\begin{proof}
First, we prove (\ref{le211}). By $[2\chi_{H1}^{(2)}]_{sym}^{(2)}+\chi_{H2}^{(2)} =1 $, the left hand side of (\ref{le211}) is equal to 
\begin{align*}
\big[ [ m_1^{(2)} [2 \chi_{H1}^{(2)}]_{sym}^{(2)} ]_{ext1}^{(3)} [m_2^{(2)} [2\chi_{H1}^{(2)} ]_{sym}^{(2)} ]_{ext2}^{(3)}  \big]_{sym}^{(3)} 
+ \big[ [ m_1^{(2)} [2 \chi_{H1}^{(2)}]_{sym}^{(2)} ]_{ext1}^{(3)} [m_2^{(2)} \chi_{H2}^{(2)} ]_{ext2}^{(3)}  \big]_{sym}^{(3)}. 
\end{align*}
Thus, it suffices to show 
\begin{align}
& \big[ [ m_1^{(2)} [2 \chi_{H1}^{(2)}]_{sym}^{(2)} ]_{ext1}^{(3)} [m_2^{(2)} [2\chi_{H1}^{(2)} ]_{sym}^{(2)} ]_{ext2}^{(3)}  \big]_{sym}^{(3)} \nonumber \\
& =  2 \big[ [ m_1^{(2)} ]_{ext1}^{(3)} [m_2^{(2)}]_{ext2}^{(3)} \chi_{NR(1,1)}^{(3)} \big]_{sym}^{(3)} 
 +  2 \big[ [ m_1^{(2)} ]_{ext1}^{(3)} [m_2^{(2)}]_{ext2}^{(3)}  \chi_{NR(1,2)}^{(3)} \big]_{sym}^{(3)}, \label{le221} \\
& \big[ [ m_1^{(2)} [2 \chi_{H1}^{(2)}]_{sym}^{(2)} ]_{ext1}^{(3)} [m_2^{(2)} \chi_{H2}^{(2)} ]_{ext2}^{(3)}   \big]_{sym}^{(3)} \nonumber \\
& = \big[ [ m_1^{(2)} ]_{ext1}^{(3)} [m_2^{(2)}]_{ext2}^{(3)} \chi_{NR(2,1)}^{(3)} \big]_{sym}^{(3)} 
+  \big[ [ m_1^{(2)} ]_{ext1}^{(3)} [m_2^{(2)}]_{ext2}^{(3)}  \chi_{NR(2,2)}^{(3)} \big]_{sym}^{(3)}. \label{le222}
\end{align}
Put $M^{(3)}:= [ m_1^{(2)} [2 \chi_{H1}^{(2)}]_{sym}^{(2)} ]_{ext1}^{(3)} [m_2^{(2)} [2\chi_{H1}^{(2)} ]_{sym}^{(2)} ]_{ext2}^{(3)}  $. 
Then, by Remark~\ref{rem_sym}, 
\begin{align*}
& M^{(3)}  = [ m_1^{(2)}]_{ext1}^{(3)}  [m_2^{(2)} ]_{ext2}^{(3)} 
[ [ 2 \chi_{H1}^{(2)} ]_{sym}^{(2)} ]_{ext1}^{(3)}  [ [ 2\chi_{H1}^{(2)} ]_{sym}^{(2)} ]_{ext2}^{(3)} \\
& = [ m_1^{(2)}]_{ext1}^{(3)}  [m_2^{(2)} ]_{ext2}^{(3)} 
\{ \chi_{H1}^{(2)} (k_1, k_{2,3}) + \chi_{H1}^{(2)} (k_{2,3}, k_1 ) \}  \{ \chi_{H1}^{(2)} (k_2, k_3) + \chi_{H1}^{(2)} (k_3, k_2) \}.
\end{align*}
Since $m_2^{(2)}$ is symmetric, 
$ [ m_1^{(2)}]_{ext1}^{(3)}  [m_2^{(2)} ]_{ext2}^{(3)} (k_1, k_2, k_3)$ is symmetric with $(k_2, k_3)$.
Therefore, it follows that 
\begin{align*}
[M^{(3)}]_{sym}^{(3)}
& = \big[ [ m_1^{(2)}]_{ext1}^{(3)}  [m_2^{(2)} ]_{ext2}^{(3)} 
\{ \chi_{H1}^{(2)} (k_1, k_{2,3}) + \chi_{H1}^{(2)} (k_{2,3}, k_1 ) \} 2 \chi_{H1}^{(2)}(k_2, k_3) \big]_{sym}^{(3)} \\
& = 2 \big[ [ m_1^{(2)}]_{ext1}^{(3)}  [m_2^{(2)} ]_{ext2}^{(3)} \chi_{NR(1,1)}^{(3)} \big]_{sym}^{(3)}
+ 2 \big[ [ m_1^{(2)}]_{ext1}^{(3)}  [m_2^{(2)} ]_{ext2}^{(3)} \chi_{NR(1,2)}^{(3)} \big]_{sym}^{(3)},
\end{align*}
which implies (\ref{le221}). Similarly, we obtain (\ref{le222}). 

Next, we prove (\ref{le212}). By $[2\chi_{H1}^{(2)}]_{sym}^{(2)} + \chi_{H2}^{(2)}=1$, the left hand side of (\ref{le212}) is equal to 
\begin{align*}
& \big[ [ m_1^{(3)}  [3  m_3^{(3)} \chi_{H1}^{(3)}   ]_{sym}^{(3)} ]_{ext1}^{(4)} [m_2^{(2)} [2\chi_{H1}^{(2)} ]_{sym}^{(2)} ]_{ext2}^{(4)}  \big]_{sym}^{(4)} \\
& + \big[ [ m_1^{(3)}  [3  m_3^{(3)} \chi_{H1}^{(3)}   ]_{sym}^{(3)} ]_{ext1}^{(4)} [m_2^{(2)} \chi_{H2}^{(2)}  ]_{ext2}^{(4)}  \big]_{sym}^{(4)}.
\end{align*}
Thus, it suffices to show 
\begin{align}
& \big[ [ m_1^{(3)}  [3  m_3^{(3)} \chi_{H1}^{(3)}   ]_{sym}^{(3)} ]_{ext1}^{(4)} [m_2^{(2)} [2\chi_{H1}^{(2)} ]_{sym}^{(2)} ]_{ext2}^{(4)}  \big]_{sym}^{(4)} 
\nonumber \\
&=2 \big[ [ m_1^{(3)}  m_3^{(3)} ]_{ext1}^{(4)} [m_2^{(2)} ]_{ext2}^{(4)} \chi_{NR(1,1)}^{(4)}  \big]_{sym}^{(4)} \nonumber \\
& + 4 \big[ [ m_1^{(3)} ]_{ext1}^{(4)} [m_{2}^{(2)}]_{ext2}^{(4)} 
[ [3  m_3^{(3)} \chi_{H1}^{(3)}   ]_{sym}^{(3)} ]_{ext1}^{(4)} \chi_{NR(1,2)}^{(4)} \big]_{sym}^{(4)}, \label{le231} \\
& \big[ [ m_1^{(3)}  [3  m_3^{(3)} \chi_{H1}^{(3)}   ]_{sym}^{(3)} ]_{ext1}^{(4)} [m_2^{(2)} \chi_{H2}^{(2)} ]_{ext2}^{(4)}  \big]_{sym}^{(4)}  \nonumber \\
&= \big[ [ m_1^{(3)}  m_3^{(3)} ]_{ext1}^{(4)} [m_2^{(2)} ]_{ext2}^{(4)} \chi_{NR(2,1)}^{(4)}  \big]_{sym}^{(4)} \nonumber \\
& + 2 \big[ [ m_1^{(3)} ]_{ext1}^{(4)} [m_{2}^{(2)}]_{ext2}^{(4)} 
[ [3  m_3^{(3)} \chi_{H1}^{(3)}   ]_{sym}^{(3)} ]_{ext1}^{(4)} \chi_{NR(2,2)}^{(4)} \big]_{sym}^{(4)}. \label{le232} 
\end{align}
Firstly, we show (\ref{le231}). 
Put $M_1^{(4)}:=[ m_1^{(3)}  [3  m_3^{(3)} \chi_{H1}^{(3)}   ]_{sym}^{(3)} ]_{ext1}^{(4)} [m_2^{(2)} [2 \chi_{H1}^{(2)} ]_{sym}^{(2)} ]_{ext2}^{(4)} $. 
Since $\supp \, [3m_3^{(3)} \chi_{H1}^{(3)}]_{sym}^{(3)} \subset \supp \, [3 \chi_{H1}^{(3)}]_{sym}^{(3)}$ and 
$[ 3 \chi_{H1}^{(3)} ]_{sym}^{(3)}$ is a characteristic function, by Remark~\ref{rem_sym}, 
\begin{align} \label{le233}
[ [3m_3^{(3)} \chi_{H1}^{(3)}]_{sym}^{(3)} ]_{ext1}^{(4)} 
& =[  [3m_3^{(3)} \chi_{H1}^{(3)}]_{sym}^{(3)}  [3 \chi_{H1}^{(3)}]_{sym}^{(3)} ]_{ext1}^{(4)} \nonumber \\
& =[ [3m_3^{(3)} \chi_{H1}^{(3)}]_{sym}^{(3)} ]_{ext1}^{(4)} [ [3\chi_{H1}^{(3)}]_{sym}^{(3)} ]_{ext1}^{(4)}. 
\end{align} 
By (\ref{le233}) and Remark~\ref{rem_sym}, 
\begin{align*}
M_1^{(4)}
 & =[m_1^{(3)} ]_{ext1}^{(4)} [m_{2}^{(2)} ]_{ext2}^{(4)}  [ [3 m_{3}^{(3)} \chi_{H1}^{(3)} ]_{sym}^{(3)} ]_{ext1}^{(4)} 
 [ [2\chi_{H1}^{(2)} ]_{sym}^{(2)} ]_{ext2}^{(4)} \\
& = [m_1^{(3)} ]_{ext1}^{(4)} [m_{2}^{(2)} ]_{ext2}^{(4)}  [ [3 m_{3}^{(3)} \chi_{H1}^{(3)} ]_{sym}^{(3)} ]_{ext1}^{(4)} 
[ [3 \chi_{H1}^{(3)}]_{sym}^{(3)} ]_{ext1}^{(4)} [ [2\chi_{H1}^{(2)} ]_{sym}^{(2)} ]_{ext2}^{(4)} \\
& = [m_1^{(3)} ]_{ext1}^{(4)} [m_{2}^{(2)} ]_{ext2}^{(4)}  [ [3 m_{3}^{(3)} \chi_{H1}^{(3)} ]_{sym}^{(3)} ]_{ext1}^{(4)} 
\{ \chi_{H1}^{(2)}(k_3, k_4) + \chi_{H1}^{(2)} (k_4, k_3)   \} \\
& \times \{ \chi_{H1}^{(3)} (k_1, k_2, k_{3,4})+\chi_{H1}^{(3)} (k_{3,4}, k_2 , k_1) + \chi_{H1}^{(3)} (k_{3,4}, k_1 , k_2) \}.
\end{align*}
Since $m_1^{(3)}, m_2^{(2)}$ are symmetric, 
$ [m_1^{(3)} ]_{ext1}^{(4)} [m_{2}^{(2)} ]_{ext2}^{(4)}  [ [3 m_{3}^{(3)} \chi_{H1}^{(3)} ]_{sym}^{(3)} ]_{ext1}^{(4)}(k_1, k_2, k_3, k_4) $ 
is symmetric with $(k_1, k_2)$ and $(k_3, k_4)$. Therefore, we have  
\begin{align} \label{le234}
[M_1^{(4)}]_{sym}^{(4)} & = 
\big[ [m_1^{(3)} ]_{ext1}^{(4)} [m_{2}^{(2)} ]_{ext2}^{(4)}  [ [3 m_{3}^{(3)} \chi_{H1}^{(3)} ]_{sym}^{(3)} ]_{ext1}^{(4)} \nonumber \\
& \hspace{0.3cm}
 \times \{ \chi_{H1}^{(3)} (k_1, k_2, k_{3,4})+2 \chi_{H1}^{(3)} (k_{3,4}, k_2 , k_1) \}  2 \chi_{H1}^{(2)} (k_3, k_4) \big]_{sym}^{(4)} \nonumber \\
&= 2\big[ [m_1^{(3)} ]_{ext1}^{(4)} [m_{2}^{(2)} ]_{ext2}^{(4)}  [ [3 m_{3}^{(3)} \chi_{H1}^{(3)} ]_{sym}^{(3)} ]_{ext1}^{(4)} 
\chi_{NR(1,1)}^{(4)} \big]_{sym}^{(4)} \nonumber \\
& + 4 \big[ [m_1^{(3)} ]_{ext1}^{(4)} [m_{2}^{(2)} ]_{ext2}^{(4)}  [ [3 m_{3}^{(3)} \chi_{H1}^{(3)} ]_{sym}^{(3)} ]_{ext1}^{(4)}  \chi_{NR(1,2)}^{(4)} ]_{sym}^{(4)}.  
\end{align}
By $\chi_{NR(1,1)}^{(4)} =[ \chi_{H1}^{(3)} ]_{ext1}^{(4)} \chi_{NR(1,1)}^{(4)} $ and Remark~\ref{rem_sym}, 
\begin{align} \label{le235}
& [ [ 3 m_{3}^{(3)} \chi_{H1}^{(3)}  ]_{sym}^{(3)} ]_{ext1}^{(4)} \chi_{NR(1,1)}^{(4)}
= [ [ 3 m_{3}^{(3)} \chi_{H1}^{(3)}  ]_{sym}^{(3)} ]_{ext1}^{(4)} [\chi_{H1}^{(3)}]_{ext1}^{(4)}  \chi_{NR(1,1)}^{(4)} \nonumber \\
&= [ [ 3 m_{3}^{(3)} \chi_{H1}^{(3)}  ]_{sym}^{(3)} \chi_{H1}^{(3)} ]_{ext1}^{(4)} \chi_{NR(1,1)}^{(4)}
=[  m_{3}^{(3)} \chi_{H1}^{(3)} ]_{ext1}^{(4)} \chi_{NR(1,1)}^{(4)}  \notag \\
& =[m_3^{(3)}]_{ext1}^{(4)} [\chi_{H1}^{(3)}]_{ext1}^{(4)} \chi_{NR(1,1)}^{(4)}
=[m_3^{(3)}]_{ext1}^{(4)} \chi_{NR(1,1)}^{(4)}.
\end{align}
Here we used that $m_{3}^{(3)} (k_1, k_2, k_3) $ is symmetric with $(k_1, k_2)$ in the third equality. 
By (\ref{le235}) and Remark~\ref{rem_sym}, 
\begin{align} \label{le236}
& \big[ [m_1^{(3)} ]_{ext1}^{(4)} [m_{2}^{(2)} ]_{ext2}^{(4)}  [ [3 m_{3}^{(3)} \chi_{H1}^{(3)} ]_{sym}^{(3)} ]_{ext1}^{(4)} 
\chi_{NR(1,1)}^{(4)} \big]_{sym}^{(4)} \nonumber \\
& = \big[ [m_1^{(3)} m_3^{(3)}]_{ext1}^{(4)} [m_2^{(2)}]_{ext2}^{(4)} \chi_{NR(1,1)}^{(4)}   \big]_{sym}^{(4)}.
\end{align}
Substituting (\ref{le236}) into (\ref{le234}), we obtain (\ref{le231}). 

Secondly, we show (\ref{le232}). 
Put $M_2^{(4)}:=[ m_1^{(3)} [3 m_{3}^{(3)} \chi_{H1}^{(3)} ]_{sym}^{(3)} ]_{ext1}^{(4)} [ m_2^{(2)} \chi_{H2}^{(2)} ]_{ext2}^{(4)}$. 
Then, by (\ref{le233}) and Remark~\ref{rem_sym}, 
\begin{align*}
M_2^{(4)}=& [ m_1^{(3)}]_{ext1}^{(4)} [ m_{2}^{(2)}]_{ext2}^{(4)} [ [3 m_{3}^{(3)} \chi_{H1}^{(3)} ]_{sym}^{(3)} ]_{ext1}^{(4)} 
\chi_{H2}^{(2)} (k_3, k_4) \\
& \times \{ \chi_{H1}^{(3)}(k_1, k_2, k_{3,4})+\chi_{H1}^{(3)} (k_{3,4}, k_2, k_1 )+ \chi_{H1}^{(3)} (k_{3,4} ,k_1, k_2 )  \}.
\end{align*}
Since $m_1^{(3)}$ is symmetric, 
$ [ m_1^{(3)}]_{ext1}^{(4)} [ m_{2}^{(2)}]_{ext2}^{(4)} [ [3 m_{3}^{(3)} \chi_{H1}^{(3)} ]_{sym}^{(3)} ]_{ext1}^{(4)}(k_1,k_2, k_3, k_4)$ 
is symmetric with $(k_1, k_2)$.  Thus, we have
\begin{align} \label{le241}
[M_2^{(4)}]_{sym}^{(4)} & = 
\big[ [ m_1^{(3)}]_{ext1}^{(4)} [ m_{2}^{(2)}]_{ext2}^{(4)} [ [3 m_{3}^{(3)} \chi_{H1}^{(3)} ]_{sym}^{(3)} ]_{ext1}^{(4)} \chi_{NR(2,1)}^{(4)} \big]_{sym}^{(4)}
\nonumber \\
& + 2 \big[ [ m_1^{(3)}]_{ext1}^{(4)} [ m_{2}^{(2)}]_{ext2}^{(4)} [ [3 m_{3}^{(3)} \chi_{H1}^{(3)} ]_{sym}^{(3)} ]_{ext1}^{(4)} 
\chi_{NR(2,2)}^{(4)} \big]_{sym}^{(4)}.
\end{align}
In a same manner as (\ref{le235}), by $\chi_{NR(2,1)}^{(4)} =[ \chi_{H1}^{(3)} ]_{ext1}^{(4)} \chi_{NR(2,1)}^{(4)} $ and Remark~\ref{rem_sym}, 
\begin{align*}
[ [ 3 m_3^{(3)} \chi_{H1}^{(3)} ]_{sym}^{(3)} ]_{ext1}^{(4)} \chi_{NR(2,1)}^{(4)}= [ m_3^{(3)}]_{ext1}^{(4)} \chi_{NR(2,1)}^{(4)},
\end{align*}
which leads that 
\begin{align} \label{le242}
&  \big[ [ m_1^{(3)}]_{ext1}^{(4)} [ m_{2}^{(2)}]_{ext2}^{(4)} [ [3 m_{3}^{(3)} \chi_{H1}^{(3)} ]_{sym}^{(3)} ]_{ext1}^{(4)} \chi_{NR(2,1)}^{(4)} \big]_{sym}^{(4)} 
\nonumber \\
& = \big[  [m_1^{(3)}  m_3^{(3)}  ]_{ext1}^{(4)} [ m_2^{(2)}  ]_{ext2}^{(4)} \chi_{NR(2,1)}^{(4)}  \big]_{sym}^{(4)}. 
\end{align}
Substituting (\ref{le242}) into (\ref{le241}), we have (\ref{le232}).  
\end{proof}

Finally, we show variants of Sobolev's inequalities.
\begin{lem} \label{lem_nl2}
Let $s \ge 1$. 
Then, for any $\{ v_l \}_{l=1}^N \subset H^s(\T)$, we have
\begin{align} \label{nl11}
\Big\| \sum_{k=k_{1,\cdots, N}} \langle k_{1, \dots, N} \rangle \langle k_{\max} \rangle^2 \prod_{l=1}^N |\ha{v}_l(k_l)|  \Big\|_{l_{s-3}^2  } \lesssim \prod_{l=1}^N \| v_l \|_{H^s}. 
\end{align}
\end{lem}

\begin{proof}
By the duality argument, we only need to show
\begin{align} \label{nl14}
\sum_{k \in \Z} \sum_{k=k_{1, \dots, N}} M(k_1, \dots, k_N) \prod_{l=1}^N |u_l(k_l)| |z(k)|
\lesssim \prod_{l=1}^N \| u_l \|_{l^2} \| z \|_{l^2}
\end{align}
for any $z\in l^2$ and $\{ u_l \}_{l=1}^N \subset l^2$, where
\begin{align} \label{nl141}
M(k_1, \dots, k_N) = \langle k_{1, \dots, N} \rangle^{s-2} \langle k_{\max} \rangle^2 \prod_{l=1}^N  \langle k_l \rangle^{-s}
\end{align}
(i) For the case of $|k_{1,\dots, N}| \sim k_{\max}$, there exists $i \in \{ 1,\dots, N\}$ such that $|k_i| =k_{\max} \sim |k_{1,\dots, N}|$.
Thus, it follows that
\begin{align*}
M(k_1, \dots, k_N) \lesssim \prod_{ l \in \{ 1, \dots, N \} \setminus \{ i \}  } \langle  k_l \rangle^{-s},
\end{align*}
which means
\begin{align*}
\sup_{k \in \Z} \sum_{k=k_{1,\dots, N}} M(k_1, \dots, k_N)^2 \lesssim 1.
\end{align*}
Therefore, by the Schwarz inequality, we conclude
\begin{align*}
\text{LHS of (\ref{nl14})} & \le 
\bigg\{ \sum_{k \in \Z}  \sum_{k=k_{1, \dots, N}} M (k_1, \dots, k_N)^2 |z (k)|^2  \bigg\}^{1/2} \prod_{l=1}^N \| u_l  \|_{l^2} \\
& \lesssim \| z \|_{l^2} \prod_{l=1}^N \| u_l \|_{l^2}.
\end{align*}
(ii) For the case of $|k_{1,\dots, N}| \ll k_{\max}$, there exist $i, j \in \{ 1, \dots ,N \}$ such that $k_{\max}=|k_{i}| \sim |k_j|$.
Since $s \ge 1$, we have
\begin{align*}
M(k_1, \dots, k_N) 
& \lesssim \langle k_{1, \dots, N} \rangle^{s-2} \langle k_{\max} \rangle^{-2s+2} \prod_{l \in \{ 1, \dots, N \} \setminus \{ i,j \} } \langle k_l \rangle^{-s} \\
& \lesssim \langle k_{1, \dots, N} \rangle^{-s} \prod_{l \in \{ 1, \dots, N  \} \setminus \{ i, j \}  }  \langle k_l  \rangle^{-s}.
\end{align*}
When $(i,j)=(N-1,N)$, by the Schwarz inequality, we obtain
\begin{align*}
\text{LHS of (\ref{nl14})} & \lesssim \sum_{k \in \Z} \langle k \rangle^{-s} |z(k)| \, \sum_{k_1 \in \Z} \langle k_1 \rangle^{-s} |u_1(k_1)| 
\, \cdots \, \sum_{k_{N-2} \in \Z} \langle k_{N-2} \rangle^{-s} |u_{N-2} (k_{N-2})| \\
& \times \sum_{k_{N-1} \in  \Z } |u_{N-1} (k_{N-1}) | |u_N(k-k_{1, \dots, N-1}) | \lesssim \| z \|_{l^2} \prod_{l=1}^N \| u_l \|_{l^2}. 
\end{align*}
Similarly, we have the same result for general $i,j \in \{1,\ldots,N\}$. 
\end{proof}

\begin{lem} \label{lem_nl3} 
For $s \ge 1$ and $i \in \{ 1,\dots , N \}$, we have
\begin{align} \label{nl41}
\Big\| \sum_{k=k_{1,\dots, N}} \langle k_{1, \dots, N}  \rangle^{-1} \langle k_{\max} \rangle^{-2} \prod_{l=1}^N |\ha{v}_l (k_l) | \Big\|_{l_{s}^{2}} 
\lesssim \| v_i \|_{H^{s-3}} \prod_{l  \in \{ 1, \dots , N \} \setminus \{ i  \} } \| v_l \|_{H^s}.
\end{align}
For $s > 1/2$ and $i \in \{1, \dots, N  \}$, we have 
\begin{align} \label{nl42}
\Big\| \sum_{k=k_{1, \dots ,N}}  \langle k_{\max} \rangle^{-1} \prod_{l=1}^N |\ha{v}_l (k_l) | \Big\|_{l_{s}^{2}} 
\lesssim  \| v_i \|_{ H^{s-1} } \prod_{l  \in \{ 1, \dots , N \} \setminus \{ i \} } \| v_l \|_{H^s}.  
\end{align}
\end{lem}
\begin{proof}
We only show (\ref{nl41}) because we  easily check (\ref{nl42}). Since
\begin{align*}
\langle k_{1, \dots , N} \rangle^{-1} \langle k_{\max} \rangle^{-2} 
& \lesssim \langle k_{1, \dots, N} \rangle^{-3} ( \langle k_{1,\dots, N} \rangle \langle k_{\max} \rangle^2 ) \langle k_{\max} \rangle^{-3} \\
& \lesssim  \langle k_{1, \dots, N} \rangle^{-3} ( \langle k_{1,\dots, N} \rangle \langle k_{\max} \rangle^2 ) \langle k_i \rangle^{-3},
\end{align*}
by Lemma~\ref{lem_nl2}, we have 
\begin{align*}
\text{LHS of (\ref{nl41})}
& \lesssim \Big\| \sum_{k=k_{1, \dots, N}} \langle k_{1, \dots, N} \rangle \langle k_{\max}  \rangle^{2} \,
\langle k_i \rangle^{-3} |\ha{v}_i (k_i)| \prod_{l \in \{ 1, \dots , N \} \setminus \{ i \}} |\ha{v}_l(k_l)|   \Big\|_{l_{s-3}^2} \\
& \lesssim \| v_i \|_{H^{s-3}} \prod_{l \in  \{ 1, \dots, N  \} \setminus \{ i \}} \| v_l \|_{H^s}.
\end{align*}
\end{proof}
\begin{lem}\label{lem_go}
Let $f, g \in L^2(\T)$ and an $N$-multiplier $m^{(N)}$ satisfy
\begin{equation}\label{ex0}
\Big\| \sum_{k=k_{1,\ldots,N}} |m^{(N)}(k_1,\ldots,k_N)| \prod_{l=1}^N |\ha{v}_l(t,k_l) | \Big\|_{L^\infty_Tl^2_s} \le C_0 \prod_{l=1}^N\|v_l\|_{L^\infty_TH^{s_l}}
\end{equation}
for any $v_l \in C([-T,T];H^{s_l}(\T))$ with $l=1,2, \ldots,N$.
Then, for any \\
$v_l \in C([-T,T];H^{s_l}(\T))$ with $l=1,2,\ldots,N$, it follows that
\begin{equation}\label{ex1}
\Big\| 
\Lambda_{f}^{(N)} ( m^{(N)} , \ha{v}_1,\ldots,\ha{v}_N ) -\Lambda_{g}^{(N)} ( m^{(N)} , \ha{v}_1,\ldots,\ha{v}_N )\Big\|_{L^\infty_Tl^2_s} \le C_*
\end{equation}
where $C_*=C_*(C_0,v_1,\ldots,v_N,s_1,\ldots,s_N,|E_0(f)-E_0(g)|,|\be|,T) \ge 0$ and $C_*\to 0$ when $|E_0(f)-E_0(g)|\to 0$.
\end{lem}
\begin{proof}
The left-hand side of \eqref{ex1} is bounded by
\begin{equation*}
\begin{split}
&\Big\| \sum_{k=k_{1,\ldots,N}} (e^{-t\Phi_{f}^{(N)}}-e^{-t\Phi_{g}^{(N)}})m^{(N)}(k_1,\ldots,k_N) \Big(\prod_{l=1}^N\ha{v}(t, k_l)-\prod_{l=1}^N\chi_{\le K}(k_l)\ha{v}(t, k_l)\Big)\Big\|_{L^\infty_Tl^2_s}\\
&+
\Big\|\sum_{k=k_{1,\ldots,N}} (e^{-t\Phi_{f}^{(N)}}-e^{-t\Phi_{g}^{(N)}})m^{(N)}(k_1,\ldots,k_N) \prod_{l=1}^N\chi_{\le K}(k_l)\ha{v}(t, k_l)\Big\|_{L^\infty_Tl^2_s}=:I_1+I_2\\
\end{split}
\end{equation*}
where $\chi_{\le K}(k)=1$ for $|k|\le K$, $=0$ for $|k|>K$.
By uniform $l^2_{s_i}$-continuity of $\ha{v}_i$,
for any $\varepsilon>0$, there exists $K_i\in \N$ such that
\[
\|\chi_{> K_i}\ha{v}_i \|_{L^\infty_Tl^2_{s_i}} <\varepsilon\Big(N \prod_{l \in  \{ 1, \dots, N  \} \setminus \{ i \}}\|v_l\|_{L^\infty_TH^{s_l}}\Big)^{-1}.
\]
Put $K=\max\{K_1,\ldots,K_N \}$. Then, by \eqref{ex0}, we have
\EQQ{
I_1 &\lec \sum_{i=1}^N \Big\|\sum_{k=k_{1,\ldots,N}} |m^{(N)}(k_1,\ldots,k_N)| |\chi_{> K}(k_i) \ha{v}_i(t, k_i)|\prod_{l \in  \{ 1, \dots, N  \} \setminus \{ i \}} |\ha{v}_l(t, k_l)|\Big\|_{L^\infty_Tl^2_s}\\
&\lec \sum_{i=1}^N \| \chi_{> K} \ha{v}_i \|_{L^\infty_Tl^2_{s_i}} \prod_{l \in  \{ 1, \dots, N  \} \setminus \{ i \}} \|v_l\|_{L^\infty_TH^{s_l}}\lec \e.
}
Since $|\Phi_{f}^{(N)}-\Phi_{g}^{(N)}|\le |\beta | |E_0(f)-E_0(g)||k_{1,\ldots,N}^3-\sum_{l=1}^N k_l^3|$,
it follows that
\begin{equation*}
\begin{split}
I_2
\le &C(|\beta|,K,T)|E_0(f)-E_0(g)| \Big\|  \sum_{k=k_{1,\ldots,N}} |m^{(N)}(k_1,\ldots,k_N)| \prod_{l=1}^N |\ha{v}(t, k_l)|\Big\|_{L^\infty_Tl^2_s}\\
\le &C(|\beta|,K,T)|E_0(f)-E_0(g)| \prod_{l=1}^N\|v_l\|_{L^\infty_TH^{s_l}}
\end{split}
\end{equation*}
for sufficiently small $|E_0(f)- E_0(g)|$.
Therefore, we obtain \eqref{ex1}.
\end{proof}
\section{the normal form reduction}

Our aim in this section is to remove the derivative losses in the right-hand side of \eqref{5KdV3} by the normal form reduction, that is the differentiation by parts.
The main proposition in this section is as below.
\begin{prop} \label{prop_NF2}
Let $s \ge 1$, $\vp \in L^2(\T)$, $L \gg \max \{ 1, |\be E_0(\vp)| \}$, $T>0$ and 
$ u \in C([-T, T]:H^s(\T))$ be a solution of \eqref{5KdV3}. 
Then $\ha{v}(t,k):= e^{-t \phi_{\vp}(k)} \ha{u} (t,k) $ satisfies the following equation for each $k \in \Z$:
\begin{align} \label{NF21}
\p_t (\ha{v}(t,k)+ F_{\vp, L} (\ha{v}) (t,k)) = G_{\vp, L} (\ha{v}) (t, k), 
\end{align} 
where
\begin{align*}
F_{\vp, L} (\ha{v})(t,k) := & \sum_{i=1}^2 \Lambda_{\vp}^{(2)} ( \ti{L}_{i, \vp}^{(2)} \chi_{>L}^{(2)},  \ha{v}(t) ) (t,k) 
+ \sum_{i=1}^{14} \Lambda_{\vp}^{(3)} ( \ti{L}_{i, \vp}^{(3)} \chi_{>L}^{(3)},  \ha{v}(t)) (t,k) \\ 
& +  \sum_{i=1}^{4} \Lambda_{\vp}^{(4)} ( \ti{L}_{i, \vp}^{(4)} \chi_{>L}^{(4)} ,  \ha{v} (t)) (t,k) 
\end{align*}
and 
\begin{align*}
& G_{\vp, L}(\ha{v}) (t,k) \\ 
& \hspace{0.5cm} := \sum_{i=1}^2 \Lambda_{\vp}^{(2)} ( \ti{L}_{i, \vp}^{(2)} \Phi_{\vp}^{(2)} \chi_{\le L}^{(2)} , \ha{v} (t) ) (t,k) 
+ \sum_{i=1}^{14} \Lambda_{\vp}^{(3)} (\ti{L}_{i, \vp}^{(3)} \Phi_{\vp}^{(3)} \chi_{\le L}^{(3)} ,  \ha{v} (t) ) (t,k) \\
& \hspace{0.5cm} + \sum_{i=1}^4 \Lambda_{\vp}^{(4)} (\ti{L}_{i, \vp}^{(4)} \Phi_{\vp}^{(4)} \chi_{ \le L}^{(4)}, \ha{v}(t)) (t,k)
+  \sum_{i=1}^{12} \Lambda_{\vp}^{(3)} (\ti{M}_{i, \vp}^{(3)}, \ha{v}(t) ) (t,k) \\ 
& \hspace{0.5cm} + \sum_{i=1}^{11} \Lambda_{\vp}^{(4)} (\ti{M}_{i, \vp}^{(4)}, \ha{v} (t) ) (t,k)
+ \sum_{i=1}^{2} \Lambda_{\vp}^{(5)} (\ti{M}_{i,  \vp}^{(5)}, \ha{v}(t)) (t,k)
+  \Lambda_{\vp}^{(6)} (\ti{M}_{1, \vp}^{(6)} , \ha{v}(t) ) (t,k).
\end{align*}
Here, we put 
\begin{align*}
& q_1^{(2)}(k_1,k_2): =-i k_{1,2} (k_1^2+k_2^2+k_{1,2}^2),  \hspace{0.5cm} 
q_2^{(2)}(k_1, k_2): =-i k_1 k_2 k_{1,2}, \\
& q_1^{(3)} (k_1, k_2, k_3):= ik_{1,2,3}, \\
& q_2^{(3)} (k_1,k_2, k_3):= - \frac{\be^2}{10} i\, \frac{k_{1,2}}{k_1 k_2} \, \big( k_3^2 +k_{1,2} k_3 +k_1^2+k_1k_2+k_2^2  \big), \\
& q_3^{(3)}(k_1, k_2, k_3):= \frac{\be^2}{5} ik_3 , \\
&Q_1^{(2)}:= \frac{\be}{4} q_1^{(2)}, \hspace{0.5cm} Q_2^{(2)}:= \big( \alpha -\frac{\be}{2} \big) q_2^{(2)}, 
\hspace{0.5cm} Q^{(2)}:=Q_1^{(2)}+ Q_2^{(2)} , \\
&Q^{(3)} : = 
10 \ga  q_1^{(3)} \chi_{NR1}^{(3)} -10 \ga q_1^{(3)} [3\chi_{R3}^{(3)}]_{sym}^{(3)} + \frac{\be^2}{5} q_1^{(3)} [ \chi_{R1}^{(3)} (1-\chi_{R2}^{(3)})]_{sym}^{(3)}.
\end{align*}
Note that all multipliers above except $q_2^{(3)}$ and $q_3^{(3)}$ are symmetric 
and $q_2^{(3)} (k_1, k_2, k_3)$ is symmetric with $(k_1, k_2)$.

(i) The multipliers $\{L_{i, \vp}^{(2)}\}_{i=1}^2$ and $\{L_{i, \vp}^{(3)}\}_{i=1}^{14}$ are defined as below:
\begin{align*}
L_{1, \vp}^{(2)}  & := -Q^{(2)} \chi_{NR1}^{(2)} \, 2 \chi_{H1}^{(2)}/\Phi_{\vp}^{(2)}, \hspace{0.5cm} 
L_{2, \vp}^{(2)}  := -Q^{(2)} \chi_{NR1}^{(2)} \chi_{H2}^{(2)}/\Phi_{\vp}^{(2)}, \\ 
L_{1, \vp}^{(3)}  & := -10 \ga q_1^{(3)} \chi_{NR1}^{(3)} \, 3\chi_{H1}^{(3)}/\Phi_{\vp}^{(3)},
\hspace{0.3cm} L_{2, \vp}^{(3)}  := -10 \ga q_1^{(3)} \chi_{NR1}^{(3)} (1- [3 \chi_{H1}^{(3)} ]_{sym}^{(3)} )/\Phi_{\vp}^{(3)},  \\
L_{3, \vp}^{(3)}  &:= 4 \Big[ \frac{Q^{(2)}}{ \Phi_{\vp}^{(2)} } \chi_{NR1}^{(2)} \chi_{>L}^{(2)}  \Big]_{ext1}^{(3)} [ Q^{(2)} \chi_{NR1}^{(2)} ]_{ext2}^{(3)}
 \chi_{NR(1,2)}^{(3)}/\Phi_{\vp}^{(3)},   \\
L_{4, \vp}^{(3)}  &:= q_2^{(3)} \chi_{H1}^{(3)} \chi_{NR3}^{(3)} /  \Phi_{\vp}^{(3)}, 
\hspace{0.5cm}
L_{5, \vp}^{(3)}= q_3^{(3)} \chi_{H1}^{(3)} \chi_{NR3}^{(3)} / \Phi_{\vp}^{(3)}, \\
 L_{6, \vp}^{(3)}  &:= 4 \Big[ \frac{Q_1^{(2)}}{\Phi_{0}^{(2)}} \chi_{NR1}^{(2)} \Big]_{ext1}^{(3)} [ Q_1^{(2)} \chi_{NR1}^{(2)} ]_{ext2}^{(3)}
 \,\chi_{NR(1,1)}^{(3)} (1-\chi_{H1}^{(3)} ) \chi_{A1}^{(3)}/\Phi_{\vp}^{(3)} , \\
 L_{7, \vp}^{(3)} &:= 4 \Big[ \frac{Q_1^{(2)}}{\Phi_{0}^{(2)}} \chi_{NR1}^{(2)} \Big]_{ext1}^{(3)} [ Q_1^{(2)} \chi_{NR1}^{(2)} ]_{ext2}^{(3)}
\, \chi_{NR(1,1)}^{(3)}(1 -\chi_{H1}^{(3)} ) (1- \chi_{A1}^{(3)}) \chi_{NR1}^{(3)}/\Phi_{\vp}^{(3)} , \\
 L_{8, \vp}^{(3)}  &:= 4 \Big[ \frac{Q_1^{(2)}}{\Phi_{0}^{(2)}} \chi_{NR1}^{(2)} \Big]_{ext1}^{(3)} [ Q_2^{(2)} \chi_{NR1}^{(2)} ]_{ext2}^{(3)}
 \,\chi_{NR(1,1)}^{(3)} (1- \chi_{R1}^{(3)})/\Phi_{\vp}^{(3)}, \\
  L_{9, \vp}^{(3)}  &:= 4 \Big[ \frac{Q_2^{(2)}}{\Phi_{0}^{(2)}} \chi_{NR1}^{(2)} \Big]_{ext1}^{(3)} [ Q_1^{(2)} \chi_{NR1}^{(2)} ]_{ext2}^{(3)}
 \,\chi_{NR(1,1)}^{(3)} (1- \chi_{R1}^{(3)})/\Phi_{\vp}^{(3)}, \\
  L_{10, \vp}^{(3)}  &:= 4 \Big[ \frac{Q_2^{(2)}}{\Phi_{0}^{(2)}} \chi_{NR1}^{(2)} \Big]_{ext1}^{(3)} [ Q_2^{(2)} \chi_{NR1}^{(2)} ]_{ext2}^{(3)}
 \,\chi_{NR(1,1)}^{(3)} (1- \chi_{R1}^{(3)})/\Phi_{\vp}^{(3)}, \\
  L_{11, \vp}^{(3)}  &:= 2 \Big[  \frac{Q^{(2)}}{\Phi_{\vp}^{(2)}} \chi_{NR1}^{(2)} \chi_{>L}^{(2)} \Big]_{ext1}^{(3)} [ Q^{(2)} \chi_{NR1}^{(2)} ]_{ext2}^{(3)} 
\, \chi_{NR(2,2)}^{(3)} \chi_{A2}^{(3)}/\Phi_{\vp}^{(3)} , \\
  L_{12, \vp}^{(3)}  &:= 2 \Big[ \frac{Q^{(2)}}{\Phi_{\vp}^{(2)}} \chi_{NR1}^{(2)} \chi_{>L}^{(2)}  \Big]_{ext1}^{(3)} [ Q^{(2)} \chi_{NR1}^{(2)} ]_{ext2}^{(3)} 
\, \chi_{NR(2,2)}^{(3)} (1-\chi_{A2}^{(3)}) \chi_{NR1}^{(3)}/\Phi_{\vp}^{(3)}, \\
 L_{13, \vp}^{(3)}  &:=2 \Big[ \frac{Q^{(2)}}{\Phi_{\vp}^{(2)}} \chi_{NR1}^{(2)} \chi_{>L}^{(2)} \Big]_{ext1}^{(3)} [ Q^{(2)} \chi_{NR1}^{(2)} ]_{ext2}^{(3)}
 \, \chi_{NR(2,1)}^{(3)} \chi_{A3}^{(3)}/\Phi_{\vp}^{(3)} , \\
  L_{14, \vp}^{(3)}  &:= 2 \Big[ \frac{Q^{(2)}}{\Phi_{\vp}^{(2)}} \chi_{NR1}^{(2)} \chi_{>L}^{(2)}  \Big]_{ext1}^{(3)} [ Q^{(2)} \chi_{NR1}^{(2)} ]_{ext2}^{(3)}
 \, \chi_{NR(2,1)}^{(3)} (1-\chi_{A3}^{(3)})\chi_{NR1}^{(3)}/\Phi_{\vp}^{(3)} , \\
 L_{15, \vp}^{(3)}  & :=2  \Big[ \frac{Q^{(2)}}{\Phi_{\vp}^{(2)}} \chi_{NR1}^{(2)}  \chi_{>L}^{(2)} \Big]_{ext1}^{(3)} 
 [ Q^{(2)} \chi_{NR1}^{(2)}  ]_{ext2}^{(3)} \, [\chi_{H2}^{(2)}]_{ext1}^{(3)} \chi_{NR1}^{(3)}/\Phi_{\vp}^{(3)}, 
\end{align*}

(ii) The multipliers $\{L_{i, \vp}^{(4)}\}_{i=1}^{4}$ are defined as below:
\begin{align*}
 L_{1, \vp}^{(4)}  &:= -2 \Big[ \frac{q_2^{(3)}}{\Phi_0^{(3)}} \chi_{NR3}^{(3)} \Big]_{ext1}^{(4)} [ Q_1^{(2)} \chi_{NR1}^{(2)}]_{ext2}^{(4)}
 \, \chi_{H1}^{(4)} (1-\chi_{R1}^{(4)}) (1-\chi_{R5}^{(4)})/\Phi_{\vp}^{(4)},  \\
  L_{2, \vp}^{(4)}  &:= -2 \Big[ \frac{q_2^{(3)}}{\Phi_{0}^{(3)}} \chi_{NR3}^{(3)} \Big]_{ext1}^{(4)} [ Q_1^{(2)} \chi_{NR1}^{(2)} ]_{ext2}^{(4)}
 \, \chi_{NR(1,1)}^{(4)}(1-\chi_{H1}^{(4)} ) \chi_{A4}^{(4)}/\Phi_{\vp}^{(4)},   \\
   L_{3, \vp}^{(4)}  &:= -2 \Big[ \frac{q_2^{(3)}}{\Phi_{0}^{(3)}} \chi_{NR3}^{(3)} \Big]_{ext1}^{(4)} [ Q_2^{(2)} \chi_{NR1}^{(2)} ]_{ext2}^{(4)}
 \, \chi_{NR(1,1)}^{(4)} \chi_{A4}^{(4)}/\Phi_{\vp}^{(4)}, \\
  L_{4, \vp}^{(4)} &:= - \Big[  \frac{q_2^{(3)}}{\Phi_{\vp}^{(3)}} \chi_{NR3}^{(3)} \chi_{>L}^{(3)} \Big]_{ext1}^{(4)} [ Q^{(2)} \chi_{NR1}^{(2)} ]_{ext2}^{(4)}
 \, \chi_{NR(2,1)}^{(4)}/\Phi_{\vp}^{(4)}.
\end{align*}

(iii) The multipliers $\{M_{i, \vp}^{(3)}\}_{i=1}^{12}$ are defined as below:
\begin{align*}
M_{1,\vp}^{(3)}&:=10 \gamma q_1^{(3)} \, 3\chi_{R3}^{(3)}, \hspace{2em} 
M_{2, \vp}^{(3)}:= - \frac{\beta^2}{5} q_1^{(3)} \chi_{H1}^{(3)} \chi_{R1}^{(3)} (1- \chi_{R2}^{(3)}), \\
M_{3,\vp}^{(3)}&:= -\frac{\beta^2}{5} q_1^{(3)} \chi_{R1}^{(3)} (1-\chi_{R2}^{(3)} ) (1-\chi_{H1}^{(3)}), \\
M_{4, \vp}^{(3)} &:=  
4 \Big[ \frac{Q_1^{(2)}   }{ \Phi_0^{(2)}  } \chi_{NR1}^{(2)} \Big]_{ext1}^{(3)} [ Q_1^{(2)} \chi_{NR1}^{(2)} ]_{ext2}^{(3)} \, \chi_{H1}^{(3)} \chi_{R1}^{(3)}, \\
M_{5, \vp}^{(3)} &:= 4 \Big [ \frac{Q_1^{(2)} }{ \Phi_0^{(2)} }  \chi_{NR1}^{(2)} \Big ]_{ext1}^{(3)} [ Q_1^{(2)} \chi_{NR1}^{(2)}  ]_{ext2}^{(3)} \, 
\chi_{NR(1,1)}^{(3)}(1-\chi_{H1}^{(3)}) (1-\chi_{A1}^{(3)}) (1- \chi_{NR1}^{(3)}) , \\
M_{6, \vp}^{(3)} &:=  4  \Big( \Big[ \frac{Q_1^{(2)}   }{ \Phi_0^{(2)}} \chi_{NR1}^{(2)} \Big]_{ext1}^{(3)} [ Q_2^{(2)} \chi_{NR1}^{(2)} ]_{ext2}^{(3)} 
+ \Big[ \frac{Q_2^{(2)}}{\Phi_0^{(2)} }  \chi_{NR1}^{(2)}  \Big]_{ext1}^{(3)} [Q_1^{(2)} \chi_{NR1}^{(2)}]_{ext2}^{(3)} \Big)
 \chi_{NR(1,1)}^{(3)} \chi_{R1}^{(3)},  \\ 
M_{7, \vp}^{(3)} &:=  
4 \Big[ \frac{Q_2^{(2)}   }{ \Phi_0^{(2)}  } \chi_{NR1}^{(2)} \Big]_{ext1}^{(3)} [ Q_2^{(2)} \chi_{NR1}^{(2)} ]_{ext2}^{(3)} \, \chi_{NR(1,1)}^{(3)} \chi_{R1}^{(3)},  \\
M_{8, \vp}^{(3)} &:= 4  \Big[ \Big( - \frac{Q^{(2)}   }{ \Phi_0^{(2)}  } \Big) \chi_{NR1}^{(2)}  \chi_{\le L}^{(2)} \Big]_{ext1}^{(3)} [Q^{(2)} \chi_{NR1}^{(2)} ]_{ext2}^{(3)}  \chi_{NR(1,1)}^{(3)} , \\
M_{9, \vp}^{(3)} &:= 4 \Big[ \Big( \frac{ Q^{(2)} }{ \Phi_{\vp}^{(2)}} - \frac{Q^{(2)}}{ \Phi_0^{(2)} } \Big) \chi_{NR1}^{(2)} \chi_{>L}^{(2)}  \Big]_{ext1}^{(3)} 
 [Q^{(2)} \chi_{NR1}^{(2)} ]_{ext2}^{(3)} \, \chi_{NR(1,1)}^{(3)} , \\
M_{10, \vp}^{(3)} &:= 2 \Big[ \frac{Q^{(2)} }{ \Phi_{\vp}^{(2)}  } \chi_{NR1}^{(2)} \chi_{>L}^{(2)}  \Big]_{ext1}^{(3)} [ Q^{(2)} \chi_{NR1}^{(2)} ]_{ext2}^{(3)} \, 
\chi_{NR(2,2)}^{(3)} (1-\chi_{A2}^{(3)} ) (1-\chi_{NR1}^{(3)}), \\
M_{11, \vp}^{(3)} &:=  2 \Big[ \frac{Q^{(2)}   }{ \Phi_{\vp}^{(2)}  } \chi_{NR1}^{(2)} \chi_{>L}^{(2)} \Big]_{ext1}^{(3)} [Q^{(2)} \chi_{NR1}^{(2)} ]_{ext2}^{(3)} \, 
\chi_{NR(2,1)}^{(3)} (1- \chi_{A3}^{(3)})(1- \chi_{NR1}^{(3)}), \\
M_{12, \vp}^{(3)} &:= 2 \Big[ \frac{Q^{(2)}   }{ \Phi_{\vp}^{(2)} } \chi_{NR1}^{(2)} \chi_{>L}^{(2)} \Big]_{ext1}^{(3)} [ Q^{(2)} \chi_{NR1}^{(2)}  ]_{ext2}^{(3)} 
[\chi_{H2}^{(2)}]_{ext1}^{(3)} (1-\chi_{NR1}^{(3)}).
\end{align*}

(iv) The multipliers $\{M_{i, \vp}^{(4)}\}_{i=1}^{11}$ are defined as below:
\begin{align*}
M_{1, \vp}^{(4)} &:= [ 2(\ti{L}_{1, \vp}^{(2)} + \ti{L}_{2, \vp}^{(2)}) \chi_{>L}^{(2)} ]_{ext1}^{(4)} [-Q^{(3)}]_{ext2}^{(4)} ,  \\
M_{2, \vp}^{(4)}& := 
[3 (\ti{L}_{1, \vp}^{(3)}+ \ti{L}_{3,\vp}^{(3)} +  \ti{L}_{5,\vp}^{(3)} +\ti{L}_{6, \vp}^{(3)} +\ti{L}_{8 ,\vp}^{(3)} + \ti{L}_{9,\vp}^{(3)} 
+ \ti{L}_{10,\vp}^{(3)} + \ti{L}_{11,\vp}^{(3)} +\ti{L}_{13,\vp}^{(3)} )  \chi_{>L}^{(3)} ]_{ext1}^{(4)}  \\
& \hspace{4.0cm} \times [-Q^{(2)} \chi_{NR1}^{(2)}  ]_{ext2}^{(4)} , \\  
M_{3, \vp}^{(4)}& := 
[3 (\ti{L}_{2, \vp}^{(3)}+ \ti{L}_{7, \vp}^{(3)}+ \ti{L}_{12,\vp}^{(3)} +  \ti{L}_{14,\vp}^{(3)}+ \ti{L}_{15,\vp}^{(3)} ) \chi_{>L}^{(3)} ]_{ext1}^{(4)} 
[-Q^{(2)} \chi_{NR1}^{(2)}  ]_{ext2}^{(4)}, \\
M_{4, \vp}^{(4)}&:= -4  \Big[ \frac{\chi_{NR3}^{(3)} \chi_{>L}^{(3)} }{ \Phi_{\vp}^{(3)} }  \Big]_{ext1}^{(4)} [Q^{(2)} \chi_{NR1}^{(2)} ]_{ext2}^{(4)} 
[ [3 q_2^{(3)} \chi_{H1}^{(3)} ]_{sym}^{(3)} ]_{ext1}^{(4)}  \, \chi_{NR(1,2)}^{(4)}, \\
M_{5, \vp}^{(4)}&:= -2  \Big[ \frac{q_2^{(3)} }{ \Phi_0^{(3)} } \chi_{NR3}^{(3)}  \Big]_{ext1}^{(4)} [Q_1^{(2)} \chi_{NR1}^{(2)}  ]_{ext2}^{(4)} \, \chi_{H1}^{(4)} \chi_{R1}^{(4)},\\
M_{6, \vp}^{(4)}&:=  -2  \Big[ \frac{q_2^{(3)} }{ \Phi_0^{(3)} } \chi_{NR3}^{(3)}  \Big]_{ext1}^{(4)}  [Q_1^{(2)} \chi_{NR1}^{(2)} ]_{ext2}^{(4)} \, 
\chi_{H1}^{(4)} (1-\chi_{R1}^{(4)}) \chi_{R5}^{(4)}, \\
M_{7, \vp}^{(4)}& :=  -2  \Big[ \frac{q_2^{(3)} }{ \Phi_0^{(3)} } m_{NR3}^{(3)}  \Big]_{ext1}^{(4)}  
[Q_1^{(2)} \chi_{NR1}^{(2)} ]_{ext2}^{(4)} \, \chi_{NR(1,1)}^{(4)}(1 -\chi_{H1}^{(4)}) (1-\chi_{A4}^{(4)}),  \\
M_{8, \vp}^{(4)}&:= -2  \Big[ \frac{q_2^{(3)} }{ \Phi_0^{(3)} } \chi_{NR3}^{(3)}  \Big]_{ext1}^{(4)} [Q_2^{(2)} \chi_{NR1}^{(2)}  ]_{ext2}^{(4)}
 \, \chi_{NR(1,1)}^{(4)} (1-\chi_{A4}^{(4)}),\\
M_{9, \vp}^{(4)}& :=  -2  \Big[ \Big( - \frac{q_2^{(3)} }{ \Phi_{0}^{(3)} } \Big) \chi_{NR3}^{(3)} \chi_{ \le L }^{(3)}  \Big]_{ext1}^{(4)}  
[ Q^{(2)} \chi_{NR1}^{(2)}  ]_{ext2}^{(4)} \, \chi_{NR(1,1)}^{(4)} ,  \\
M_{10, \vp}^{(4)}& := -2  \Big[ \Big( \frac{q_2^{(3)} }{ \Phi_{\vp}^{(3)} }- \frac{q_2^{(3)}}{ \Phi_0^{(3)} }  \Big) 
\chi_{NR3}^{(3)} \chi_{>L}^{(3)} \Big]_{ext1}^{(4)}  [Q^{(2)} \chi_{NR1}^{(2)} ]_{ext2}^{(4)} \, \chi_{NR(1,1)}^{(4)},  \\
M_{11, \vp}^{(4)}& := -2   \Big[ \frac{ \chi_{NR3}^{(3)} \chi_{>L}^{(3)}  }{ \Phi_{\vp}^{(3)} }  \Big]_{ext1}^{(4)}  [Q^{(2)} \chi_{NR1}^{(2)} ]_{ext2}^{(4)} 
[ [3 q_2^{(3)} \chi_{H1}^{(3)} ]_{sym}^{(3)} ]_{ext1}^{(4)} \, \chi_{NR(2,2)}^{(4)} , 
\end{align*}

(v) The multipliers $\{M_{i, \vp}^{(5)}\}_{i=1}^{2}$ and $M_{1, \vp}^{(6)}$ are defined as below:
\begin{align*}
M_{1, \vp}^{(5)} & := \Big[ 3 \sum_{i=1}^{14} \ti{L}_{i, \vp}^{(3)} \chi_{>L}^{(3)} \Big]_{ext1}^{(5)} [-Q^{(3)}]_{ext2}^{(5)}, \\
M_{2, \vp}^{(5)} & := [ 4 (\ti{L}_{1, \vp}^{(4)} + \ti{L}_{2,\vp}^{(4)} +\ti{L}_{3, \vp}^{(4)} + \ti{L}_{4, \vp}^{(4)} ) \chi_{>L}^{(4)} ]_{ext1}^{(5)} [-Q^{(2)} \chi_{NR1}^{(2)} ]_{ext2}^{(5)}, \\
M_{1, \vp}^{(6)}  & := [ 4 (\ti{L}_{1, \vp}^{(4)} + \ti{L}_{2,\vp}^{(4)} +\ti{L}_{3, \vp}^{(4)}+ \ti{L}_{4, \vp}^{(4)} ) \chi_{>L}^{(4)} ]_{ext1}^{(6)} [-Q^{(3)}]_{ext2}^{(6)}.
\end{align*}
\end{prop}
\begin{rem}
Precisely speaking, $L_{1, \vp}^{(2)}   := -Q^{(2)} \chi_{NR1}^{(2)} \, 2 \chi_{H1}^{(2)}/\Phi_{\vp}^{(2)}$ means
\begin{align*}
L_{1, \vp}^{(2)} :=
\begin{cases}
-Q^{(2)} \chi_{NR1}^{(2)} \, 2 \chi_{H1}^{(2)}/\Phi_{\vp}^{(2)} \ \ &\text{ when } \Phi_{\vp}^{(2)}(k_1,k_2)\neq 0\\
0\ \ &\text{ when } \Phi_{\vp}^{(2)}(k_1,k_2)= 0.
\end{cases}
\end{align*}
We adapt the same rule in the definition of each multiplier when its denominator is equal to $0$.
\end{rem}

Before we prove Proposition \ref{prop_NF2}, we prepare some propositions and lemmas.
\begin{prop} \label{prop_req1}
Let $s \ge 1$, $\vp \in L^2(\T)$, $T>0$ and $u \in C([-T, T]:H^s(\T))$ be a solution of \eqref{5KdV3}. 
Put $v:= U_\vp(-t)u$. Then, $v \in C([-T, T]: H^s(\T))\cap C^1([-T, T]: H^{s-3}(\T))$ and 
$\ha{v}(t,k)$ satisfies 
\begin{align} \label{eq20} 
\p_t \ha{v} (t,k) = \La_{\vp}^{(2)}( -Q^{(2)} \chi_{NR1}^{(2)}, \ha{v})(t,k) + \La_{\vp}^{(3)}( -Q^{(3)}, \ha{v})(t,k),
\end{align}
for each $k \in \Z$, where $U_{\vp}(t)= \mathcal{F}^{-1} \exp( t \phi(k)) \mathcal{F}_x$. Moreover, we have 
\begin{align}
\|\p_t v\|_{L_T^\infty H^{s-3}}\lec \|v\|_{L_T^\infty H^{s}}^2+\|v\|_{L_T^\infty H^{s}}^3. \label{eq_es}
\end{align}
\end{prop}
\begin{proof}
Since $u\in C([-T, T]:H^s(\T))$ satisfies \eqref{5KdV3} and $\ha{v}(t,k)=e^{-t\phi_{\vp} (k)} \ha{u}(t,k)$, it follows that
$v\in C([-T,T]: H^{s}(\T))$ and
\begin{equation}\label{EE2}
\p_t \ha{v} (t,k) = \sum_{j=1,2,3} e^{-t \phi_{\vp} (k) } \mathcal{F}_x [J_j(u)](t, k)
\end{equation}
By Remark \ref{rem_thm_main}, $J_1(u), J_2(u), J_3(u) \in C([-T,T]: H^{s-3}(\T))$.
Thus, we obtain $v\in C^1([-T,T]: H^{s-3}(\T))$ and \eqref{eq_es}.
A direct computation yields that  
\EQQ{
& e^{-t \phi_{\vp} (k) } \mathcal{F}_x \big[ -\alpha \p_x( (\p_x u )^2 ) - \beta \p_x (u \p_x^2 u)  \big] (t,k) \\
&\hspace{0.5cm}  =\sum_{k=k_{1,2}} e^{- t \Phi_{\vp}^{(2)}} 
\Big\{ -\Big( \alpha- \frac{\be}{2}  \Big) q_2^{(2)}  - \frac{\beta}{4} q_1^{(2)}  \Big\} \prod_{j=1}^2 \ha{v}(t,k_j) \\
& \hspace{0.5cm} =\sum_{k=k_{1,2}} e^{- t \Phi_{\vp}^{(2)}}
(-Q^{(2)}) \prod_{j=1}^2 \ha{v}(t,k_j) 
}
and 
\EQQ{
e^{-t \phi_{\vp} (k) } \mathcal{F}_x \Big[ \beta \int_{\mathbb{T}} u~dx \p_x^3 u  \Big](t,k)
= \sum_{k=k_{1,2}} e^{- t \Phi_{\vp}^{(2)}} Q^{(2)}  [2 \chi_{R1}^{(2)}]_{sym}^{(2)} \prod_{j=1}^2 \ha{v}(t,k_j).
} 
Note that $1= \chi_{NR2}^{(2)} + [2 \chi_{R1}^{(2)}]_{sym}^{(2)} -\chi_{R4}^{(2)}$ and $Q^{(2)} \chi_{R4}^{(2)}=0$ 
where 
\begin{align*}
\chi_{NR2}^{(2)} = 
\begin{cases}
\dis 1, \hspace{0.2cm} \text{when} \hspace{0.2cm} k_1 k_2 \neq 0 \\
0, \hspace{0.2cm} \text{otherwise}
\end{cases}.
\end{align*}
Thus, we have 
\EQ{ \label{EE4}
e^{-t \phi_{\vp} (k) } \mathcal{F}_x [ J_1 (u) ] (t,k) & 
= \sum_{k=k_{1,2}} e^{- t \Phi_{\vp}^{(2)}} 
(-Q^{(2)}) (1-[2\chi_{R1}^{(2)}]_{sym}^{(2)} ) \prod_{j=1}^2 \ha{v}(t,k_j)  \\
&  = \sum_{k=k_{1,2}} e^{- t \Phi_{\vp}^{(2)}} 
(-Q^{(2)}) \chi_{NR2}^{(2)} \prod_{j=1}^2 \ha{v}(t,k_j) \\
&  = \sum_{k=k_{1,2}} e^{- t \Phi_{\vp}^{(2)}} 
(-Q^{(2)}) \chi_{NR1}^{(2)} \prod_{j=1}^2 \ha{v}(t,k_j) \\
& =\Lambda_{\vp}^{(2)} (-Q^{(2)} \chi_{NR1}^{(2)}, \ha{v} (t)) (t,k) .
}
Here we used $Q^{(2)} \chi_{NR1}^{(2)}= Q^{(2)} \chi_{NR2}^{(2)} $ in the third equality. 
Since
\EQQ{
e^{-t \phi_{\vp} (k) } \mathcal{F}_x \Big[ \int_{\mathbb{T}} u^2 \, dx \, \p_x u  \Big] (t,k)
= \sum_{k=k_{1,2,3}} e^{- t \Phi_{\vp}^{(3)}} q_1^{(3)}  [ \chi_{R1}^{(3)}]_{sym}^{(3)} \prod_{j=1}^3 \ha{v}(t,k_j)
}
and 
\EQQ{ 
e^{-t \phi_{\vp} (k) } \mathcal{F}_x \Big[ \Big(   \int_{\mathbb{T}} u \, dx \Big)^2  \p_x u  \Big](t,k)
= \sum_{k=k_{1,2,3}}  e^{- t \Phi_{\vp}^{(3)}} q_1^{(3)}  [\chi_{R1}^{(3)} \chi_{R2}^{(3)} ]_{sym}^{(3)} \prod_{j=1}^3 \ha{v}(t,k_j),
}
we have
\EQ{ \label{EE5}
 e^{-t \phi_{\vp} (k) } \mathcal{F}_x [ J_2 (u ) ] (t,k)
 &= \sum_{k=k_{1,2,3} } e^{- t \Phi_{\vp}^{(3)}} \Big( - \frac{\be^2}{5} \Big) q_1^{(3)}  [ \chi_{R1}^{(3)} (1-\chi_{R2}^{(3)}) ]_{sym}^{(3)} 
 \prod_{j=1}^3 \ha{v}(t,k_j)  \\
&= \La_{\vp}^{(3)} \Big( -\frac{\be^2}{5} q_1^{(3)} [\chi_{R1}^{(3)} (1-\chi_{R2}^{(3)}) ]_{sym}^{(3)}, \ha{v} (t) \Big) (t,k). 
}
Since $1= \chi_{NR1}^{(3)}+  [3 \chi_{R1}^{(3)} ]_{sym}^{(3)}- [3 \chi_{R3}^{(3)}]_{sym}^{(3)} +\chi_{R4}^{(3)}$ and $q_1^{(3)} \chi_{R4}^{(3)} =0$, we have 
\EQ{ \label{EE6}
& e^{-t \phi_{\vp} (k) } \mathcal{F}_x [J_3(u)] (t,k)  = 
\sum_{k=k_{1,2,3} } e^{- t \Phi_{\vp}^{(3)}} ( -10 \ga) q_1^{(3)}  (1-  [3 \chi_{R1}^{(3)} ]_{sym}^{(3)} ) 
 \prod_{j=1}^3 \ha{v}(t ,k_j)  \\
& = \sum_{k=k_{1,2,3} } e^{- t \Phi_{\vp}^{(3)}} ( -10 \ga) q_1^{(3)}  
(\chi_{NR1}^{(3)} - [3 \chi_{R3}^{(3)}]_{sym}^{(3)} )  \prod_{j=1}^3 \ha{v}(t,k_j) \\
& = \Lambda_{\vp}^{(3)} ( -10 \ga q_1^{(3)} \chi_{NR1}^{(3)}, \ha{v} (t) ) (t,k)
+ \Lambda_{\vp}^{(3)} (10 \ga q_1^{(3)} [3 \chi_{R3}^{(3)}]_{sym}^{(3)}, \ha{v}(t) ) (t,k)
}
Therefore, collecting (\ref{EE2})--(\ref{EE6}), we obtain (\ref{eq20}). 
\end{proof}

\begin{prop} \label{prop_NF11}
Let $N \in \{ 2,3,4 \}$, $s \ge 1$, $\vp \in L^2(\T)$, $T>0$ and an $N$-multiplier $m^{(N)}$ satisfy
\begin{align*}
|m^{(N)} (k_1, \dots , k_N ) | \lesssim \langle k_{\max} \rangle^{-1}, 
\hspace{0.5cm} | (m^{(N)} \Phi_{\vp}^{(N)}) (k_1, \dots, k_N) | \lesssim \langle k_{1,\dots, N}  \rangle \langle k_{\max} \rangle^2.
\end{align*} 
If $ u \in C([-T, T]:H^s(\T))$ is a solution of (\ref{5KdV3}), then 
$\ha{v}(t,k)=e^{-t \phi_{\vp} (k)} \ha{u}(t,k) $ satisfies  the following equation for each $k \in \Z$:
\begin{align}
\label{NF11}
& \p_t \Lambda_{\vp}^{(N)} (\ti{m}^{(N)}, \ha{v}(t))(t,k) = 
\Lambda_{\vp}^{(N)} (-\ti{m}^{(N)} \Phi_{\vp}^{(N)}, \ha{v} (t) )(t,k) \nonumber \\
& \hspace{0.3cm} 
+ \Lambda_{\vp}^{(N+1)} \big( \big[ [N \ti{m}^{(N)} ]_{ext1}^{(N+1)} [ -Q^{(2)} \chi_{NR1}^{(2)} ]_{ext2}^{(N+1)}  \big]_{sym}^{(N+1)}, \ha{v} (t) \big)(t,k) \nonumber \\
& \hspace{0.3cm} 
+ \Lambda_{\vp}^{(N+2)} \big( \big[ [N \ti{m}^{(N)} ]_{ext1}^{(N+2)} [ -Q^{(3)} ]_{ext2}^{(N+2)}  \big]_{sym}^{(N+2)}, \ha{v} (t) \big)(t,k).
\end{align} 
\end{prop}

Note that any $L_{j, \vp}^{(N)}$ defined in Proposition~\ref{prop_NF2} satisfies 
\begin{align*}
|L_{j, \vp}^{(N)} \chi_{>L}^{(N)} | \lesssim \langle k_{\max}  \rangle^{-1},
\hspace*{2em} 
|L_{j, \vp}^{(N)} \Phi_{\vp}^{(N)} \chi_{>L}^{(N)} | \lesssim \langle k_{1,\dots, N} \rangle \langle k_{\max} \rangle^2
\end{align*}
for $L \gg \max\{ 1, |\be E_0(\vp)|  \}$. Thus, we can apply Proposition~\ref{prop_NF11} 
with $\ti{m}^{(N)}=\ti{L}_{j, \vp}^{(N)} \chi_{>L}^{(N)}$. 
For details, see Remark~\ref{rem_pwb11}.

\begin{proof}
At least formally, we have
\begin{align}\label{NF112}
\p_t \La_{\vp}^{(N)} (\ti{m}^{(N)}, \ha{v} (t))(t, k)
&=\sum_{k=k_{1, \dots, N}} \frac{\p}{\p t} \bigg\{ e^{-t \Phi_{\vp}^{(N)}} \ti{m}^{(N)} \prod_{j=1}^N \ha{v} (t, k_j) \bigg\}\\
&= \mathcal{N}_1(t,k) +\mathcal{N}_2(t, k) \notag
\end{align}
where
\begin{align*}
\mathcal{N}_1(t,k)&=  \sum_{k=k_{1,\cdots, N}}  e^{-t \Phi_{\vp}^{(N)}} (-\ti{m}^{(N)} \Phi_{\vp}^{(N)} ) \prod_{j=1}^N \ha{v}(t , k_j),\\
&=\Lambda_{\vp}^{(N)} (-\ti{m}^{(N)} \Phi_{\vp}^{(N)}, \ha{v})(t,k)\\
\mathcal{N}_2(t, k)&= \sum_{k=k_{1,\cdots, N}}  e^{-t \Phi_{\vp}^{(N)}} \ti{m}^{(N)} \frac{\p}{\p t} \, \prod_{j=1}^N \ha{v}(t , k_j).
\end{align*}
By $| \ti{m}^{(N)} \Phi_{\vp}^{(N)} | \le [ |m^{(N)} \Phi_{\vp}^{(N)} | ]_{sym}^{(N)}\lesssim \langle k_{1, \dots, N}  \rangle  \langle k_{\max} \rangle^2 $ and Lemma~\ref{lem_nl2}, 
\begin{align*} 
\Big\| \sum_{k=k_{1,\cdots, N}}  |e^{-t \Phi_{\vp}^{(N)}}| |(-\ti{m}^{(N)} \Phi_{\vp}^{(N)} )| \prod_{j=1}^N |\ha{v}(t , k_j)| \Big\|_{L_T^{\infty} l_{s-3}^2} \lesssim \| v \|_{L_T^{\infty} H^s}^N
\end{align*}
and by $|\ti{m}^{(N)}| \langle k_{\max} \rangle^3  \lesssim \langle k_{1, \dots, N} \rangle \langle k_{\max} \rangle^2 $, 
Lemma~\ref{lem_nl2} and \eqref{eq_es} in Proposition \ref{prop_req1},
\begin{align*}
&\Big\| \sum_{k=k_{1,\dots, N}}  |e^{-t \Phi_{\vp}^{(N)}}| |\ti{m}^{(N)} | \sum_{i=1}^N |\p_t \ha{v}(t , k_i)|\, \prod_{j\in \{1,\ldots,N\}\setminus \{i\}} |\ha{v}(t , k_j)|  \Big\|_{L_T^{\infty} l_{s-3}^2  }\\
& \lesssim \Big\| \sum_{k=k_{1,\cdots, N}}  |e^{-t \Phi_{\vp}^{(N)}}| |\ti{m}^{(N)}|\langle k_{\max} \rangle^3  \sum_{i=1}^N \langle k_i \rangle^{-3}|\p_t \ha{v}(t , k_i)|\, \prod_{j\in \{1,\ldots,N\}\setminus \{i\}} |\ha{v}(t , k_j)|  \Big\|_{L_T^{\infty} l_{s-3}^2  }\\
& \lesssim \| \p_t v \|_{L_T^{\infty} H^{s-3}} \| v \|_{L_T^{\infty} H^s }^{N-1} 
\lesssim \| v \|_{L_T^{\infty} H^s}^{N+1} + \| v \|_{L_T^{\infty} H^s}^{N+2}. 
\end{align*}
Thus, for each $k\in \Z$, the convergence  of $\mathcal{N}_1(t,k)$ and $\mathcal{N}_2(t,k)$ is absolute and uniform in $t \in [-T, T]$. Therefore, changing the sum and the time differentiation in \eqref{NF112} can be verified strictly. 
Moreover,
\begin{align*} 
\mathcal{N}_2 (t, k) &= \sum_{k=k_{1,\cdots, N}}  e^{-t \Phi_{\vp}^{(N)}}  N \ti{m}^{(N)} \, \p_t \ha{v} (t, k_N)\prod_{j=1}^{N-1} \ha{v}(t , k_j)  \nonumber \\
&= \La_{\vp}^{(N)} \big(N \ti{m}^{(N)}, \ha{v}(t), \dots, \ha{v}(t), \p_t \ha{v}(t) \big) (t,k)
\end{align*}
Substituting \eqref{eq20} for it, we get
\begin{align*}
\mathcal{N}_2(t,k)  
 = & \La_{\vp}^{(N)} \big(N \ti{m}^{(N)}, \ha{v}(t), \dots, \ha{v}(t), \La_{\vp}^{(2)} (-Q^{(2)} m_{NR1}^{(2)}, \ha{v} (t)  )   \big) (t,k) \nonumber \\
 & \hspace{0.3cm} +   \La_{\vp}^{(N)} \big(N \ti{m}^{(N)}, \ha{v}(t), \dots, \ha{v}(t), \La_{\vp}^{(3)}  (-Q^{(3)}, \ha{v} (t) )  \big) (t,k) \nonumber \\
 = &  \La_{\vp}^{(N+1)} \big( [ N \ti{m}^{(N)}]_{ext1}^{(N+1)} [ -Q^{(2)} \chi_{NR1}^{(2)} ]_{ext2}^{(N+1)}, \ha{v}(t) \big) (t,k) \nonumber \\
 & \hspace{0.3cm} +  \La_{\vp}^{(N+2)} \big( [ N \ti{m}^{(N)}]_{ext1}^{(N+2)} [ -Q^{(3)}]_{ext2}^{(N+2)}, \ha{v}(t) \big) (t,k) \nonumber \\
 = &  \La_{\vp}^{(N+1)} \big( \big[  [ N \ti{m}^{(N)}]_{ext1}^{(N+1)} [ -Q^{(2)} \chi_{NR1}^{(2)} ]_{ext2}^{(N+1)} \big]_{sym}^{(N+1)}, \ha{v}(t) \big) (t,k) \nonumber \\
 & \hspace{0.3cm} + \La_{\vp}^{(N+2)} \big( \big[ [ N \ti{m}^{(N)}]_{ext1}^{(N+2)} [ -Q^{(3)}]_{ext2}^{(N+2)} \big]_{sym}^{(N+2)} , \ha{v}(t) \big) (t,k).
\end{align*}
Therefore, we obtain \eqref{NF11}.
\end{proof}

\begin{lem}\label{lem_sym}
Let a $3$-multiplier $m^{(3)}(k_1,k_2,k_3)$ be symmetric with $(k_1, k_2)$.
Then, it follows that
\begin{align*}
4\big[ \big[ Q_1^{(2)}/\Phi_0^{(2)}\big]_{ext1}^{(3)} [ Q_1^{(2)}]_{ext2}^{(3)}m^{(3)} \big]_{sym}^{(3)}=\big[q_2^{(3)} m^{(3)}\big]_{sym}^{(3)}
+\big[ q_3^{(3)} m^{(3)} \big]_{sym}^{(3)}.
\end{align*}
\end{lem}

\begin{proof}
Since $Q_1^{(2)}/\Phi_0^{(2)}=\beta/10k_1k_2$ and $m^{(3)} (k_1, k_2, k_3) $ is symmetric with $(k_1, k_2)$, it follows that 
\begin{align*}
& 4 \Big[ \frac{Q_1^{(2)}}{\Phi_0^{(2)}} (k_1, k_{2,3})   Q_1^{(2)} (k_2, k_3)  m^{(3)} (k_1, k_2, k_3) \Big]_{sym}^{(3)} =2 \big[   M   m^{(3)}  \big]_{sym}^{(3)}.
\end{align*}
where
\begin{align*}
M(k_1,k_2,k_3)&:=\frac{\beta}{10k_1k_{2,3}}Q_1^{(2)} (k_2, k_3)+\frac{\beta}{10k_2k_{1,3}}Q_1^{(2)} (k_1, k_3)\\
&=-\frac{\beta^2 i}{40}\Big( \frac{k_2^2+k_3^2+k_{2,3}^2}{k_1}+\frac{k_1^2+k_3^2+k_{1,3}^2}{k_2} \Big) \\
& =-\frac{\beta^2 i}{20} \Big( \frac{k_1+k_2}{k_1k_2}k_3^2 + \frac{k_1^2 +k_2^2}{k_1k_2} k_3+ \frac{k_1^3+k_2^3}{k_1k_2} \Big) 
=\frac{q_2^{(3)}}{2}+ \frac{q_3^{(3)}}{2}.
\end{align*}
\end{proof}

\begin{lem}\label{L30}
\begin{align} \label{eq41}
& 4\Big[\Big[ \frac{ Q_1^{(2)}  }{ \Phi_0^{(2)} } \chi_{NR1}^{(3)} \Big]_{ext1}^{(3)} [ Q_1^{(2)} \chi_{NR1}^{(2)} ]_{ext2}^{(3)} \chi_{H1}^{(3)} (1-\chi_{R1}^{(3)}) \Big]_{sym}^{(3)} \nonumber \\
&\hspace{0.5cm} = [ q_2^{(3)}  \chi_{H1}^{(3)}  ]_{sym}^{(3)} \chi_{NR3}^{(3)}+ [ q_3^{(3)} \chi_{H1}^{(3)} ]_{sym}^{(3)} \chi_{NR3}^{(3)}.   
\end{align}
\end{lem}
\begin{proof}
Put $M:=\Big[ \frac{ Q_1^{(2)}}{ \Phi_0^{(2)} } \chi_{NR1}^{(3)} \Big]_{ext1}^{(3)} [ Q_1^{(2)} \chi_{NR1}^{(2)} ]_{ext2}^{(3)} \chi_{H1}^{(3)} (1-\chi_{R1}^{(3)})$. 
Then, by Remark~\ref{rem_sym}, we have 
\begin{align*}
M = 4\, 
\big[ Q_1^{(2)}/\Phi_0^{(2)}\big]_{ext1}^{(3)} [ Q_1^{(2)}]_{ext2}^{(3)}\chi_{NR2}^{(3)} \chi_{H1}^{(3)} (1-\chi_{R1}^{(3)}).
\end{align*}
Since $\chi_{NR2}^{(3)} (1-\chi_{R1}^{(3)} ) = \chi_{NR3}^{(3)} $ for $(k_1, k_2, k_3) \in \supp \chi_{H1}^{(3)}$, it follows that
\begin{align*}
\chi_{NR2}^{(3)} \chi_{H1}^{(3)} (1-\chi_{R1}^{(3)})=\chi_{NR3}^{(3)} \chi_{H1}^{(3)}.
\end{align*}
By Remark~\ref{rem_mNR}, $ \chi_{NR3}^{(3)} \chi_{H1}^{(3)} (k_1, k_2, k_3)$ is symmetric with $(k_1, k_2)$. 
Thus, by Lemma~\ref{lem_sym} with $m^{(3)}=\chi_{NR3}^{(3)} \chi_{H1}^{(3)} $, we have
\begin{align*}
\ti{M}=[ q_2^{(3)}  \chi_{H1}^{(3)}  \chi_{NR3}^{(3)} ]_{sym}^{(3)}+ [q_3^{(3)} \chi_{H1}^{(3)} \chi_{NR3}^{(3)}  ]
= [ q_2^{(3)} \chi_{H1}^{(3)} ]_{sym}^{(3)}  \chi_{NR3}^{(3)}+ [q_3^{(3)} \chi_{H1}^{(3)} ]_{sym}^{(3)}  \chi_{NR3}^{(3)} . 
\end{align*}
Here we used that $\chi_{NR3}^{(3)}$ is symmetric in the second equality. 
\end{proof}

\begin{lem}\label{L31}
\begin{align} 
&\partial_t \Lambda_{\vp}^{(2)} (  \ti{L}_{1, \vp}^{(2)} \chi_{>L}^{(2)}, \ha{v} (t) ) (t,k)
=\Lambda_{\vp}^{(2)} (  -\ti{L}_{1, \vp}^{(2)} \Phi_{\vp}^{(2)}  \chi_{>L}^{(2)}, \ha{v} (t) ) (t,k) \notag \\
&+ \Lambda_{\vp}^{(3)} \big( \sum_{i=3}^{14} \ti{L}_{i ,\vp}^{(3)} \Phi_{\vp}^{(3)}  \chi_{ > L}^{(3)} + 
\sum_{i=3}^{14} \ti{L}_{i, \vp}^{(3)} \Phi_{\vp}^{(3)} \chi_{ \le L}^{(3)} +
\sum_{i=4}^{11} \ti{M}_{i, \vp}^{(3)}, \ha{v} (t) \big) (t,k) \notag \\ 
& \hspace{0.5cm} + \Lambda_{\vp}^{(4)} 
\big( \big[ [ 2 \ti{L}_{1, \vp}^{(2)}  \chi_{>L}^{(2)} ]_{est1}^{(4)} [-Q^{(3)}]_{ext2}^{(4)} \big]_{sym}^{(4)}, \ha{v} (t) \big) (t,k). 
\label{eq36}
\end{align}
\end{lem}
\begin{proof}
By Proposition \ref{prop_NF11} with $N=2$ and $\ti{m}^{(N)}=\ti{L}_{1, \vp}^{(2)} \chi_{>L}^{(2)}$, we have
\begin{align*} 
&\partial_t \Lambda_{\vp}^{(2)} (  \ti{L}_{1, \vp}^{(2)} \chi_{>L}^{(2)}, \ha{v} (t) ) (t,k)
=\Lambda_{\vp}^{(2)} (  -\ti{L}_{1, \vp}^{(2)} \Phi_{\vp}^{(2)}  \chi_{>L}^{(2)}, \ha{v} (t) ) (t,k)\\
&+ \Lambda_{\vp}^{(3)} \big( \big[ [2 \ti{L}_{1, \vp}^{(2)} \chi_{>L}^{(2)}  ]_{ext1}^{(3)} [ -Q^{(2)} \chi_{NR1}^{(2)} ]_{ext2}^{(3)} \big]_{sym}^{(3)}, \ha{v} (t) \big) (t,k) \nonumber \\ 
& \hspace{0.5cm} + \Lambda_{\vp}^{(4)} 
\big( \big[ [ 2 \ti{L}_{1, \vp}^{(2)}  \chi_{>L}^{(2)} ]_{est1}^{(4)} [-Q^{(3)}]_{ext2}^{(4)} \big]_{sym}^{(4)}, \ha{v} (t) \big) (t,k).
\end{align*}
Thus, we only need to show
\begin{align}
\big[ [2 \ti{L}_{1, \vp}^{(2)} \chi_{>L}^{(2)}  ]_{ext1}^{(3)} [ -Q^{(2)} \chi_{NR1}^{(2)} ]_{ext2}^{(3)} \big]_{sym}^{(3)} = \sum_{i=3}^{14} \ti{L}_{i, \vp}^{(3)} \Phi_{\vp}^{(3)}+  \sum_{i=4}^{11} \ti{M}_{i, \vp}^{(3)}, \label{eq38}
\end{align}
By \eqref{le211} with $m_1^{(2)} =Q^{(2)} \chi_{NR1}^{(2)} \chi_{>L}^{(2)} / \Phi_{\vp}^{(2)}$ and $m_2^{(2)} = Q^{(2)} \chi_{NR1}^{(2)}$, we have
\begin{align}
& \big[  [2 \ti{L}_{1,\vp}^{(2)} \chi_{>L}^{(2)} ]_{ext1}^{(3)} [-Q^{(2)} \chi_{NR1}^{(2)} ]_{ext2}^{(3)} \big]_{sym}^{(3)}\notag\\
& = 2 \Big[ \Big[ \frac{Q^{(2)} }{ \Phi_{\vp}^{(2)}} \chi_{NR1}^{(2)} \chi_{>L}^{(2)} [ 2\chi_{H1}^{(2)} ]_{sym}^{(2)}  \Big]_{ext1}^{(3)} [Q^{(2)} \chi_{NR1}^{(2)} ]_{ext2}^{(3)} \Big]_{sym}^{(3)}  \nonumber \\
& =4 \Big[ \Big[ \frac{Q^{(2)} }{\Phi_{\vp}^{(2)} } \chi_{NR1}^{(2)} \chi_{>L}^{(2)}  \Big]_{ext1}^{(3)} [Q^{(2)} \chi_{NR1}^{(2)}  ]_{ext2}^{(3)} \, \chi_{NR(1,1)}^{(3)}  \Big]_{sym}^{(3)}+ \ti{L}_{3, \vp}^{(3)} \Phi_{\vp}^{(3)}\notag\\
&+2 \Big[ \Big[ \frac{Q^{(2)} }{\Phi_{\vp}^{(2)} } \chi_{NR1}^{(2)} \chi_{>L}^{(2)}  \Big]_{ext1}^{(3)} [Q^{(2)} \chi_{NR1}^{(2)}  ]_{ext2}^{(3)} \, \chi_{NR(2,1)}^{(3)}  \Big]_{sym}^{(3)}\label{eqn311}\\
&+2 \Big[ \Big[ \frac{Q^{(2)} }{\Phi_{\vp}^{(2)} } \chi_{NR1}^{(2)} \chi_{>L}^{(2)}  \Big]_{ext1}^{(3)} [Q^{(2)} \chi_{NR1}^{(2)}  ]_{ext2}^{(3)} \, \chi_{NR(2,2)}^{(3)}  \Big]_{sym}^{(3)}\label{eqn312}
\end{align}
Since
\[
\chi_{NR(2,1)}^{(3)}=\chi_{NR(2,1)}^{(3)} \chi_{A3}^{(3)}+ \chi_{NR(2,1)}^{(3)} (1-\chi_{A3}^{(3)}) \chi_{NR1}^{(3)}+ \chi_{NR(2,1)}^{(3)} (1-\chi_{A3}^{(3)}) (1-\chi_{NR1}^{(3)}),
\]
\eqref{eqn311} is equal to
\[
\ti{L}_{11, \vp}^{(3)} \Phi_{\vp}^{(3)} + \ti{L}_{12, \vp}^{(3)} \Phi_{\vp}^{(3)} +\ti{M}_{10, \vp}^{(3)}.
\]
Since
\[
\chi_{NR(2,2)}^{(3)}=\chi_{NR(2,2)}^{(3)} \chi_{A2}^{(3)}  + \chi_{NR(2,2)}^{(3)} (1- \chi_{A2}^{(3)}) \chi_{NR1}^{(3)}+ \chi_{NR(2,2)}^{(3)} (1- \chi_{A2}^{(3)}) (1- \chi_{NR1}^{(3)}),
\]
 \eqref{eqn312} is equal to
\[
\ti{L}_{13, \vp}^{(3)} \Phi_{\vp}^{(3)} + \ti{L}_{14, \vp}^{(3)} \Phi_{\vp}^{(3)} +\ti{M}_{11, \vp}^{(3)}.
\]
Therefore, we only need to show
\begin{align}
4 \Big[\Big[ \frac{Q^{(2)}}{  \Phi_{\vp}^{(2)} } \chi_{NR1}^{(2)} \chi_{>L}^{(2)} \Big]_{ext1}^{(3)} [Q^{(2)} \chi_{NR1}^{(2)}]_{ext2}^{(3)} \chi_{NR(1,1)}^{(3)}  \Big]_{sym}^{(3)}
= \sum_{i=4}^{10} \ti{L}_{i, \vp}^{(3)} \Phi_{\vp}^{(3)}+ \sum_{i=4}^{9} \ti{M}_{i, \vp}^{(3)}.\label{eq39}
\end{align}
By the definition,
\begin{align} \label{eqn33}
& 4 \Big[ \Big[ \frac{Q^{(2)}}{  \Phi_{\vp}^{(2)} } \chi_{NR1}^{(2)} \chi_{>L}^{(2)} \Big]_{ext1}^{(3)} 
[Q^{(2)} \chi_{NR1}^{(2)}]_{ext2}^{(3)} \chi_{NR(1,1)}^{(3)} \Big]_{sym}^{(3)} \nonumber \\
& = \sum_{i=8}^{9} \ti{M}_{i, \vp}^{(3)} +
4 \Big[ \Big[ \frac{Q^{(2)}}{  \Phi_{0}^{(2)} } \chi_{NR1}^{(2)} \Big]_{ext1}^{(3)} [Q^{(2)} \chi_{NR1}^{(2)}]_{ext2}^{(3)} \chi_{NR(1,1)}^{(3)} \Big]_{sym}^{(3)}.
\end{align}
Since 
$\chi_{NR(1,1)}^{(3)}= \chi_{NR(1,1)}^{(3)} \chi_{R1}^{(3)}+\chi_{NR(1,1)}^{(3)} (1-\chi_{R1}^{(3)})$ and 
$Q^{(2)}=Q_1^{(2)}+ Q_2^{(2)}$, we get
\begin{align} \label{eqn321}
& 4 \Big[ \Big[ \frac{Q^{(2)}}{  \Phi_{0}^{(2)} } \chi_{NR1}^{(2)} \Big]_{ext1}^{(3)} [Q^{(2)} \chi_{NR1}^{(2)}]_{ext2}^{(3)} m_{NR(1,1)}^{(3)} \Big]_{sym}^{(3)} \nonumber \\
& \hspace{0.3cm} = \sum_{i=8}^{10} \ti{L}_{i, \vp}^{(3)} \Phi_{\vp}^{(3)}+ \sum_{i=6}^{7} \ti{M}_{i, \vp}^{(3)}
+ 4 \Big[ \Big[ \frac{Q_1^{(2)}}{  \Phi_{0}^{(2)} } m_{NR1}^{(3)} \Big]_{ext1}^{(3)} 
[Q_1^{(2)} \chi_{NR1}^{(2)}]_{ext2}^{(3)} \chi_{NR(1,1)}^{(3)}  \Big]_{sym}^{(3)}.
\end{align}
Since $\supp \chi_{H1}^{(3)} \subset \supp \chi_{NR(1,1)}^{(3)}$,
\begin{align*}
\chi_{NR(1,1)}^{(3)}&=\chi_{NR(1,1)}^{(3)}(1-\chi_{H1}^{(3)})+ \chi_{NR(1,1)}^{(3)}\chi_{H1}^{(3)}=\chi_{NR(1,1)}^{(3)}(1-\chi_{H1}^{(3)})+ \chi_{H1}^{(3)} \\
& = \chi_{NR(1,1)}^{(3)}(1-\chi_{H1}^{(3)}) \chi_{A1}^{(3)} +\chi_{NR(1,1)}^{(3)}(1-\chi_{H1}^{(3)}) (1-\chi_{A1}^{(3)}) \chi_{NR1}^{(3)}  \\
&+ \chi_{NR(1,1)}^{(3)}(1-\chi_{H1}^{(3)})(1-\chi_{A1}^{(3)})(1- \chi_{NR1}^{(3)})+\chi_{H1}^{(3)} \chi_{R1}^{(3)} + \chi_{H1}^{(3)} (1-\chi_{R1}^{(3)}).
\end{align*}
Thus, we have
\begin{align} \label{eqn32}
&4 \big[ \big[ \frac{Q_1^{(2)}}{  \Phi_{0}^{(2)} } \chi_{NR1}^{(2)} \big]_{ext1}^{(3)} 
[Q_1^{(2)} \chi_{NR1}^{(2)}]_{et2}^{(3)} \chi_{NR(1,1)}^{(3)} \big]_{sym}^{(3)}
= \sum_{i=6}^7 \ti{L}_{i, \vp}^{(3)} \Phi_{\vp}^{(3)}+ \sum_{i=4}^{5} \ti{M}_{i, \vp}^{(3)} \nonumber \\
& \hspace{1.5cm} +
4 \big[ \big[ \frac{Q_1^{(2)}}{  \Phi_{0}^{(2)} } \chi_{NR1}^{(2)} \big]_{ext1}^{(3)}
[Q_1^{(2)} \chi_{NR1}^{(2)}]_{ext2}^{(3)} \chi_{H1}^{(3)} (1- \chi_{R1}^{(3)}) \big]_{sym}^{(3)}.
\end{align}
By Lemma~\ref{L30}, it follows that 
\begin{align}  \label{eqn322}
4 \big[ \big[ \frac{Q_1^{(2)}}{  \Phi_{0}^{(2)} } \chi_{NR1}^{(2)} \big]_{ext1}^{(3)}
[Q_1^{(2)} \chi_{NR1}^{(2)}]_{ext2}^{(3)} \chi_{H1}^{(3)} (1- \chi_{R1}^{(3)}) \big]_{sym}^{(3)}
= \ti{L}_{4, \vp}^{(3)} \Phi_{\vp}^{(3)}+ \ti{L}_{5, \vp}^{(3)} \Phi_{\vp}^{(3)}. 
\end{align}
Collecting \eqref{eqn33}--\eqref{eqn322}, we obtain \eqref{eq39}
\end{proof}

\begin{lem}\label{L32}
\begin{align} \label{eq46}
&\partial_t \Lambda_{\vp}^{(3)} (\ti{L}_{4, \vp}^{(3)} \chi_{>L}^{(3)}, \ha{v} (t) ) (t,k)
=\Lambda_{\vp}^{(3)} (-\ti{L}_{4, \vp}^{(3)} \Phi_{\vp}^{(3)} \chi_{>L}^{(3)} , \ha{v} (t) ) (t,k) \nonumber \\
&+  \Lambda_{\vp}^{(4)} \Big( \sum_{i=1}^4 \ti{L}_{i, \vp}^{(4)} \Phi_{\vp}^{(4)} \chi_{> L}^{(4)}+  
\sum_{i=1}^4 \ti{L}_{i, \vp}^{(4)} \Phi_{\vp}^{(4)} \chi_{\le L}^{(4)}
+ \sum_{i=4}^{11} \ti{M}_{i, \vp}^{(4)} , \ha{v} (t) \Big) (t,k) \nonumber \\
& +\Lambda_{\vp}^{(5)} (\big[ [ 3 \ti{L}_{4, \vp}^{(3)} \chi_{>L}^{(3)}]_{ext1}^{(5)} [-Q^{(3)}]_{ext2}^{(5)} \big]_{sym}^{(5)}, \ha{v} (t) ) (t,k).
\end{align}
\end{lem}

\begin{proof}
By Proposition~\ref{prop_NF11} with $N=3$ and $\ti{m}^{(N)}=\ti{L}_{4, \vp}^{(3)} \chi_{>L}^{(3)}$, we have
\begin{align*}
&\partial_t \Lambda_{\vp}^{(3)} (\ti{L}_{4, \vp}^{(3)} \chi_{>L}^{(3)}, \ha{v} (t) ) (t,k)
=\Lambda_{\vp}^{(3)} (-\ti{L}_{4, \vp}^{(3)} \Phi_{\vp}^{(3)} \chi_{>L}^{(3)} , \ha{v} (t) ) (t,k)\\
&+  \Lambda_{\vp}^{(4)} \big( \big[ [3 \ti{L}_{4, \vp}^{(3)} \chi_{>L}^{(3)} ]_{ext1}^{(4)} [-Q^{(2)} \chi_{NR1}^{(2)} ]_{ext2}^{(4)} 
 \big]_{sym}^{(4)}, \ha{v}(t) \big) (t,k)\\
& +\Lambda_{\vp}^{(5)} (\big[ [ 3 \ti{L}_{4, \vp}^{(3)} \chi_{>L}^{(3)}]_{ext1}^{(5)} [-Q^{(3)}]_{ext2}^{(5)} \big]_{sym}^{(5)}, \ha{v} (t) ) (t,k).
\end{align*}
Thus, we only need to show
\begin{align}\label{eq49}
 \big[ [3 \ti{L}_{4, \vp}^{(3)} \chi_{>L}^{(3)} ]_{ext1}^{(4)} [-Q^{(2)} \chi_{NR1}^{(2)} ]_{ext2}^{(4)}  \big]_{sym}^{(4)}
= \sum_{i=1}^4 \ti{L}_{i, \vp}^{(4)} \Phi_{\vp}^{(4)} + \sum_{i=4}^{11} \ti{M}_{i, \vp}^{(4)}
\end{align}
By the definition, the left-hand side of (\ref{eq49}) is equal to
\[
-\Big[ \Big[ \frac{\chi_{NR3}^{(3)} \chi_{>L}^{(3)}}{ \Phi_{\vp}^{(3)} } \, [ 3q_2^{(3)} \chi_{H1}^{(3)}  ]_{sym}^{(3)} \Big]_{ext1}^{(4)} 
[Q^{(2)} \chi_{NR1}^{(2)}   ]_{ext2}^{(4)} \Big]_{sym}^{(4)},
\]
which is equal to
\begin{align}
& \ti{L}_{4, \vp}^{(4)} \Phi_{\vp}^{(4)}+  \ti{M}_{4, \vp}^{(4)} + \ti{M}_{11, \vp}^{(4)} \notag \\
&+ (-2) \Big[ \Big[ \frac{q_2^{(3)}}{ \Phi_{\vp}^{(3)} } \chi_{NR3}^{(3)} \chi_{>L}^{(3)} \Big]_{ext1}^{(4)} 
 [Q^{(2)} \chi_{NR1}^{(2)} ]_{ext2}^{(4)} \chi_{NR(1,1)}^{(4)} \Big]_{sym}^{(4)}  \label{eqn421}
\end{align}
by \eqref{le212} with $m_1^{(3)} = \chi_{NR3}^{(3)} \chi_{>L}^{(3)} / \Phi_{\vp}^{(3)} $, $m_2^{(2)} = Q^{(2)} \chi_{NR1}^{(2)} $, $m_3^{(3)} =q_2^{(3)}$.
Therefore, we only need to show that 
\begin{align} \label{eq42}
(-2) \Big[ \Big[ \frac{q_2^{(3)}}{\Phi_{\vp}^{(3)}} \chi_{NR3}^{(3)} \chi_{>L}^{(3)} \Big]_{ext1}^{(4)} 
[ Q^{(2)} \chi_{NR1}^{(2)}  ]_{ext2}^{(4)} \chi_{NR(1,1)}^{(4)}  \Big]_{sym}^{(4)}=
\sum_{i=1}^{3} \ti{L}_{i,\vp}^{(4)} \Phi_{\vp}^{(4)}  
+ \sum_{i=5}^{10} \ti{M}_{i, \vp}^{(4)}.
\end{align} 
By the definition, the left hand side of (\ref{eq42}) is equal to 
\begin{align} \label{eqn422}
\sum_{i=9}^{10} \ti{M}_{i, \vp}^{(4)}+
(-2) \Big[ \Big[ \frac{q_2^{(3)}}{\Phi_{0}^{(3)}} \chi_{NR3}^{(3)} \Big]_{ext1}^{(4)} 
[ Q^{(2)} \chi_{NR1}^{(2)}  ]_{ext2}^{(4)} \chi_{NR(1,1)}^{(4)}  \Big]_{sym}^{(4)}. 
\end{align}
Since $\chi_{NR(1,1)}^{(4)} = \chi_{NR(1,1)}^{(4)} \chi_{A4}^{(4)} + \chi_{NR(1,1)}^{(4)} (1-\chi_{A4}^{(4)} ) $ and $Q^{(2)} =Q_1^{(2)} +Q_2^{(2)} $, 
we have 
\begin{equation} \label{eqn423}
\begin{split}
& (-2) \Big[ \Big[ \frac{q_2^{(3)}}{\Phi_{0}^{(3)}} \chi_{NR3}^{(3)} \Big]_{ext1}^{(4)} 
[ Q^{(2)} \chi_{NR1}^{(2)}  ]_{ext2}^{(4)} \chi_{NR(1,1)}^{(4)}  \Big]_{sym}^{(4)}\\
& = \ti{L}_{3, \vp}^{(4)} \Phi_{\vp}^{(4)} + \ti{M}_{8, \vp}^{(4)}+  (-2) \Big[ \Big[ \frac{q_2^{(3)}}{\Phi_{0}^{(3)}} \chi_{NR3}^{(3)} \Big]_{ext1}^{(4)} 
[ Q_1^{(2)} \chi_{NR1}^{(2)}  ]_{ext2}^{(4)} \chi_{NR(1,1)}^{(4)}  \Big]_{sym}^{(4)}. 
\end{split}
\end{equation} 
Since $\supp \, \chi_{H1}^{(4)} \subset \supp \, \chi_{NR(1,1)}^{(4)}$, 
\begin{align*}
\chi_{NR(1,1)}^{(4)}= 
& \chi_{NR(1,1)}^{(4)} (1-\chi_{H1}^{(4)})+ \chi_{NR(1,1)}^{(4)} \chi_{H1}^{(4)}=\chi_{NR(1,1)}^{(4)}(1-\chi_{H1}^{(4)})+\chi_{H1}^{(4)}   \\ 
= &\chi_{NR(1,1)}^{(4)}(1-\chi_{H1}^{(4)}) \chi_{A4}^{(4)}+ \chi_{NR(1,1)}^{(4)}(1-\chi_{H1}^{(4)})(1-\chi_{A4}^{(4)}) \\
+& \chi_{H1}^{(4)} \chi_{R1}^{(4)} + \chi_{H1}^{(4)} (1-\chi_{R1}^{(4)}) \chi_{R5}^{(4)}+ \chi_{H1}^{(4)} (1-\chi_{R1}^{(4)})(1- \chi_{R5}^{(4)}). 
\end{align*}
Therefore, we  have 
\begin{align} \label{eqn424}
 (-2) \Big[ \Big[ \frac{q_2^{(3)}}{\Phi_{0}^{(3)}} \chi_{NR3}^{(3)} \Big]_{ext1}^{(4)} 
[ Q_1^{(2)} \chi_{NR1}^{(2)}  ]_{ext2}^{(4)} \chi_{NR(1,1)}^{(4)}  \Big]_{sym}^{(4)}
= \sum_{i=1}^2 \ti{L}_{i, \vp}^{(4)} \Phi_{\vp}^{(4)} + \sum_{i=5}^{7} \ti{M}_{i, \vp}^{(4)}. 
\end{align}
Collecting (\ref{eqn422})--(\ref{eqn424}), we obtain (\ref{eq42}). 
\end{proof}

Now, we prove Proposition \ref{prop_NF2}.
\begin{proof}[Proof of Proposition \ref{prop_NF2}]
By direct computation, it follows that
\begin{align*}
 & -10 \ga q_1^{(3)} \chi_{NR1}^{(3)}  =  -10 \ga q_1^{(3)} \chi_{NR1}^{(3)}[3\chi_{H1}^{(3)}]_{sym}^{(3)} 
-10 \ga q_1^{(3)} \chi_{NR1}^{(3)} (1-[3\chi_{H1}^{(3)} ]_{sym}^{(3)}) \nonumber \\
& \hspace{2.8cm} =  \sum_{i=1}^2 \ti{L}_{i, \vp}^{(3)} \Phi_{\vp}^{(3)} \chi_{> L}^{(3)} 
+ \sum_{i=1}^2 \ti{L}_{i,\vp}^{(3)} \Phi_{\vp}^{(3)} \chi_{\le L}^{(3)},\\
& -\frac{\beta^2}{5} q_1^{(3)} [ \chi_{R1}^{(3)} (1- \chi_{R2}^{(3)}) ]_{sym}^{(3)}
= \ti{M}_{2, \vp}^{(3)}+ \ti{M}_{3, \vp}^{(3)}. 
\end{align*}
By $[2\chi_{H1}^{(2)}]_{sym}^{(2)} +\chi_{H2}^{(2)} = 1 $, we have
\begin{align*}
-Q^{(2)} \chi_{NR1}^{(2)}
& = -Q^{(2)} \chi_{NR1}^{(2)} [2 \chi_{H1}^{(2)}]_{sym}^{(2)} -Q^{(2)} \chi_{NR1}^{(2)} \chi_{H2}^{(2)} \nonumber \\
&=  \sum_{i=1}^2 \ti{L}_{i, \vp}^{(2)} \Phi_{\vp}^{(2)} \chi_{> L}^{(2)}
+ \sum_{i=1}^2 \ti{L}_{i, \vp}^{(2)} \Phi_{\vp}^{(2)} \chi_{\le L}^{(2)}.
\end{align*}
Therefore, by Proposition \ref{prop_req1}, we have
\begin{equation}\label{eq370}
\begin{split}
&\p_t \ha{v} (t,k) = \La_{\vp}^{(2)}\left(\sum_{i=1}^2 \ti{L}_{i, \vp}^{(2)} \Phi_{\vp}^{(2)} \chi_{> L}^{(2)}
+ \sum_{i=1}^2 \ti{L}_{i, \vp}^{(2)} \Phi_{\vp}^{(2)} \chi_{\le L}^{(2)} , \ha{v}(t) \right)(t,k)\\
& + \La_{\vp}^{(3)}\left( \sum_{i=1}^2 \ti{L}_{i, \vp}^{(3)} \Phi_{\vp}^{(3)} \chi_{> L}^{(3)} 
+ \sum_{i=1}^2 \ti{L}_{i,\vp}^{(3)} \Phi_{\vp}^{(3)} \chi_{\le L}^{(3)}+ \sum_{i=1}^3\ti{M}_{i, \vp}^{(3)}, \ha{v} (t) \right)(t,k)
\end{split}
\end{equation}
Since
\begin{align*}
\big[ [2 \ti{L}_{2, \vp}^{(2)}  \chi_{>L}^{(2)} ]_{ext1}^{(3)} [-Q^{(2)} \chi_{NR1}^{(2)}  ]_{ext2}^{(3)} \big]_{sym}^{(3)}
=\ti{L}_{15, \vp}^{(3)} \Phi_{\vp}^{(3)} \chi_{ > L}^{(3)} + 
\ti{L}_{15, \vp}^{(3)} \Phi_{\vp}^{(3)} \chi_{ \le L}^{(3)} + \ti{M}_{12, \vp}^{(3)},
\end{align*}
by Proposition~\ref{prop_NF11} with $N=2$ and $\ti{m}^{(N)}=\ti{L}_{2, \vp}^{(2)} \chi_{>L}^{(2)}$, we have
\begin{equation} \label{eq37}
\begin{split}
&\partial_t \Lambda_{\vp}^{(2)} (\ti{L}_{2, \vp}^{(2)} \chi_{>L}^{(2)}, \ha{v} (t) ) (t,k)
=\Lambda_{\vp}^{(2)} (-\ti{L}_{2, \vp}^{(2)} \Phi_{\vp}^{(2)} \chi_{>L}^{(2)} , \ha{v}(t) ) (t,k)\\
&+ \Lambda_{\vp}^{(3)} 
\big( \ti{L}_{15, \vp}^{(3)} \Phi_{\vp}^{(3)} \chi_{> L}^{(3)}+  \ti{L}_{15, \vp}^{(3)} \Phi_{\vp}^{(3)} \chi_{\le L}^{(3)}  +\ti{M}_{12, \vp}^{(3)}, \ha{v}(t) \big) (t,k) \\
& + \Lambda_{\vp}^{(4)}(\big[ [ 2\ti{L}_{2, \vp}^{(2)} \chi_{>L}^{(2)}]_{ext1}^{(4)} [-Q^{(3)}]_{ext2}^{(4)} \big]_{sym}^{(4)}, \ha{v} (t)) (t,k).
\end{split}
\end{equation}
By Proposition~\ref{prop_NF11} with $N=3$ and 
$\ti{m}^{(N)}=\big( \sum_{i=1}^{3} \ti{L}_{i, \vp}^{(3)} + \sum_{i=5}^{15} \ti{L}_{i, \vp}^{(3)} \big) \chi_{>L}^{(3)}$, we have
\begin{equation} \label{eq31}
\begin{split}
&\partial_t \Lambda_{\vp}^{(3)}  \Big(  \Big( \sum_{i=1}^{3} \ti{L}_{i, \vp}^{(3)} 
+ \sum_{i=5}^{15} \ti{L}_{i, \vp}^{(3)} \Big) \chi_{>L}^{(3)}, \ha{v}(t) \Big) (t,k) \\
& =\Lambda_{\vp}^{(3)} 
\Big(  \Big( -\sum_{i=1}^{3} \ti{L}_{i, \vp}^{(3)} - \sum_{i=5}^{15} \ti{L}_{i, \vp}^{(3)} \Big) \Phi_{\vp}^{(3)} \chi_{>L}^{(3)} , \ha{v}(t) \Big) (t,k)
+  \Lambda_{\vp}^{(4)} (\ti{M}_{2, \vp}^{(4)} + \ti{M}_{3, \vp}^{(4)} , \ha{v}(t) ) (t,k) \\
& + \Lambda_{\vp}^{(5)} 
\Big(\Big[ [ 3 \Big( \sum_{i=1}^{3} \ti{L}_{i, \vp}^{(3)} + \sum_{i=5}^{15} \ti{L}_{i, \vp}^{(3)} \Big) \Phi_{\vp}^{(3)} \chi_{>L}^{(3)}\Big]_{ext1}^{(5)} [-Q^{(3)}]_{ext2}^{(5)} \Big]_{sym}^{(5)}, 
\ha{v}(t) \Big) (t,k).
\end{split}
\end{equation}
By Proposition~\ref{prop_NF11} with $N=4$ and $\ti{m}^{(N)}=\sum_{i=1}^4 \ti{L}_{i, \vp}^{(4)} \chi_{>L}^{(4)}$, we have
\begin{equation} \label{eq51}
\begin{split}
&\partial_t \Lambda_{\vp}^{(4)} \Big( \sum_{i=1}^4 \ti{L}_{i, \vp}^{(4)} \chi_{>L}^{(4)}, \ha{v}(t) \Big) (t,k)
=\Lambda_{\vp}^{(4)} \Big( -\sum_{i=1}^4 \ti{L}_{i, \vp}^{(4)} \Phi_{\vp}^{(4)} \chi_{>L}^{(4)} , \ha{v}(t)  \Big) (t,k)\\
&+ \Lambda_{\vp}^{(5)} (\ti{M}_{2, \vp}^{(5)}, \ha{v}(t) ) (t,k)+
 \Lambda_{\vp}^{(6)} (\ti{M}_{1, \vp}^{(6)}, \ha{v}(t) ) (t,k).
\end{split}
\end{equation}
By \eqref{eq370}--\eqref{eq51}, Lemmas \ref{L31} and \ref{L32}, we conclude \eqref{NF21}.
\end{proof}

\section{cancellation properties}

In Lemma \ref{lem_pwb2} and Lemmas \ref{lem_nl10}--\ref{lem_nl12}, we show that all multipliers $\{ M_{j, \vp}^{(N)} \}$ except 
for $M_{2,\vp}^{(3)}$, $M_{4,\vp}^{(3)}$, $M_{5,\vp}^{(4)}$ and $M_{6,\vp}^{(3)}$ have no derivative loss.
As we explain below, there are some difficulties to estimate $M_{2,\vp}^{(3)}$, $M_{4,\vp}^{(3)}$, $M_{5,\vp}^{(4)}$ and $M_{6,\vp}^{(3)}$
and the normal form reduction does not work to overcome the difficulties since these are resonant parts.
Therefore, we use a kinds of cancellation properties.

(i) $M_{4,\vp}^{(3)}$ has two derivative losses and $M_{2,\vp}^{(3)}$ has one derivative loss.
That is, for $(k_1,k_2,k_3)\in \supp \chi_{NR2}^{(3)}\chi_{H1}^{(3)} \chi_{R1}^{(3)}$, it follows that $16 \max\{ |k_1|, |k_2|\} <|k_3|$ and
\[
|M_{4,\vp}^{(3)}| \sim |k_1|^{-1}|k_3|^2, 
\]
for $(k_1, k_2, k_3) \in \supp \chi_{H1}^{(3)} \chi_{R1}^{(3)} (1-\chi_{R2}^{(3)} )$, it follows that $16 \max\{ |k_1|, |k_2| \} < |k_3|$ and 
\[
|M_{2, \vp}^{(3)}| \sim |k_3|.
\]
In Proposition \ref{prop_res1}, we compute the sum of the symmetrization of these two multipliers and show it is equal to $0$.

(ii) $M_{5, \vp}^{(4)}$ has one derivative loss.
That is, for $(k_1,k_2,k_3,k_4) \in \supp \chi_{NR1}^{(4)} \chi_{H1}^{(4)} \chi_{R1}^{(4)}$, it follows that $64 \max\{ |k_1|, |k_2|, |k_3| \} < |k_4|$ and 
\[
|M_{5,\vp}^{(4)}| \sim |k_1|^{-1} |k_2|^{-1} |k_4|. 
\]
In Proposition \ref{prop_res2}, we compute the symmetrization of it and show it has no derivative loss.

(iii) $M_{6,\vp}^{(3)}$ is the sum of the following two terms:
\begin{align*}
\Big[  \frac{Q_1^{(2)}}{ \Phi_0^{(2)} } \chi_{NR1}^{(2)}  \Big]_{ext1}^{(3)} [Q_2 ^{(2)}\chi_{NR1}^{(2)}  ]_{ext2}^{(3)}  \chi_{NR(1,1)}^{(3)} \chi_{R1}^{(3)},
\hspace{0.3cm}
\Big[  \frac{Q_2^{(2)}}{ \Phi_0^{(2)} } \chi_{NR1}^{(2)} \Big]_{ext1}^{(3)} [Q_1^{(2)} \chi_{NR1}^{(2)}]_{ext2}^{(3)} \chi_{NR(1,1)}^{(3)}
 \chi_{R1}^{(3)}
\end{align*} 
and each of them has one derivative loss.
That is,
for  $(k_1,k_2,k_3)\in\supp  \chi_{NR2}^{(3)} \chi_{NR(1,1)}^{(3)} \chi_{R1}^{(3)}$, it follows that 
$4|k_1|=4|k_2| < |k_3|$ and 
\[
\Big| \Big[  \frac{Q_1^{(2)}}{ \Phi_0^{(2)} }  \Big]_{ext1}^{(3)} [Q_2^{(2)}]_{ext2}^{(3)} \Big| \sim
\Big| \Big[  \frac{Q_2^{(2)}}{ \Phi_0^{(2)} }  \Big]_{ext1}^{(3)} [Q_1^{(2)}]_{ext2}^{(3)} \Big| \sim |k_3|. 
\]
In Proposition \ref{prop_res3}, we compute the sum of these two terms and show $M_{6,\vp}^{(3)}$ has no derivative loss.

\begin{prop} \label{prop_res1}
It follows that 
\begin{equation} \label{re11}
\ti{M}_{2,\vp}^{(3)}+ \ti{M}_{4, \vp}^{(3)}=0.
\end{equation}
\end{prop}
\begin{proof}
By Remark \ref{rem_sym}, we have
\begin{align*}
M_{4,\vp}^{(3)}  = 4\, 
\big[ Q_1^{(2)}/\Phi_0^{(2)}\big]_{ext1}^{(3)} [ Q_1^{(2)}]_{ext2}^{(3)}\chi_{NR2}^{(3)}\chi_{H1}^{(3)} \chi_{R1}^{(3)}
\end{align*}
For $(k_1,k_2,k_3) \in \supp \chi_{H1}^{(3)}$, it follows that $\chi_{NR2}^{(3)}=1-\chi_{R2}^{(3)}$, that is to say 
$$\chi_{NR2}^{(3)}\chi_{H1}^{(3)} \chi_{R1}^{(3)}= \chi_{H1}^{(3)} \chi_{R1}^{(3)} (1-\chi_{R2}^{(3)}). $$
Since $\chi_{H1}^{(3)} \chi_{R1}^{(3)} (1-\chi_{R2}^{(3)}) (k_1, k_2, k_3)  $ is symmetric with $(k_1, k_2)$, by Lemma~\ref{lem_sym}
with $m^{(3)}=\chi_{H1}^{(3)} \chi_{R1}^{(3)} (1-\chi_{R2}^{(3)})$, we get
\begin{equation*}
\ti{M}_{4, \vp}^{(3)} =
\big[ q_2^{(3)} \chi_{H1}^{(3)} \chi_{R1}^{(3)} (1-\chi_{R2}^{(3)}) \big]_{sym}^{(3)}
+\big[ q_3^{(3)} \chi_{H1}^{(3)} \chi_{R1}^{(3)} (1-\chi_{R2}^{(3)})  \big]_{sym}^{(3)} .
\end{equation*}
For $(k_1,k_2,k_3)\in \supp \chi_{R1}^{(3)}$, it follows that $k_{1,2}=0$, $q_2^{(3)}=0$ and 
$q_3^{(3)}=\frac{\beta^2}{5}q_1^{(3)}$.
Therefore, we obtain \eqref{re11}.
\end{proof}

\begin{prop} \label{prop_res2}
It follows that $|\ti{M}_{5, \vp}^{(4)}| \lesssim [ \chi_{H1}^{(4)} \chi_{R1}^{(4)}  ]_{sym}^{(4)}$. 
\end{prop}

\begin{proof}
By Remark~\ref{rem_sym}, we have
\begin{equation} \label{re30}
M_{5, \vp}^{(4)} =-2  \big[ \frac{q_2^{(3)}}{ \Phi_{0}^{(3)} } \big]_{ext1}^{(4)} [ Q_1^{(2)} ]_{ext2}^{(4)} 
\chi_{NR1}^{(4)} \chi_{H1}^{(4)} \chi_{R1}^{(4)}
\end{equation}
Put $M:= \chi_{H1}^{(4)} \chi_{R1}^{(4)} \chi_{NR1}^{(4)}$. 
For $(k_1,k_2,k_3,k_4)\in \supp \chi_{H1}^{(4)}\chi_{R1}^{(4)}$, it follows that 
\begin{equation*} 
\chi_{NR1}^{(4)} =
\begin{cases}
1, \ \text{ when } k_1k_2k_3k_{1,2}k_{2,3}k_{1,3} \neq 0\\
0, \ \text{ otherwise }.
\end{cases}
\end{equation*}
Therefore, $M(k_1, k_2, k_3, k_4)$ is symmetric with $(k_1, k_2)$,  $(k_2,k_3)$ and $(k_3,k_1)$, which leads
\begin{equation*}
M (k_1, k_2, k_3, k_4)= M (k_2, k_3, k_1, k_4) = M (k_3, k_1, k_2, k_4).
\end{equation*}
Thus, by \eqref{re30}, we have
\begin{align} \label{re34}
 \ti{M}_{5, \vp}^{(4)} &= -\frac{2}{3}
\Big[ \Big\{ \frac{q_2^{(3)}(k_1, k_2, k_{3,4})Q_1^{(2)}(k_3, k_4)  }{\Phi_0^{(3)}(k_1, k_2, k_{3,4})  } 
+\frac{q_2^{(3)}(k_2, k_3, k_{1,4})Q_1^{(2)}(k_1, k_4)  }{\Phi_0^{(3)}(k_2, k_3, k_{1,4})} \nonumber \\
& \hspace{3em}
+\frac{q_2^{(3)}(k_3, k_1, k_{2,4})Q_1^{(2)}(k_2, k_4)  }{\Phi_0^{(3)}(k_3, k_1, k_{2,4})} \Big\}
M (k_1, k_2, k_3, k_4)  \Big]_{sym}^{(4)}  \nonumber \\
&= -\frac{2}{3}
\Big[ \Big\{ \frac{L(k_1,k_2,k_3,k_4)+L(k_2,k_3,k_1,k_4)+L(k_3,k_1,k_2,k_4) }{\Phi_0^{(3)}(k_1, k_2, k_{3,4})\Phi_0^{(3)}(k_2, k_3, k_{1,4})\Phi_0^{(3)}(k_3, k_1, k_{2,4})} \Big\} \nonumber \\
& \hspace{3em} \times M (k_1, k_2, k_3, k_4)  \Big]_{sym}^{(4)}
\end{align}
where we put
$$L(k_1,k_2,k_3,k_4):=q_2^{(3)}(k_1, k_2, k_{3,4})Q_1^{(2)}(k_3, k_4)\Phi_0^{(3)}(k_2, k_3, k_{1,4})\Phi_0^{(3)}(k_3, k_1, k_{2,4}).$$
By direct computation and the definition,
\begin{align*}
&L(k_1,k_2,k_3,k_4)\\
=&\frac{5\beta^3}{32}  \frac{k_{1,2}}{k_1 k_2}
\Big\{(k_{3,4}^2+k_{1,2}k_{3,4}+k_1^2+k_1k_2+k_2^2 ) \, k_{3,4}(k_3^2+k_4^2+k_{3,4}^2)\\
&\times k_{2,3}k_{1,3,4}k_{1,2,4}(k_{2,3}^2+k_{1,3,4}^2+k_{1,2, 4}^2)k_{1,3}k_{1,2,4}k_{2,3,4}(k_{1,3}^2+k_{1,2,4}^2+k_{2,3,4}^2)
\Big\}\\
=&\frac{5\beta^3}{32}\frac{k_{2,3}k_{1,3}}{k_1k_2}\sum_{l=1}^{14} p_l(k_1,k_2,k_3)k_4^{14-l}
\end{align*}
where each $p_l(k_1,k_2,k_3)$ is a polynomial of degree $l$ for $l=1,2,\ldots,14$ and
\begin{align*}
p_1(k_1,k_2,k_3)=8k_{1,2}, \ \ p_2(k_1,k_2,k_3)=8k_{1,2} (7k_1+7k_2+ 8k_3).
\end{align*}
Especially, for $(k_1,k_2,k_3,k_4)\in \supp \chi^{(4)}_{R1}$, it follows that
$$ p_1(k_1,k_2,k_3)=- 8 k_{3}, \ \ p_2(k_1,k_2,k_3)=-8k_3^2.$$
Therefore, for $(k_1,k_2,k_3,k_4)\in \supp \chi^{(4)}_{R1}$, we obtain
\begin{align}\label{re35}
&L(k_1,k_2,k_3,k_4)+L(k_2,k_3,k_1,k_4)+L(k_3,k_1,k_2,k_4) \nonumber \\
&=\frac{5\beta^3}{32}\sum_{l=2}^{14} (p_l(k_1,k_2,k_3)+p_l(k_2,k_3,k_1)+p_l(k_3,k_1,k_2))k_4^{14-l}.
\end{align}
By \eqref{2eq2}, for $(k_1,k_2,k_3,k_4)\in \supp M $, it follows that
\begin{equation}\label{re36}
|\Phi_0^{(3)} (k_1, k_2, k_{3,4}) \Phi_0^{(3)} (k_2, k_3, k_{1,4}) \Phi_0^{(3)} (k_3, k_1, k_{2,4})|
  \gtrsim |k_1k_2k_3| |k_4|^{12}.
\end{equation}
Collecting \eqref{re34}--\eqref{re36}, we conclude $|\ti{M}_{5, \vp}^{(4)} | \lesssim [M]_{sym}^{(4)}  \le [ \chi_{H1}^{(4)} \chi_{R1}^{(4)}  ]_{sym}^{(4)}$.
\end{proof}

\begin{prop} \label{prop_res3} 
It follows that 
\begin{equation*}
| M_{6, \vp}^{(3)}  (k_1, k_2, k_3) | \lesssim |k_1| \, \chi_{R1}^{(3)}(k_1, k_2, k_3).
\end{equation*}
\end{prop}

\begin{proof}
Put $M:=\chi_{NR2}^{(3)} \chi_{NR(1,1)}^{(3)} \chi_{R1}^{(3)}$.
Then, by Remark \ref{rem_sym}, we have 
\EQQ{
 M_{6, \vp}^{(3)}
&= 4 \Big( \Big[  \frac{Q_1^{(2)}}{ \Phi_0^{(2)} }  \Big]_{ext1}^{(3)} [Q_2]_{ext2}^{(3)}+ 
\Big[ \frac{Q_2^{(2)}}{\Phi_0^{(2)}} \Big]_{ext1}^{(3)} [Q_1]_{ext2}^{(3)} \Big) M  \notag \\
&= \beta \Big(\al - \frac{\be}{2} \Big)
\frac{ [q_1^{(2)}]_{ext1}^{(3)} [q_2^{(2)}]_{ext2}^{(3)} + [q_2^{(2)}]_{ext1}^{(3)} [q_1^{(2)}]_{ext2}^{(3)} }{[ \Phi_0^{(2)} ]_{ext1}^{(3)}} 
\, M. 
}
By a direct computation, we have
\EQQ{
& ([q_1^{(2)}]_{ext1}^{(3)} [q_2^{(2)}]_{ext2}^{(3)} + [q_2^{(2)}]_{ext1}^{(3)} [q_1^{(2)}]_{ext2}^{(3)})(k_1, k_2, k_3)\\
& = -2 k_{1,2,3} k_{2,3} 
\{  (k_1+k_2) k_3^3+(3k_1+ 2k_2)k_2 k_3^2 + (k_1^2+3 k_1k_2+ k_2^2) k_2 k_3+ k_1 k_2^3   \}.
}
For $(k_1, k_2, k_3)\in \supp \chi_{R1}^{(3)}$, it follows that $k_{1,2} = 0$.
Thus,
\EQQ{
M_{6, \vp}^{(3)} = -2 \beta \Big( \al- \frac{\be}{2}  \Big) \, 
\frac{k_{1} k_2 k_3 k_{2,3} (k_3^2+ k_2 k_3 +k_2^2)  }{ \Phi_0^{(2)} (k_1, k_{2,3}) } \, M(k_1, k_2, k_3)
}
It follows that $k_1 k_2 k_{2,3}k_{1,2,3}\neq 0$ for $(k_1,k_2,k_3) \in \supp \, \chi_{NR2}^{(3)}$ and
it follows that $|k_1|=|k_2| \lesssim |k_3|\sim |k_{2,3}| \sim |k_{1,2,3}|$
for $(k_1,k_2,k_3) \in \supp \chi_{NR(1,1)}^{(3)}\chi_{R1}^{(3)}$. 
Thus, by Lemma \ref{Le1}, we obtain $|\Phi_0^{(2)} (k_1, k_{2,3})| \gtrsim |k_1| |k_{2,3}|^4$ for $(k_1,k_2,k_3) \in \supp M$.
Therefore, we conclude
\begin{align*}
| M_{6, \vp}^{(3)}  | \lesssim |k_1| M (k_1, k_2, k_3) \lesssim |k_1| \chi_{R1}^{(3)} (k_1, k_2, k_3).
\end{align*}
\end{proof}


\section{nonlinear estimates}

In this section, we present several nonlinear estimates in order to prove main estimates 
which are stated in Section 7. 

\begin{lem} \label{lem_nes01}
Let $s \ge 1/2$ and a $3$-multiplier $N_1^{(3)}$ satisfy 
\begin{align} \label{ml41}
|N_1^{(3)} (k_1, k_2, k_3)| \lesssim \langle k_1 \rangle ( \chi_{R1}^{(3)} (k_1, k_2, k_3) + \chi_{R1}^{(3)} (k_1, k_3, k_2) ). 
\end{align}
Then, for any $\{ v_l\}_{l=1}^3 \subset H^s(\T)$, we have 
\begin{align}
\big\| \sum_{k=k_{1,2,3}} |N_1^{(3)} (k_1, k_2, k_3) | \, \prod_{l=1}^3 |\ha{v}_l(k_l)| \big\|_{l_s^2} & 
\lesssim \prod_{l=1}^3 \| v_l \|_{H^s}^3, \label{nl92} \\
\big\| \sum_{k=k_{1,2,3}} |\ti{N}_1^{(3)} (k_1, k_2, k_3) | \, \prod_{l=1}^3 |\ha{v}_l(k_l)| \big\|_{l_s^2} & 
\lesssim \prod_{l=1}^3 \| v_l \|_{H^s}^3. \label{nl93} 
\end{align}
\end{lem}

\begin{proof}
We only prove (\ref{nl92}). By (\ref{ml41}), it follows that 
\begin{align*}
\langle k_{1,2,3} \rangle^{s} | N_1^{(3)} (k_1, k_2, k_3) | 
& \lesssim \langle k_1  \rangle^{-2s+1} \prod_{l=1}^3 \langle k_l \rangle^s
 (\chi_{R1}^{(3)} (k_1, k_2, k_3) + \chi_{R1}^{(3)}(k_1, k_3, k_2) ) \\
&  \lesssim  \prod_{l=1}^3 \langle k_l \rangle^s
 (\chi_{R1}^{(3)} (k_1, k_2, k_3) + \chi_{R1}^{(3)}(k_1, k_3, k_2) ).
\end{align*}
Here we used $s \ge 1/2$ in the last inequality. 
Thus, by the Schwarz inequality, the left hand side of (\ref{nl92}) is bounded by 
\begin{align*}
& \big\| \langle k \rangle^s |\ha{v}_3(k)| \, \sum_{k_1 \in \Z} \langle k_1 \rangle^{2s}  
|\ha{v}_1 (k_1)| |\ha{v}_2 (-k_1)| \big\|_{l^2} \\
& +\big\| \langle k  \rangle^s |\ha{v}_2(k)| \, \sum_{k_1 \in \Z} \langle k_1 \rangle^{2s}
 |\ha{v}_1(k_1)| |\ha{v}_3(-k_1)|  \big\|_{l^2} \lesssim \prod_{l=1}^3 \| v_l \|_{H^s}. 
\end{align*}
\end{proof}

\begin{lem} \label{lem_nl5}
Let $\vp \in L^2(\T)$, $L \gg \max\{ 1, |\beta E_0(\vp)| \}$ and a 3-multiplier $N_2^{(3)}$ satisfy 
\begin{align} \label{ml31}
|N_2^{(3)} (k_1, k_2, k_3) | \lesssim \langle k_{1,2,3} \rangle^{-1} \langle k_{\max} \rangle^2  
(1-[3 \chi_{H1}^{(3)}]_{sym}^{(3)}) \chi_{NR1}^{(3)}(k_1, k_2, k_3).
\end{align}
Put 
\begin{align*}
& L_{\vp}^{(3)}:= N_2^{(3)}/ \Phi_{\vp}^{(3)} :=
\begin{cases}
N_2^{(3)} / \Phi_{\vp}^{(3)} \hspace{0.3cm} &\text{when}~~ \Phi_{\vp}^{(3)} (k_1, k_2, k_3) \neq 0 \\
0 &\text{when} ~~ \Phi_{\vp}^{(3)} (k_1, k_2, k_3 ) =0 
\end{cases}, \\
& N_{3, \vp}^{(4)}:=[ 3 \ti{L}_{\vp}^{(3)} \chi_{>L}^{(3)} ]_{ext1}^{(4)} [-Q^{(2)} \chi_{NR1}^{(2)} ]_{ext2}^{(4)}. 
\end{align*}
Then, the followings hold: \\
(1) It follows that
\begin{align} \label{ml320}
& |L_{\vp}^{(3)} \chi_{>L}^{(3)} | \lesssim \langle k_{1,2,3} \rangle^{-1} \langle k_{\max} \rangle^{-1} \Lambda_1^{-1} 
(1-[3\chi_{H1}^{(3)}]_{sym}^{(3)})\chi_{NR1}^{(3)} \chi_{>L}^{(3)}
\end{align}
where 
\begin{align*}
\La_1= \min  \{ \langle k_{1,2} \rangle \langle k_{1,3} \rangle, \,  \langle k_{1,2} \rangle  \langle k_{2,3} \rangle, \, 
\langle k_{1,3} \rangle \langle k_{2,3} \rangle \}. 
\end{align*}
(2) When $s \ge 1$, we have 
\begin{align}
\Big\| \sum_{k=k_{1,2,3,4}} |\ti{N}_{3, \vp}^{(4)} (k_1, k_2, k_3, k_4) | \, \prod_{l=1}^4 |\ha{v}_l (k_l)|   \Big\|_{ l_s^2 } 
\lesssim \prod_{l=1}^4 \| v_l \|_{H^s} \label{nl83}
\end{align}
for any $\{ v_l \}_{l=1}^4 \subset H^{s} (\T)$. 
\end{lem}

\begin{proof}
(1) By (\ref{ml31}) and Lemma~\ref{Le2}, we have (\ref{ml320}). 
(2) It suffices to show 
\begin{align} \label{nl82}
\Big\| \sum_{k=k_{1,2,3,4}} |N_{3, \vp}^{(4)} (k_1, k_2, k_3, k_4) | \, \prod_{l=1}^4 |\ha{v}_l (k_l)|   \Big\|_{l_s^2 } 
\lesssim \prod_{l=1}^4 \| v_l \|_{H^s}.
\end{align}
Now, we prove (\ref{nl82}). By (\ref{ml320}) and Remark~\ref{rem_sym}, it suffices to prove
\begin{align} \label{nl91}
\sum_{k \in \Z} \sum_{k=k_{1,2,3,4} } \mathcal{N}_{3}^{(4)} (k_1, k_2, k_3, k_4) \prod_{l=1}^4 |u_l (k_l)| \, |z(k)| 
\lesssim \prod_{l=1}^4 \| u_l \|_{l^2} \| z \|_{l^2},
\end{align}
for any $z \in l^2$ and $\{u_l  \}_{l=1}^4 \subset l^2$, where  
\begin{align*}
& \mathcal{N}_{3}^{(4)} (k_1, k_2, k_3, k_4) =  \langle k_{1,2,3,4} \rangle^{s-1} \Lambda_2^{-1} (|k_3|^2+ |k_4|^2)
\prod_{l=1}^{4} \langle  k_l \rangle^{-s} \\ 
 & \hspace{3cm} \times [ (1-[3 \chi_{H1}^{(3)} ]_{sym}^{(3)} ) \chi_{NR1}^{(3)} ]_{ext1}^{(4)} [ \chi_{NR1}^{(2)}]_{ext2}^{(4)} (k_1,k_2, k_3, k_4), \\
& \Lambda_2=  \min \{ \langle k_{1,2} \rangle  \langle k_{1,3,4} \rangle, \langle k_{1,2} \rangle   \langle k_{2,3,4} \rangle,  \langle k_{1,3,4} \rangle   \langle k_{2,3,4} \rangle   \}. 
\end{align*}
Since $\mathcal{N}_{3}^{(4)}(k_1, k_2, k_3, k_4)$ is symmetric with $(k_1, k_2)$ and $(k_3, k_4)$, 
we assume $|k_1| \le |k_2|$ and $|k_3| \le |k_4|$. 
Then, for $(k_1, k_2, k_3, k_4) \in \supp [ (1-[3 \chi_{H1}^{(3)}]_{sym}^{(3)}) ]_{ext1}^{(4)}$, 
either $|k_1| \ll |k_2| \sim |k_{3,4}|$ or $|k_{3,4}| \lesssim |k_1| \sim |k_2|$ holds. 
Thus, we handle the following four cases. \\
(i) Suppose that $|k_1| \ll |k_2| \sim |k_{3,4}|$ and $|k_3| \sim |k_4|$ hold. 
This implies $|k_{1,2,3,4}| \lesssim |k_2|$ and
\begin{align*}
\mathcal{N}_{3}^{(4)} (k_1, k_2, k_3, k_4) \lesssim \langle k_{1,2,3,4} \rangle^{-1} \La_2^{-1} \langle k_1 \rangle^{-s} \langle k_4 \rangle^{-2s+2}
\lesssim \langle k_{1,2,3,4} \rangle^{-1} \La_2^{-1} \langle k_1 \rangle^{-s}.
\end{align*}
Since $\Lambda_{2}^{-1} \lesssim \max \{  \langle k_1+k_2 \rangle^{-1}, \langle k-k_2 \rangle^{-1}  \}$ with $k=k_{1,2,3,4}$, 
by the Schwarz inequality, we obtain
\begin{align*}
\text{LHS of (\ref{nl91})} & \lesssim  \sum_{k \in \Z} \langle k \rangle^{-1} |z(k)| \, \sum_{k_1 \in \Z} \langle k_1 \rangle^{-s} |u_1(k_1)| \\
& \times \Big\{  \sum_{k_2 \in \Z} \La_2^{-1} |u_2(k_2)| \, \Big( \sum_{k_3 \in \Z} |u_3(k_3)| |u_4 (k-k_{1,2,3})|    \Big) \Big\}
\lesssim \prod_{l=1}^4 \| u_l \|_{l^2} \| z \|_{l^2}. 
\end{align*} 
(ii) Suppose that $|k_{3,4}| \lesssim |k_1| \sim |k_2|$ and $|k_{3}| \sim |k_4|$ hold. 
Then, it follows that $|k_{1,2,3,4}| \lesssim |k_2|$, which leads that 
\begin{align*}
\mathcal{N}_{3}^{(4)} (k_1, k_2, k_3, k_4) \lesssim \langle k_{1,2,3,4} \rangle^{-1} \La_2^{-1} \langle k_1 \rangle^{-s}   \langle k_4 \rangle^{-2s+2}
\lesssim  \langle k_{1,2,3,4} \rangle^{-1}  \La_2^{-1}  \langle k_1 \rangle^{-s}. 
\end{align*}
As before, by the Schwarz inequality, we obtain (\ref{nl91}). \\
(iii) Suppose that $|k_1| \ll |k_2| \sim |k_{3,4}|$ and $|k_3| \ll |k_4|$ hold. 
Then, it follows that $|k_{1,3, 4}| \sim |k_{3,4}| \sim |k_4|$ and $|k_{1,2}| \sim |k_2| \sim |k_4|$. 
Thus, we have $\La_2^{-1} \lesssim \langle k_4 \rangle^{-1}$ and 
\begin{align*}
\mathcal{N}_{3}^{(4)}(k_1, k_2, k_3, k_4) \lesssim \langle k_1 \rangle^{-s} \langle k_2 \rangle^{-1} \langle k_3 \rangle^{-s} \langle k_4 \rangle^{-s+1} 
\lesssim \langle k_1 \rangle^{-s} \langle  k_2 \rangle^{-1} \langle k_3 \rangle^{-s},
\end{align*}
which implies that
\begin{align*}
\sup_{k \in \Z} \sum_{k=k_{1,2,3,4}} \big( \mathcal{N}_{3}^{(4)}(k_1, k_2, k_3, k_4) \big)^2 \lesssim 1.
\end{align*}
Hence, by the Schwarz inequality, we obtain (\ref{nl91}). \\
(iv) Suppose that $|k_{3,4}| \lesssim |k_1| \sim |k_2|$ and $|k_3| \ll |k_4|$ hold. 
Then, it follows that $|k_3| \ll |k_{4}| \sim |k_{3,4}| \lesssim |k_1| \sim |k_2| $ and 
$\La_2^{-1} \lesssim \max \{ \langle k-k_1 \rangle^{-1}, \langle k-k_2  \rangle^{-1}  \}$ with $k=k_{1,2,3,4}$. 
Thus, we have 
\begin{align} \label{mbd1}
\mathcal{N}_{3}^{(4)} (k_1, k_2, k_3, k_4) \lesssim \langle k_{1,2,3,4} \rangle^{-1} \La_2^{-1}  \langle k_3 \rangle^{-s}  \langle k_4 \rangle^{-2s+2}
\lesssim  \langle k_{1,2,3,4} \rangle^{-1}  \La_2^{-1} \langle k_3 \rangle^{-s}. 
\end{align}
When $\La_2^{-1} \lesssim \langle k-k_1 \rangle^{-1}$, by (\ref{mbd1}) and the Schwarz inequality, we have 
\begin{align*}
& \text{LHS of (\ref{nl91})}  \lesssim  \sum_{k \in \Z} \langle k \rangle^{-1} |z(k)| \, \sum_{k_3 \in \Z} \langle k_3 \rangle^{-s} |u_3 (k_3)| \\
& \hspace{1cm} \times \Big\{  \sum_{k_1 \in \Z}  \langle k-k_1 \rangle^{-1} |u_1(k_1)| \, \Big( \sum_{k_2 \in \Z} |u_2(k_2)| |u_4(k-k_{1,2,3})|    \Big) \Big\}
\lesssim \prod_{l=1}^4 \| u_l \|_{l^2} \| z \|_{l^2}.
\end{align*}
Similarly, we obtain (\ref{nl91}) when  $\La_2^{-1} \lesssim \langle k-k_2 \rangle^{-1}$. 
Therefore, we conclude (\ref{nl91}). 
\end{proof}


\section{pointwise upper bounds}

In this section, we present pointwise upper bounds of the multipliers $L_{j, \vp}^{(N)}$ and $M_{j, \vp}^{(N)}$ defined in Proposition~\ref{prop_NF2}. 
We now put 
\begin{align*}
& J_1=\{(3,1), (3,3), (3,5), (3,6),  (3,8), (3,9), (3,10), (3,11), (3,13), \\
& \hspace{4.5cm} (4,1), (4,2), (4,3), (4,4)  \}, \\
& J_2=\{ (3,2), (3,7), (3, 12), (3,14), (3,15)  \}, \hspace{0.5cm} J_3=\{ (2,1), (2,2), (3,4)  \}.
\end{align*}
\begin{lem} \label{lem_pwb1}
Let $f \in L^2(\T)$ 
and $L_{j, f}^{(N)}$ with $(N, j) \in J_1 \cup J_2 \cup J_3$ be as in Proposition~\ref{prop_NF2}. 
Then, the followings hold for $L \gg \max \{1, | \be E_0(f)|  \}$: \\
(I) When $(N, j) \in J_3$, we have
\begin{align} 
|L_{j,f}^{(N)} \chi_{>L}^{(N)} | \lesssim \langle k_{\max}  \rangle^{-1} \chi_{>L}^{(N)}. \label{pwb3} 
\end{align}
(II) When $(N,j) \in J_1$, we have
\begin{align} \label{pwb1}
|L_{j,f}^{(N)} \chi_{>L}^{(N)} | \lesssim \langle k_{1, \dots ,N}  \rangle^{-1} \langle k_{\max} \rangle^{-2} \chi_{>L}^{(N)}. 
\end{align}
(III) When $(N.j) \in J_2$, we have 
\begin{align} \label{pwb2}
|L_{j,f}^{(N)} \chi_{>L}^{(N)} | 
\lesssim \langle k_{1, 2,3}  \rangle^{-1} \langle k_{\max} \rangle^{-1} \La_1^{-1} \chi_{NR1}^{(3)} (1-[ 3 \chi_{H1}^{(3)}  ]_{sym}^{(3)} ) \chi_{>L}^{(3)}
\end{align}
where $\La_1$ is as in Lemma~\ref{lem_nl5}. 
\end{lem}

\begin{rem} \label{rem_pwb11}
By Lemma~\ref{lem_pwb1}, we can easily check that 
\begin{align} \label{pwb01}
|L_{j , f}^{(N)} \chi_{>L}^{(N)}| \lesssim \langle k_{\max} \rangle^{-1}
\end{align} 
for $(N,j) \in J_1 \cup J_2 \cup J_3$. 
By the definition of the multipliers and (\ref{pwb01}), it follows that 
\begin{align} \label{pwb02}
| L_{j, f}^{(N)} \Phi_f^{(N)} \chi_{>L}^{(N)}| \lesssim \langle k_{1,\dots, N} \rangle \langle k_{\max} \rangle^2 
\end{align}
with $(N,j) \in J_1 \cup J_2 \cup J_3$. 
For instance, by the definition, we have
\begin{align*}
|L_{j , \vp}^{(2)} \Phi_{\vp}^{(2)} \chi_{>L}^{(2)} | \le |L_{j , \vp}^{(2)} \Phi_{\vp}^{(2)} | \le |Q^{(2)}| \lesssim \langle k_{1,2} \rangle \langle k_{\max} \rangle^{2}
\end{align*}
for $j=1,2$. This implies that (\ref{pwb02}) holds for $(N,j)=(2,1), (2,2)$. 
Next, we consider the case $(N,j)=(3, 15)$. By the definition and Remark~\ref{rem_sym}, it follows that
\begin{align*}
L_{15, \vp}^{(3)} \Phi_{\vp}^{(3)} =2 \Big[ \frac{Q^{(2)}}{ \Phi_{\vp}^{(2)} } \chi_{NR1}^{(2)} \chi_{H2}^{(2)} \chi_{>L}^{(2)} \Big]_{ext1}^{(3)} [Q^{(2)} \chi_{NR1}^{(2)}]_{ext2}^{(3)}
=-2 [ L_{2, \vp}^{(2)} \chi_{>L}^{(2)} ]_{ext1}^{(3)} [Q^{(2)} \chi_{NR1}^{(2)} ]_{ext2}^{(3)}.
\end{align*}
Thus, by $|L_{2, \vp}^{(2)} \chi_{>L}^{(2)} | \lesssim \langle k_{\max}  \rangle^{-1} $, we have
\begin{align*}
|L_{15 , \vp}^{(3)} \Phi_{\vp}^{(3)} \chi_{>L}^{(3)} | \leq |L_{15 , \vp}^{(3)} \Phi_{\vp}^{(3)} | \lesssim |k_2|^2 + |k_3|^2 \lesssim \langle k_{\max}  \rangle^2.
\end{align*}
This implies that (\ref{pwb02}) holds with $(N, j)=(3,15)$. 

Therefore, we can apply Proposition~\ref{prop_NF11} with $\ti{m}^{(N)}= \ti{L}_{j, f}^{(N)}  \chi_{>L}^{(N)} $. 
\end{rem}

\begin{proof}
(I) Firstly, we prove (\ref{pwb3}) for $(N, j) \in J_3$. \\ 
(Ia) Estimate of $L_{1, f}^{(2)} \chi_{>L}^{(2)}$ and $L_{2, f}^{(2)} \chi_{>L}^{(2)}$: 
By Lemma~\ref{Le1}, we have
\begin{align*}
&  |L_{1, f}^{(2)} \chi_{>L}^{(2)} | \lesssim \frac{1}{ |k_1| |k_2| } \chi_{H1}^{(2)} \chi_{NR1}^{(2)} \chi_{>L}^{(2)} 
\lesssim \langle k_{\max} \rangle^{-1} \chi_{>L}^{(2)}  , \\
& |L_{2, f}^{(2)} \chi_{>L}^{(2)} | \lesssim \frac{1}{ |k_2|^2 } \chi_{H2}^{(2)} \chi_{NR1}^{(2)} \chi_{>L}^{(2)} 
\lesssim \langle k_{\max} \rangle^{-2} \chi_{>L}^{(2)} .
\end{align*} 
(Ib) Estimate of $L_{4,f}^{(3)} \chi_{>L}^{(3)}$: 
By Lemma~\ref{Le2}, it follows that 
\begin{align} \label{pbd01}
|\Phi_{f}^{(3)} \chi_{NR3}^{(3)} \chi_{H1}^{(3)}  m_{>L}^{(3)}   | \gtrsim |k_{1,2}| |k_3|^4 \chi_{NR3}^{(3)} \chi_{H1}^{(3)} \chi_{>L}^{ (3)}. 
\end{align}
Thus, we have 
\begin{align*} 
|L_{4, f}^{(3)} \chi_{>L}^{(3)}  | & \le \frac{1}{ |k_1| |k_{1,2}|  |k_3|^2  } \, \chi_{NR3}^{(3)} \chi_{H1}^{(3)} \chi_{>L}^{(3)}
\lesssim \langle k_{\max} \rangle^{-2} \chi_{>L}^{(3)}.
\end{align*}
Therefore, we conclude (\ref{pwb3}) for $(N, j) \in J_3 $.

\vspace{0.5em}

\noindent
(II) Secondly, we prove (\ref{pwb1}) for $(N,j) \in J_1$. By Lemma~\ref{Le1} and Remark~\ref{rem_sym}, we have 
\begin{align} \label{pest01}
& \Big| \Big[ \frac{Q^{(2)}}{\Phi_{f}^{(2)}} \chi_{NR1}^{(2)} \chi_{>L}^{(2)}  \Big]_{ext1}^{(3)} [Q^{(2)} \chi_{NR1}^{(2)} ]_{ext2}^{(3)} 
\chi^{(3)}  \Big| \nonumber \\
& \lesssim \frac{1}{ |k_1| |k_{2,3}|  } [\chi_{NR1}^{(2)}]_{ext1}^{(3)} \, |k_{2,3}| (|k_2|^2 +|k_3|^2) [ \chi_{NR1}^{(2)} ]_{ext2}^{(3)} \chi^{(3)}
\lesssim \frac{|k_2|^2 +|k_3|^2}{|k_1|} \chi_{NR2}^{(3)} \, \chi^{(3)}
\end{align}
for any characteristic function $\chi^{(3)}$. \\
(IIa) Estimate of $L_{1,f}^{(3)} \chi_{>L}^{(3)}$: By Lemma~\ref{Le2}, we have 
\begin{align*}
|\Phi_{f}^{(3)} \chi_{H1}^{(3)} \chi_{NR1}^{(3)} \chi_{>L}^{(3)} | \gtrsim  |k_{1,2}| |k_3|^4  \chi_{H1}^{(3)} \chi_{NR1}^{(3)} \chi_{>L}^{(3)}, 
\end{align*}
which implies that 
\begin{align*} 
|L_{1, f}^{(3)} \chi_{>L}^{(3)}  | \lesssim  \frac{1}{ |k_{1,2}| |k_3|^3  }  \chi_{H1}^{(3)} \chi_{NR1}^{(3)} \chi_{>L}^{(3)} 
 \lesssim \langle k_{max} \rangle^{-3} \chi_{>L}^{(3)}.
\end{align*}
(IIb) Estimate of $L_{3, f}^{(3)} \chi_{>L}^{(3)}$:
Since $12|k_2| < 3|k_3| < |k_1|=k_{\max}$ for $(k_1, k_2, k_3) \in \supp \chi_{NR(1,2)}^{(3)}$, by Lemma~\ref{Le2}, we have 
\begin{align*}
| \Phi_{f}^{(3)} \chi_{NR2}^{(3)} \,  \chi_{NR(1,2)}^{(3)} \chi_{>L}^{(3)} | \gtrsim |k_3| |k_1|^4 \chi_{NR2}^{(3)} \, \chi_{NR(1,2)}^{(3)} \chi_{>L}^{(3)}.
\end{align*}
Thus, by (\ref{pest01}), we have
\begin{align*}
|L_{3, f}^{(3)}  \chi_{>L}^{(3)}| \lesssim \frac{|k_3|}{ |k_1|^5  } \chi_{NR2}^{(3)} \chi_{NR(1,2)}^{(3)} \chi_{>L}^{(3)} 
\lesssim \langle k_{\max} \rangle^{-4} \chi_{>L}^{(3)} .
\end{align*}
(IIc) Estimate of $L_{5, f}^{(3)} \chi_{>L}^{(3)}$: 
By (\ref{pbd01}), we have 
\begin{align*}
|L_{5,f}^{(3)} \chi_{>L}^{(3)} | \lesssim \frac{1}{ |k_{1,2}| |k_3|^3 } \chi_{NR3}^{(3)} \chi_{H1}^{(3)} \chi_{>L}^{(3)}
  \lesssim \langle k_{\max}  \rangle^{-3} \chi_{>L}^{(3)}.
\end{align*}
(IId) Estimate of $L_{6, f}^{(3)} \chi_{>L}^{(3)}$:
For $(k_1,k_2, k_3) \in \supp \chi_{NR(1,1)}^{(3)} (1-\chi_{H1}^{(3)}) \chi_{A1}^{(3)} $, it follows that
\begin{align*}
16 |k_1| < 4|k_2| < |k_3|=k_{\max} \le 16 |k_2|, 
\end{align*}
which leads that $|k_{1,2}| \sim |k_2| \sim |k_3|$ and $|k_{1,3} | \sim |k_{2,3}| \sim |k_3|$. 
Thus, by  Lemma~\ref{Le2}, 
\begin{align*}
|\Phi_{f}^{(3)}  \chi_{NR2}^{(3)} \chi_{NR(1,1)}^{(3)}(1 - \chi_{H1}^{(3)} ) \chi_{A1}^{(3)} \chi_{>L}^{(3)}  | 
\gtrsim |k_3|^5  \chi_{NR2}^{(3)} \chi_{NR(1,1)}^{(3)}(1 - \chi_{H1}^{(3)} ) \chi_{A1}^{(3)} \chi_{>L}^{(3)}. 
\end{align*}
Hence, by (\ref{pest01}), we have
\begin{align*} 
|L_{6, f}^{(3)} \chi_{> L}^{(3)} | \lesssim \frac{1}{|k_1| |k_3|^3 }  \chi_{NR2}^{(3)} \chi_{NR(1,1)}^{(3)}(1 - \chi_{H1}^{(3)} ) \chi_{A1}^{(3)} \chi_{>L}^{(3)} 
\lesssim \langle k_{\max}  \rangle^{-3} \chi_{>L}^{(3)}.
\end{align*}
(IIe) Estimate of $L_{8, f}^{(3)} \chi_{>L}^{(3)} $, $L_{9, f}^{(3)} \chi_{>L}^{(3)}$ and $L_{10, f}^{(3)} \chi_{>L}^{(3)}$:
By Lemma~\ref{Le1} and Remark~\ref{rem_sym}, 
\begin{align} \label{pbd1}
|L_{8, f}^{ (3) } \Phi_{f}^{(3)}| &
\lesssim  \frac{1}{|k_1| |k_{2,3}| } [\chi_{NR1}^{(2)}]_{ext1}^{(3)} |k_2| |k_3| |k_{2,3}| [\chi_{NR1}^{(2)}]_{ext2}^{(3)} (1-\chi_{R1}^{(3)}) \chi_{NR(1,1)}^{(3)} 
\notag \\
& \lesssim \frac{|k_2| |k_3|}{|k_1|} \chi_{NR2}^{(3)} (1-\chi_{R1}^{(3)})  \chi_{NR(1,1)}^{(3)}.
\end{align}
For $(k_1, k_2, k_3) \in \supp \chi_{NR(1,1)}^{(3)}$, it follows that $\frac{16}{5} \max \{ |k_1|, |k_2|  \}<|k_3|=k_{\max}$ and 
$\chi_{NR2}^{(3)}(1-\chi_{R1}^{(3)})=\chi_{NR3}^{(3)}$. Thus, by Lemma~\ref{Le2}, 
\begin{align} \label{pest11}
|\Phi_f^{(3)} \chi_{NR3}^{(3)} \chi_{NR(1,1)}^{(3)} \chi_{>L}^{(3)}  | \gtrsim |k_{1,2}| |k_3|^4 \chi_{NR3}^{(3)} \chi_{NR(1,1)}^{(3)} \chi_{>L}^{(3)}.
\end{align}
Hence, by (\ref{pbd1}),we obtain
\begin{align*} 
|L_{8, f}^{(3)} \chi_{>L}^{(3)} | \lesssim  \frac{|k_2|}{|k_1| |k_{1,2}|} \frac{1}{|k_3|^3} \chi_{NR3}^{(3)} \chi_{NR(1,1)}^{(3)} \chi_{>L}^{(3)}
\lesssim \langle k_{\max} \rangle^{-3} \chi_{>L}^{(3)}.
\end{align*}
Similarly, by (\ref{pest11}), we have
\begin{align*}
& |L_{9, f}^{(3)} \chi_{>L}^{(3)} | \lesssim  \frac{1}{|k_3|^3} \chi_{NR3}^{(3)} \chi_{NR(1,1)}^{(3)}  \chi_{>L}^{(3)}
\lesssim \langle k_{\max} \rangle^{-3} \chi_{>L}^{(3)}, \\ 
& |L_{10, f}^{(3)} \chi_{>L}^{(3)} | \lesssim  \frac{1}{|k_3|^3} \chi_{NR3}^{(3)} \chi_{NR(1,1)}^{(3)} \chi_{>L}^{(3)}
\lesssim \langle k_{\max} \rangle^{-3} \chi_{>L}^{(3)}. 
\end{align*}
(IIf) Estimate of $L_{11, f}^{(3)} \chi_{>L}^{(3)}$: 
Since $|k_{1,2}| \sim |k_{1,3}| \sim |k_1|=k_{\max}$ holds for $(k_1, k_2, k_3) \in \supp \chi_{NR(2,2)}^{(3)} \chi_{A2}^{(3)}$, 
by Lemma~\ref{Le2}, we have
\begin{align*}
|\Phi_{f}^{(3)} \chi_{NR2}^{(3)}  \chi_{NR(2,2)}^{(3)} \chi_{A2}^{(3)} \chi_{>L}^{(3)} |
 \gtrsim |k_{2,3}| |k_1|^4  \chi_{NR2}^{(3)}  \chi_{NR(2,2)}^{(3)} \chi_{A2}^{(3)} \chi_{>L}^{(3)}.
\end{align*}
Hence, by (\ref{pest01}), we have
\begin{align*} 
|L_{11, f}^{(3)} \chi_{>L}^{(3)} | \lesssim \frac{1}{ |k_{2,3}| |k_1|^3  }  \chi_{NR2}^{(3)}  \chi_{NR(2,2)}^{(3)} \chi_{A2}^{(3)} \chi_{>L}^{(3)} 
\lesssim \langle k_{\max} \rangle^{-3} \chi_{>L}^{(3)}. 
\end{align*}
(IIg) Estimate of $L_{13, f}^{(3)} \chi_{>L}^{(3)}$: For $ (k_1, k_2, k_3) \in \supp \chi_{NR(2,1)}^{(3)} \chi_{A3}^{(3)} $, it follows that 
\begin{align*}
\frac{16}{5} |k_1| < \min \{ |k_2|, |k_3|\}, \hspace{0.5cm} |k_{1,2,3}| \sim |k_{2,3}| \lesssim |k_2| \sim |k_3|. 
\end{align*}
Then, by Lemma~\ref{Le2}, we have
\begin{align*}
|\Phi_{f}^{(3)} \chi_{NR2}^{(3)} \chi_{NR(2,1)}^{(3)} \chi_{A3}^{(3)} \chi_{>L}^{(3)} | 
\gtrsim |k_{1,2,3}| |k_3|^4  \chi_{NR2}^{(3)} \chi_{NR(2,1)}^{(3)} \chi_{A3}^{(3)} \chi_{>L}^{(3)}. 
\end{align*}
Thus, by (\ref{pest01}), we have
\begin{align*} 
|L_{13, f}^{(3)} \chi_{>L}^{(3)}| \lesssim \frac{1}{ |k_1| |k_{1,2,3}| |k_3|^2 }  \chi_{NR2}^{(3)} \chi_{NR(2,1)}^{(3)} \chi_{A3}^{(3)} \chi_{>L}^{(3)}
\lesssim \langle k_{1,2,3} \rangle^{-1}  \langle k_{\max} \rangle^{-2} \chi_{>L}^{(3)}.
\end{align*}
Now, we present some estimates in order to obtain pointwise upper bounds of $L_{j, f}^{(4)} \chi_{>L}^{(4)}$ with $j=1,2,3$ and $4$. 
By Lemma~\ref{Le2}, we have 
\begin{align} \label{pbd3}
\Big|  \frac{q_2^{(3)}}{ \Phi_{f}^{(3)}} \chi_{H1}^{(3)} \chi_{NR3}^{(3)} \chi_{>L}^{(3)} \Big| 
\lesssim \frac{1}{|k_1| |k_2| |k_{3}|^2  } \chi_{H1}^{(3)} \chi_{NR3}^{(3)} \chi_{>L}^{(3)}. 
\end{align}
Let $\chi^{(4)}$ be a characteristic function such that $\supp \chi^{(4)} \subset  \supp [ \chi_{H1}^{(3)} ]_{ext1}^{(4)}$. 
Then, by (\ref{pbd3}), $\chi^{(4)} = [\chi_{H1}^{(3)}]_{ext1}^{(4)} \chi^{(4)}$ and Remark~\ref{rem_sym}, we have 
\begin{align}
& \Big| \Big[ \frac{q_2^{(3)}}{\Phi_{f}^{(3)} } \chi_{NR3}^{(3)} \chi_{>L}^{(3)}  \Big]_{ext1}^{(4)} \chi^{(4)} \Big|
= \Big| \Big[ \frac{q_2^{(3)}}{\Phi_{f}^{(3)} } \chi_{H1}^{(3)} \chi_{NR3}^{(3)} \chi_{>L}^{(3)}  \Big]_{ext1}^{(4)} \chi^{(4)} \Big| \notag  \\
& \lesssim   \frac{1}{|k_1| |k_2| |k_{3,4}|^2} [\chi_{H1}^{(3)} \chi_{NR3}^{(3)} ]_{ext1}^{(4)} \chi^{(4)}
 \lesssim \frac{1}{|k_1| |k_2| |k_{3,4}|^2 } [\chi_{NR3}^{(3)} ]_{ext1}^{(4)} \chi^{(4)}. \label{pest03} 
\end{align}
Similarly, we have
\begin{align} \label{pest02}
\Big| \Big[ \frac{q_2^{(3)}}{\Phi_{0}^{(3)} } \chi_{NR3}^{(3)}  \Big]_{ext1}^{(4)} \chi^{(4)} \Big|
\lesssim \frac{1}{|k_1| |k_2| |k_{3,4}|^2 } [\chi_{NR3}^{(3)} ]_{ext1}^{(4)} \chi^{(4)}.
\end{align}
By Remark~\ref{rem_sym} and (\ref{pest03}), we get
\begin{align} \label{pest04}
& \Big| \Big[ \frac{q_2^{(3)}}{ \Phi_f^{(3)}} \chi_{NR3}^{(3)} m_{>L}^{(3)} \Big]_{ext1}^{(4)} 
[Q^{(2)} \chi_{NR1}^{(2)}  ]_{ext2}^{(4)}  \chi^{(4)} \Big| \notag \\
& \lesssim \frac{1}{|k_1| |k_2| |k_{3,4}|^2 } [\chi_{NR3}^{(3)}]_{ext1}^{(4)} |k_{3,4}| (|k_3|^2 +|k_4|^2) [ \chi_{NR1}^{(2)} ]_{ext2}^{(4)} \chi^{(4)} \notag \\
& = \frac{|k_3|^2 + |k_4|^2}{|k_1| |k_2| |k_{3,4}| } \chi_{NR1}^{(4)} \chi^{(4)}
\end{align} 
As before, by Remark~\ref{rem_sym} and (\ref{pest02}), we have
\begin{align}
& \Big| \Big[ \frac{q_2^{(3)}}{ \Phi_{0}^{(3)}} \chi_{NR3}^{(3)} \Big]_{ext1}^{(4)} [Q_1^{(2)} \chi_{NR1}^{(2)}  ]_{ext2}^{(4)}  \chi^{(4)} \Big| 
\lesssim \frac{|k_3|^2 + |k_4|^2}{|k_1||k_2| |k_{3,4}| } \, \chi_{NR1}^{(4)} \chi^{(4)} ,  \label{pest05} \\
&  \Big| \Big[ \frac{q_2^{(3)}}{ \Phi_{0}^{(3)}} \chi_{NR3}^{(3)} \Big]_{ext1}^{(4)} [Q_2^{(2)} \chi_{NR1}^{(2)}  ]_{ext2}^{(4)}  \chi^{(4)} \Big| 
\lesssim \frac{|k_3| |k_4|}{|k_1||k_2| |k_{3,4}| } \, \chi_{NR1}^{(4)}  \chi^{(4)} . \label{pest06} 
\end{align}
We notice that
\begin{align*}
& \supp \chi_{H1}^{(4)} \subset \supp \chi_{NR(1,1)}^{(4)} \subset \supp [ \chi_{H1}^{(3)} ]_{ext1}^{(4)},  \\
& \supp \chi_{NR(1,1)}^{(4)}(1-\chi_{H1}^{(4)}) \subset \supp [\chi_{H1}^{(3)}]_{ext1}^{(4)}, 
\hspace{0.3cm} \supp \chi_{NR(2,1)}^{(4)} \subset \supp [\chi_{H1}^{(3)}]_{ext1}^{(4)}. 
\end{align*}
\noindent
(IIh) Estimate of $L_{1, f}^{(4)} \chi_{>L}^{(4)}$: 
For $(k_1, k_2, k_3, k_4) \in \supp \chi_{H1}^{(4)} (1- \chi_{R1}^{(4)}) (1-\chi_{R5}^{(4)})$, at least one of 
\begin{align*}
& k_{1,2,3} \neq 0 \hspace{0.3cm} \text{and} \hspace{0.3cm} |k_4| > 4^3 \max \{  |k_1|, |k_2|, |k_3| \} > 4^5 \text{med} \{ |k_1|, |k_2|, |k_3| \}, \\
& k_{1,2,3} \neq 0 \hspace{0.3cm} \text{and} \hspace{0.3cm} |k_4|^{4/5}> 4^3 \max \{ |k_1|, |k_2|, |k_3|  \}
\end{align*}
or 
\begin{align*}
k_{1,2,3} \neq 0, \hspace{0.3cm} |k_4| > 4^3 \max\{ |k_1|, |k_2|, |k_3|  \} \hspace{0.3cm} \text{and} \hspace{0.3cm} |k_{1,2,3}| > 16 \min\{ |k_1|, |k_2|, |k_3| \}
\end{align*}
holds. Thus, we can apply Lemma~\ref{Le3} to have 
\begin{align*}
& |\Phi_{f}^{(4)} \chi_{NR1}^{(4)} \chi_{H1}^{(4)} (1-\chi_{R1}^{(4)}) (1- \chi_{R5}^{(4)}) \chi_{>L}^{(4)} | \\
& \gtrsim |k_{1,2,3}| |k_4|^4 \chi_{NR1}^{(4)}  \chi_{H1}^{(4)} (1-\chi_{R1}^{(4)}) (1- \chi_{R5}^{(4)}) \chi_{>L}^{(4)}. 
\end{align*}
Hence, by (\ref{pest05}) with $\chi^{(4)}= \chi_{H1}^{(4)} (1-\chi_{R1}^{(4)}) (1- \chi_{R5}^{(4)})  $, we obtain
\begin{align*} 
|L_{1, f}^{(4)} \chi_{>L}^{(4)} | \lesssim 
\frac{1}{ |k_1| |k_2|  |k_{1,2,3}| |k_4|^3 } \chi_{NR1}^{(4)} \chi_{H1}^{(4)} (1-\chi_{R1}^{(4)}) (1-\chi_{R5}^{(4)}) \chi_{>L}^{(4)}
\lesssim \langle  k_{\max} \rangle^{-3} \chi_{>L}^{(4)} .
\end{align*}
(IIi) Estimate of $L_{2, f}^{(4)} \chi_{>L}^{(4)}$ and $L_{3,f}^{(4)} \chi_{>L}^{(4)}$: 
Now, we assume that $|k_1| \le |k_2|$ because 
$L_{2,f}^{(4)} \chi_{>L}^{(4)} (k_1, k_2, k_3, k_4)$ and $L_{3,f}^{(4)} \chi_{>L}^{(4)} (k_1, k_2, k_3, k_4)$ are symmetric with $(k_1, k_2)$.  
Then, if $(k_1, k_2, k_3, k_4) \in  \supp \chi_{NR(1,1)}^{(4)} (1-\chi_{H1}^{(4)}) \chi_{A4}^{(4)}$, (\ref{rel3}) follows. 
Thus, by Lemma~\ref{Le3}, we have 
\begin{align*}
& \big| \Phi_{f}^{(4)}  \chi_{NR1}^{(4)} \chi_{NR(1,1)}^{(4)} (1-\chi_{H1}^{(4)} ) \chi_{A4}^{(4)} \chi_{>L }^{(4)} \big| \\
& \gtrsim |k_{1,2,3}| |k_4|^4 \, \chi_{NR1}^{(4)} \chi_{NR(1,1)}^{(4)}(1 -\chi_{H1}^{(4)} ) \chi_{A4}^{(4)} \chi_{>L }^{(4)}.
\end{align*}
Hence, by (\ref{pest05}) with $\chi^{(4)}=\chi_{NR(1,1)}^{(4)} (1-\chi_{H1}^{(4)}) \chi_{A4}^{(4)}$, we obtain 
\begin{align*} 
|L_{2, f}^{(4)} \chi_{>L}^{(4)}| \lesssim \frac{1}{ |k_1| |k_2| |k_{1,2,3}| |k_4|^3 } \, \chi_{NR1}^{(4)} 
\chi_{NR(1,1)}^{(4)} (1-\chi_{H1}^{(4)} ) \chi_{A4}^{(4)} \chi_{>L }^{(4)} \lesssim  \langle k_{\max} \rangle^{-3} \chi_{>L}^{(4)}.
\end{align*}
We notice that $(k_1, k_2, k_3, k_4) \in \supp \chi_{NR(1,1)}^{(4)} \chi_{A4}^{(4)} $ means (\ref{rel3}). 
Thus, by Lemma~\ref{Le3} and (\ref{pest06}) with $\chi^{(4)}=\chi_{NR(1,1)}^{(4)} \chi_{A4}^{(4)}$, we have
\begin{align*} 
|L_{3, f}^{(4)} \chi_{>L}^{(4)}| \lesssim \frac{1}{  |k_1| |k_2| |k_{1,2,3}| |k_4|^3} \chi_{NR1}^{(4)} \chi_{NR(1,1)}^{(4)} \chi_{A4}^{(4)} \chi_{>L}^{(4)} 
 \lesssim \langle k_{\max}  \rangle^{-3} \chi_{>L}^{(4)}.
\end{align*}
(IIj) Estimate of $L_{4, f}^{(4)} \chi_{>L}^{(4)}$: 
Since $L_{4,f}^{(4)} \chi_{>L}^{(4)} (k_1, k_2, k_3, k_4)$ is symmetric with $(k_1, k_2)$ and $(k_3, k_4)$, 
we assume $|k_1| \le |k_2|$, $|k_3| \le |k_4|$. 
Then, $(k_1, k_2, k_3, k_4) \in  \supp \chi_{NR(2,1)}^{(4)}$ means (\ref{L12}). 
Thus, by Lemma~\ref{Le4}, we have
\begin{align*}
|\Phi_f^{(4)} \chi_{NR1}^{(4)} \chi_{NR(2,1)}^{(4)} \chi_{>L}^{(4)} | \gtrsim |k_{1,2,3,4}| |k_4|^4 \chi_{NR1}^{(4)} \chi_{NR(2,1)}^{(4)} \chi_{>L}^{(4)}. 
\end{align*}
Hence, by (\ref{pest04}) with $\chi^{(4)}= \chi_{NR(2,1)}^{(4)} $, we obtain 
\begin{align*}
|L_{4,f}^{(4)} m_{>L}^{(4)} | \lesssim \frac{1}{|k_1| |k_2| |k_{1,2,3,4} |^2 |k_4|^2 } \, \chi_{NR1}^{(4)} \chi_{NR(2,1)}^{(4)} \chi_{>L}^{(4)}
\lesssim \langle k_{1,2,3,4} \rangle^{-2} \langle k_{\max} \rangle^{-2} \chi_{>L}^{(4)}.
\end{align*}
Therefore, we conclude (\ref{pwb1}) for $(N,j) \in J_1$.

\vspace{0.5em}

\noindent
(III) Finally, we prove (\ref{pwb2}) for $(N,j) \in J_2$. 
If $L_{j, f}^{(N)} \Phi_{f}^{(N)}$ satisfies (\ref{ml31}) in Lemma~\ref{lem_nl5}, by (\ref{ml320}),  
it is verified that $L_{j, f}^{(N)} \chi_{>L}^{(N)}$ satisfies (\ref{pwb2}). 
Thus, we now prove that $L_{j, f}^{(N)} \Phi_{f}^{(N)}$ with $(N, j) \in J_2$ satisfies (\ref{ml31}). \\
(IIIa) By $|k_{1,2,3}| \le \langle k_{1,2,3} \rangle^{-1} \langle k_{\max} \rangle^2$, $L_{2, f}^{(3)} \Phi_{f}^{(3)}$ satisfies (\ref{ml31}). \\
(IIIb) Estimate of $L_{7, f}^{(3)} \Phi_{f}^{(3)}$:
If $(k_1, k_2, k_3) \in \supp \chi_{NR(1,1)}^{(3)} (1-\chi_{H1}^{(3)}) (1- \chi_{A1}^{(3)}) $, then 
\begin{align*}
 |k_2| \le 4|k_1|, \hspace{0.5cm} k_{\max} = |k_3| \le 16 \max \{ |k_1|, |k_2|  \}, 
\end{align*}
which means that $(k_1, k_2, k_3) \in \supp (1-[3 \chi_{H1}^{(3)}]_{sym}^{(3)} )$ and $|k_{1,2,3}| \lesssim |k_1|$. 
Thus, by (\ref{pest01}), we have
\begin{align*}
|L_{7, f}^{(3)} \Phi_{f}^{(3)}|&  \lesssim \frac{|k_3|^2}{|k_1|} \chi_{NR2}^{(3)} \chi_{NR(1,1)}^{(3)} (1-\chi_{H1}^{(3)}) (1- \chi_{A1}^{(3)}) \chi_{NR1}^{(3)} \\
& \lesssim \langle k_{1,2,3} \rangle^{-1} \langle k_{\max} \rangle^{2} (1-[3 \chi_{H1}^{(3)} ]_{sym}^{(3)}) \chi_{NR1}^{(3)}. 
\end{align*}
Similarly, by the definition of the characteristic function and (\ref{pest01}), 
it is verified that $L_{j,f}^{(3)} \Phi_f^{(3)}$ satisfies (\ref{ml31}) for $j=12, 14$ and $15$.
\end{proof}


Next, we give pointwise upper bounds of some multipliers $M_{j, \vp}^{(N)}$ defined in Proposition~\ref{prop_NF2}. 

\begin{lem} \label{lem_pwb2}
Let $f \in L^2(\T)$ and $M_{j, f}^{(N)}$ be as in Proposition~\ref{prop_NF2}. 
Then, the followings hold for $L \gg \max\{ 1, |\be E_0(f)| \}$: \\
(I) When $(N,j) \in \{(3,8), (3,9), (4,4), (4,7), (4,8), (4,9), (4,10)\}$, we have 
\begin{align}
|M_{j,f}^{(N)} | \lesssim L^2. \label{pwb11} 
\end{align}
(II) It follows that
\begin{align}
& |M_{6,f}^{(4)} | \lesssim \langle \min\{ |k_1|, |k_2|, |k_3| \}  \rangle^{-1} \notag \\
& \hspace{2cm} \times \langle \max\{ |k_1|, |k_2|, |k_3| \} \rangle^{1/8} \langle \text{{\upshape med}} \{ |k_1|, |k_2|, |k_3| \} \rangle^{1/8} 
\chi_{R5}^{(4)}  \chi_{H1}^{(4)}, \label{pwb12}  \\
& |M_{11,f}^{(4)} | \lesssim  \langle k_1 \rangle^{-2}  \langle k_4 \rangle^2 \chi_{NR(2,2)}^{(4)}.  \label{pwb13}
\end{align}
(III) When $(N, j) \in  \{   (3,3),  (3,5), (3,7), (3,10), (3,11), (3,12) \}$, we have 
\begin{align} \label{pwb15}
|M_{j, f}^{(3)} (k_1, k_2, k_3) | \lesssim \langle  k_1 \rangle ( \chi_{R1}^{(3)} (k_1, k_2, k_3) + \chi_{R1}^{(3)} (k_1, k_3, k_2)  ). 
\end{align}
\end{lem}

\begin{proof} 
(I) Firstly, we show (\ref{pwb11}).\\ 
(Ia) Estimate of $M_{8, f}^{(3)}$: 
For $(k_1, k_2, k_3) \in \supp \chi_{NR(1,1)}^{(3)} [\chi_{\le L}^{(2)} ]_{ext1}^{(3)}$, it follows that 
$k_{\max} =|k_3| < 4|k_{2,3}|/ 3 \le 4L/3$. 
Thus, by Lemma~\ref{Le1} and Remark~\ref{rem_sym}, we have
\begin{align*}
|M_{8, f}^{(3)}| \lesssim \frac{|k_3|^2}{|k_1|} \chi_{NR2}^{(3)} \chi_{NR(1,1)}^{(3)} [\chi_{ \le L}^{(2)}]_{ext1}^{(3)} \lesssim L^2.
\end{align*} 
(Ib) Estimate of $M_{9, f}^{(3)}$: 
Since $\chi_{NR(1,1)}^{(3)} =[\chi_{H1}^{(2)}]_{ext1}^{(3)} \chi_{NR(1,1)}^{(3)}$, 
by Lemma~\ref{Le1} and Remark~\ref{rem_sym}, we have 
\begin{equation} \label{pest25}
\begin{split}
& \Big| \Big[ \Big( \frac{Q^{(2)}}{\Phi_{f}^{(2)} }- \frac{Q^{(2)}}{\Phi_0^{(2)}}  \Big) \chi_{NR1}^{(2)} m_{>L}^{(2)} \Big]_{ext1}^{(3)} \chi_{NR(1,1)}^{(3)} \Big|\\
= & \Big| \Big[ \Big( \frac{Q^{(2)}}{\Phi_{f}^{(2)} }- \frac{Q^{(2)}}{\Phi_0^{(2)}}  \Big) \chi_{H1}^{(2)} \chi_{NR1}^{(2)} \chi_{>L}^{(2)}  \Big]_{ext1}^{(3)} \chi_{NR(1,1)}^{(3)} \Big|\\
& \lesssim \Big[ \frac{|\be E_0(f)|}{|k_1| |k_{2}|^3 } \chi_{H1}^{(2)} \chi_{NR1}^{(2)}  \Big]_{ext1}^{(3)} \chi_{NR(1,1)}^{(3)} 
= \frac{|\be E_0(f)|}{ |k_1| |k_{2,3}|^3} [\chi_{NR1}^{(2)} ]_{ext1}^{(3)} \chi_{NR(1,1)}^{(3)}.
\end{split}
\end{equation}
Thus, by the definition of $\chi_{NR(1,1)}^{(3)}$ and $|\be E_0(f)| \lesssim L$, 
we have $|M_{9,f}^{(3)}| \lesssim L $. \\
(Ic) Estimate of $M_{4, f}^{(4)}$: 
Since $(q_2^{(3)} \chi_{H1}^{(3)}) (k_1, k_2, k_3)$ is symmetric with $(k_1, k_2)$ and 
$\chi_{H1}^{(3)}  (k_{3,4}, k_2, k_1) \chi_{NR(1,2)}^{(4)}(k_1, k_2, k_3, k_4) = \chi_{NR(1,2)}^{(4)}(k_1, k_2, k_3, k_4)$, it follows that 
\begin{equation} \label{net24}
\begin{split}
& [[3 q_2^{(3)} \chi_{H1}^{(3)}]_{sym}^{(3)}]_{ext1}^{(4)} \, \chi_{NR(1,2)}^{(4)} (k_1, k_2, k_3, k_4)\\
=&  [[3 q_2^{(3)} \chi_{H1}^{(3)}]_{sym}^{(3)}]_{ext1}^{(4)} \chi_{H1}^{(3)} (k_{3,4}, k_2, k_1) \,  \chi_{NR(1,2)}^{(4)} (k_1, k_2, k_3, k_4) \\
=&  (q_2^{(3)} \chi_{H1}^{(3)} ) (k_{3,4}, k_2, k_1) \, \chi_{NR(1,2)}^{(4)} (k_1, k_2, k_3, k_4)\\
=& q_2^{(3)} (k_{3,4}, k_2, k_1) \, \chi_{NR(1,2)}^{(4)} (k_1, k_2, k_3, k_4). 
\end{split}
\end{equation}
Since $\chi_{NR3}^{(3)} \chi_{>L}^{(3)}/ \Phi_{f}^{(3)} $ is symmetric, by (\ref{net24}), we have 
\begin{align*}
M_{4, f}^{(4)}= -4 \Big( \frac{ q_2^{(3)}  }{\Phi_f^{(3)} } \chi_{NR3}^{(3)} \chi_{>L}^{(3)} \Big) (k_{3,4}, k_2, k_1) 
(Q^{(2)} \chi_{NR1}^{(2)} )(k_3, k_4)  \chi_{NR(1,2)}^{(4)}  (k_1, k_2, k_3, k_4).
\end{align*}
For $(k_1, k_2, k_3, k_4) \in \supp \chi_{NR(1,2)}^{(4)}$, it follows that $48 |k_3| < 12 |k_4| < |k_1|$ and $16|k_2|< |k_1|$. 
Thus, by Lemma~\ref{Le2}, we obtain 
\begin{align*}
|M_{4, f}^{(4)}| \lesssim \frac{1}{|k_2|  |k_{3,4}| |k_1|^2 } \chi_{NR3}^{(3)} (k_{3,4}, k_2, k_1)
|k_{3,4}| |k_4|^2 \chi_{NR1}^{(2)}(k_3,k_4) \chi_{NR(1,2)}^{(4)}
\lesssim \langle k_2 \rangle^{-1}. 
\end{align*}
(Id) Estimate of $M_{7, f}^{(4)}$ and $M_{8, f}^{(4)}$: 
Since $M_{7, f}^{(4)} (k_1, k_2,k_3, k_4)$ is symmetric with $(k_1, k_2)$, we assume $|k_1| \le |k_2|$. 
Then, $(k_1, k_2, k_3, k_4) \in \supp \chi_{NR(1,1)}^{(4)} (1- \chi_{H1}^{(4)}) (1-\chi_{A4}^{(4)})$ implies that 
$|k_2| \sim |k_{3,4}| \sim |k_4|=k_{\max}$. 
Thus, by (\ref{pest05}) with $\chi^{(4)}=\chi_{NR(1,1)}^{(4)} (1-\chi_{H1}^{(4)}) (1-\chi_{A4}^{(4)})$, we have $|M_{7, f}^{(4)}| \lesssim 1$. 
Since $(k_1, k_2, k_3, k_4) \in \supp \chi_{NR(1,1)}^{(4)} (1- \chi_{A4}^{(4)}) $ implies that 
$|k_3| \le \max \{  |k_1|, |k_2| \}$, by (\ref{pest06}) with $\chi^{(4)}=\chi_{NR(1,1)}^{(4)} (1-\chi_{A4}^{(4)})$, we have $|M_{8, f}^{(4)}| \lesssim 1$. \\
(Ie) Estimate of $M_{9, f}^{(4)}$: 
Since $(k_1, k_2, k_3, k_4) \in \supp [\chi_{\le L}^{(3)} ]_{ext1}^{(4)} \chi_{NR(1,1)}^{(4)}$ implies that $|k_{3,4}| \sim |k_4|=k_{\max} \le 4L/3$, 
by Lemma~\ref{Le2} and Remark~\ref{rem_sym}, we have 
\begin{align*}
|M_{9, f}^{(4)}| \lesssim \frac{1}{|k_1| |k_2| |k_{3,4}|^2 } [\chi_{NR3}^{(3)}]_{ext1}^{(4)} 
|k_{3,4}| |k_4|^2 [\chi_{NR1}^{(2)}]_{ext2}^{(4)}  [\chi_{\le L}^{(3)}]_{ext1}^{(4)} \chi_{NR(1,1)}^{(4)}  \lesssim L. 
\end{align*}
(If) Estimate of $M_{10, f}^{(4)}$: Note that $\chi_{NR(1,1)}^{(4)}=[\chi_{H1}^{(3)}]_{ext1}^{(4)} \chi_{NR(1,1)}^{(4)}$. 
In a similar manner as (\ref{pest25}), by Lemma~\ref{Le2} and Remark~\ref{rem_sym}, it follows that 
\begin{align*}
\Big| \Big[ \Big( \frac{q_2^{(3)} }{ \Phi_{f}^{(3)}}- \frac{ q_2^{(3)} }{ \Phi_0^{(3)} }  \Big) 
\chi_{NR3}^{(3)} \chi_{>L}^{(3)}  \Big]_{ext1}^{(4)} \chi_{NR(1,1)}^{(4)} \Big|
\lesssim \frac{|\be E_0(f)|}{ |k_1| |k_2| |k_{3,4}|^4} [ \chi_{NR3}^{(3)}  ]_{ext1}^{(4)} \chi_{NR(1,1)}^{(4)}. 
\end{align*}
Thus, by $|\be E_0(f)| \lesssim L$, we have $|M_{10, f}^{(4)}| \lesssim L $. 
Therefore, we obtain (\ref{pwb11}) for $(N, j) \in \{  (3,8), (3,9), (4,4), (4,7), (4,8), (4,9), (4,10) \}$. \\
(II) Secondly, we show (\ref{pwb12}) and (\ref{pwb13}). \\
\noindent
(IIa) Estimate of $M_{6, f}^{(4)}$: 
Since $M_{6,f}^{(4)} (k_1, k_2, k_3, k_4)$ is symmetric with $(k_1, k_2)$, we assume $|k_1| \le |k_2|$. 
Then either $|k_2|= \max\{ |k_1|, |k_2|, |k_3|  \}$ or $|k_2|= \text{med} \{  |k_1|, |k_2|, |k_3| \}$ holds.  
For $(k_1, k_2, k_3, k_4) \in \supp \chi_{H1}^{(4)} \chi_{R5}^{(4)}$, it follows that 
\begin{align*}
|k_{3,4}| \sim |k_4|=k_{\max}, \hspace{0.5cm} |k_4|^{4/5} \lesssim \max \{ |k_1|, |k_2|, |k_3|  \} \sim \text{med} \{ |k_1|, |k_2|, |k_3|  \}.
\end{align*}
Thus, by (\ref{pest05}) with $\chi^{(4)}=\chi_{H1}^{(4)} (1-\chi_{R1}^{(4)}) \chi_{R5}^{(4)}$, we have
\begin{align*}
& |M_{6, f}^{(4)}|  \lesssim \frac{|k_4|}{|k_1| |k_2|} \chi_{NR1}^{(4)} \chi_{H1}^{(4)} (1-\chi_{R1}^{(4)}) \chi_{R5}^{(4)} 
\lesssim \frac{|k_2|^{1/4}}{|k_1|} \chi_{NR1}^{(4)} \chi_{H1}^{(4)} (1-\chi_{R1}^{(4)}) \chi_{R5}^{(4)} \\
& \lesssim \langle \min\{ |k_1|, |k_2|, |k_3|\} \rangle^{-1} 
\langle \max \{ |k_1|, |k_2|, |k_3| \}   \rangle^{1/8} \langle \text{med} \{ |k_1|, |k_2|, |k_3| \}  \rangle^{1/8}  \chi_{H1}^{(4)} \chi_{R5}^{(4)}. 
\end{align*}
(IIb) Estimate of $M_{11,f}^{(4)}$: 
We notice that $\chi_{H1}^{(3)} (k_{3,4}, k_2, k_1) \chi_{NR(2,2)}^{(4)} (k_1, k_2, k_3, k_4) =\chi_{NR(2,2)}^{(4)} (k_1, k_2, k_3, k_4)$. 
In a similar manner as (\ref{net24}), it follows that
\begin{align*}
[ [3q_2^{(3)} \chi_{H1}^{(3)} ]_{sym}^{(3)} ]_{ext1}^{(4)} \chi_{NR(2,2)}^{(4)} (k_1, k_2, k_3, k_4)
= q_2^{(3)} (k_{3,4}, k_2, k_1 ) \chi_{NR(2,2)}^{(4)}(k_1, k_2, k_3, k_4), 
\end{align*}
which implies that 
\begin{align*} 
M_{11, f}^{(4)}= - 2 \big( \frac{q_2^{(3)}}{\Phi_f^{(3)}} \chi_{NR3}^{(3)} \chi_{>L}^{(3)} \big)(k_{3,4}, k_2, k_1) (Q^{(2)} \chi_{NR1}^{(2)} ) (k_3, k_4) 
\chi_{NR(2,2)}^{(4)} (k_1, k_2, k_3, k_4).
\end{align*}
Thus, by the definition of $\chi_{NR(2,2)}^{(4)}$ and Lemma~\ref{Le2}, we have  
\begin{align*}
|M_{11, f}^{(4)}| & \lesssim 
\frac{1}{ |k_2| |k_{3,4}| |k_1|^2}  \chi_{NR3}^{(3)} (k_{3,4}, k_2, k_1)  \, |k_{3,4}| |k_4|^2  \chi_{NR(2,2)}^{(4)} (k_1, k_2, k_3, k_4) \\
& \lesssim \langle k_1  \rangle^{-2} \langle k_4 \rangle^2 \chi_{NR(2,2)}^{(4)} (k_1, k_2, k_3, k_4). 
\end{align*}
\noindent
(III) Finally, we show (\ref{pwb15}) for $(N,j) \in \{  (3,3), (3,5), (3,7), (3,10), (3,11), (3,12) \}$. \\
(IIIa) Estimate of $M_{3,f}^{(3)}$: 
$(k_1, k_2, k_3) \in \supp \chi_{R1}^{(3)} (1-\chi_{H1}^{(3)})$ implies that $k_{1,2}=0$ and $|k_{1,2,3}| \le 16 |k_1|$. 
Thus, we have $|M_{3, f}^{(3)}| \lesssim \langle k_1 \rangle  \chi_{R1}^{(3)} (k_1, k_2, k_3)  $.\\
(IIIb) Estimate of $M_{7, f}^{(3)}$: Since $(k_1, k_2, k_3) \in \supp \chi_{NR(1,1)}^{(3)} \chi_{R1}^{(3)}$ means that 
$4|k_1|=4|k_2| < |k_3|=k_{\max}$ and $k_{1,2}=0$, by Lemma~\ref{Le1}, we have 
$|M_{7, f}^{(3)}| \lesssim \langle k_1 \rangle \chi_{R1}^{(3)} (k_1, k_2, k_3) $. \\
(IIIc) Let $(N, j) \in \{ (3,5), (3,10), (3,11), (3,12) \}$. By the definition of $M_{j, f}^{(N)}$ and (\ref{pest01}), we have
\begin{align} \label{pwb151}
|M_{j, f}^{(N)}| \lesssim \frac{ |k_2|^2+ |k_3|^2}{|k_1|} \chi_{NR2}^{(3)} (1- \chi_{NR1}^{(3)}) (1-[ 3 \chi_{H1}^{(3)}]_{sym}^{(3)}) 
(k_1, k_2, k_3).
\end{align}
For $(k_1, k_2,k_3) \in \supp \chi_{NR2}^{(3)} (1- \chi_{NR1}^{(3)}) (1-[3\chi_{H1}^{(3)}]_{sym}^{(3)})$, it follows that 
$k_{1,2} k_{1,3}=0$, $k_{2,3} \neq 0$, $|k_3| \le 16 \max \{ |k_1|, |k_2| \}$ and $|k_2| \le 16 \max \{ |k_1|, |k_3| \}$. 
Thus, we handle the following two cases. 
Firstly, support that $k_{1,2}=0$ holds. 
Then, $(k_1, k_2, k_3) \in \supp  \chi_{NR2}^{(3)} (1- \chi_{NR1}^{(3)}) (1-[3\chi_{H1}^{(3)}]_{sym}^{(3)} ) $ leads that 
$\chi_{R1}^{(3)}  (k_1, k_2, k_3)=1 $ and $|k_3| \le 16 |k_1|=16 |k_2|$. 
Thus, by (\ref{pwb151}), we have $|M_{j, f}^{(N)}| \lesssim \langle k_1 \rangle \chi_{R1}^{(3)} (k_1, k_2, k_3)$. 
Secondly, suppose that $k_{1,3} =0$ holds. Then, $(k_1, k_2, k_3) \in \supp  \chi_{NR2}^{(3)} (1- \chi_{NR1}^{(3)}) (1-[3 \chi_{H1}^{(3)}]_{sym}^{(3)}) $ 
leads that $\chi_{R1}^{(3)} (k_1, k_3, k_2) =1 $ and $|k_2| \le 16 |k_1| = 16|k_3| $. 
Thus, by (\ref{pwb151}), we have $|M_{j, f}^{(3)}| \lesssim \langle k_1 \rangle  \chi_{R1}^{(3)} (k_1, k_3, k_2) $. 
Therefore, we obtain (\ref{pwb15}).   
\end{proof}

Now, we state a property of multipliers $L_{j, \vp}^{(N)}$ and $M_{j, \vp}^{(N)}$ 
which is needed to prove main estimates as in Section 7.

\begin{lem} \label{rem_pwb3}
Let $f, g \in L^2 (\T)$ and $L_{j, f}^{(N)}$ and $M_{j, f}^{(N)}$ be as in Proposition~\ref{prop_NF2}. 
Then, for $L \gg \max \{ 1, |\beta E_0(f)|, |\beta E_0(g)|  \}$ and $(N, j) \in J_1 \cup J_2 \cup J_3$, we have
\begin{align} \label{pwb21}
|(L_{j, f}^{(N)}- L_{j, g}^{(N)} ) \chi_{>L}^{(N)} | \lesssim |\be| |E_0(f)- E_0(g)| ( |L_{j, f}^{(N)} \chi_{>L}^{(N)} |+|L_{j, g}^{(N)}  \chi_{>L}^{(N)} | ). 
\end{align}
When $(N, j) \in \{  (3,9), (3,10), (3,11), (3,12), (4,4), (4,10), (4,11) \}$, we have 
\begin{align} \label{pwb22}
|M_{j, f}^{(N)} - M_{j, g}^{(N)}| \lesssim |\be| |E_0(f)- E_0(g)| |M_{j, f}^{(N)}|. 
\end{align}
\end{lem}
\begin{proof}
For any $3$-multiplier $M^{(3)}$ and $4$-multiplier $M^{(4)}$, by Lemmas~\ref{Le1} and \ref{Le2}, 
\begin{align} \label{pwb31}
& \Big| \Big[  \Big( \frac{1}{ \Phi_f^{(2)}}- \frac{1}{\Phi_g^{(2)}}\Big) \chi_{NR1}^{(2)}  \chi_{>L}^{(2)}  \Big]_{ext1}^{(3)} M^{(3)} \Big|
=  \Big[ \Big| \Big( \frac{1}{ \Phi_f^{(2)}}- \frac{1}{\Phi_g^{(2)}}\Big) \chi_{NR1}^{(2)}  \chi_{>L}^{(2)} \Big|  \Big]_{ext1}^{(3)} | M^{(3)} | \nonumber \\
& \hspace{3cm} \lesssim |\be| |E_0(f)- E_0(g)| \, \Big| \Big[ \frac{1}{\Phi_f^{(2)}} \chi_{NR1}^{(2)} \chi_{>L}^{(2)} \Big]_{ext1}^{(3)} M^{(3)} \Big|
\end{align}
and
\begin{align} \label{pwb32}
\Big| \Big[  \Big( \frac{1}{ \Phi_f^{(3)}}- \frac{1}{\Phi_g^{(3)}}\Big) \chi_{NR3}^{(3)}  \chi_{>L}^{(3)}  \Big]_{ext1}^{(4)} M^{(4)} \Big|
\lesssim |\be| |E_0(f)- E_0(g)| \Big| \Big[ \frac{1}{\Phi_f^{(3)}} \chi_{NR3}^{(3)} \chi_{>L}^{(3)} \Big]_{ext1}^{(4)} M^{(4)} \Big|.
\end{align}
By the above inequalities and Remark~\ref{rem_sym}, (\ref{pwb22}) follows. 
Let us prove (\ref{pwb21}). 
For $(N, j) \in \{ (3,3), (3,11), (3,12), (3,13), (3,14), (3,15), (4,4) \}=: K_1$, by Remark~\ref{rem_sym} and (\ref{pwb31})--(\ref{pwb32}), 
\begin{align*} 
| L_{j, f}^{(N)} \Phi_f^{(N)}- L_{j, g}^{(N)} \Phi_g^{(N)}  | \lesssim |\be| |E_0(f) - E_0(g)| |L_{j, f}^{(N)} \Phi_f^{(N)} |.
\end{align*}
By Lemma~\ref{Le2} and Lemma~\ref{Le4}, 
\begin{align*} 
\Big| \Big(  \frac{1}{\Phi_f^{(N)}}- \frac{1}{\Phi_g^{(N)}} \Big) L_{j, g}^{(N)} \Phi_g^{(N)} \chi_{>L}^{(N)}  \Big| 
\lesssim |\be| |E_0(f)- E_0(g)| |L_{j, g}^{(N)} \chi_{>L}^{(N)}|. 
\end{align*}
Thus, we have 
\begin{align} 
| (L_{j, f}^{(N)}-L_{j, g}^{(N)} ) \chi_{>L}^{(N)} | 
& \le \bigg| \frac{L_{j, f}^{(N)} \Phi_f^{(N)}- L_{j, g}^{(N)} \Phi_{g}^{(N)} }{\Phi_f^{(N)}} \chi_{>L}^{(N)} \bigg|
+ \Big| \Big( \frac{1}{\Phi_f^{(N)}}- \frac{1}{\Phi_g^{(N)}} \Big) L_{j, g}^{(N)} \Phi_g^{(N)} \chi_{>L}^{(N)}  \Big| \notag \\ 
& \lesssim |\be| |E_0(f)-E_0(g)| (|L_{j, f}^{(N)} \chi_{>L}^{(N)}|+ |L_{j, g}^{(N)} \chi_{>L}^{(N)} |). \label{pwb33}
\end{align}
For $(N, j) \in J_1 \cup J_2 \cup J_3 \setminus K_1$, it follows that $L_{j, f}^{(N)} \Phi_f^{(N)}=L_{j, g}^{(N)} \Phi_g^{(N)}$. 
Then, by Lemmas~\ref{Le1}--\ref{Le3}, we have 
\begin{align} \label{pwb34}
|( L_{j, f}^{(N)}-L_{j, g}^{(N)} ) \chi_{>L}^{(N)}|& = \Big| \Big( \frac{1}{\Phi_f^{(N)}}- \frac{1}{\Phi_g^{(N)}} \Big) L_{j, g}^{(N)} \Phi_g^{(N)} \chi_{>L}^{(N)}  \Big| \notag\\
&\lesssim |\be| |E_0(f)-E_0(g)| |L_{j, g}^{(N)} \chi_{>L}^{(N)}|. 
\end{align}
Collecting (\ref{pwb33}) and (\ref{pwb34}), we obtain (\ref{pwb21}). 
\end{proof}

\section{main estimates}

The main estimates of the proof of Theorem~\ref{thm_main} are as below.
 
\begin{prop} \label{prop_main1}
Let $s \ge 1$, $f, g \in L^2(\T)$ and $F_{f, L}$ and $G_{f, L}$ be as in Proposition~\ref{prop_NF2}. 
Then, for $v, w \in C([-T, T]: H^s(\T))$ and $L \gg \max\{ 1, |\be E_0(f)|, | \be E_0(g)| \}$, we have 
\begin{align}
& \big\| F_{f, L} (\ha{v}) - F_{f, L} (\ha{w}) \big\|_{L_T^{\infty} l_s^2} 
\lesssim  L^{-1} (1+ \| v \|_{L_T^{\infty} H^s}+ \| w \|_{L_T^{\infty} H^s} )^3 \| v- w \|_{L_T^{\infty} H^s}, \label{mes11} \\
& \big\| G_{f, L} (\ha{v}) - G_{f, L} (\ha{w}) \|_{L_T^{\infty} l_s^2} 
\lesssim L^3 (1+ \| v \|_{L_T^{\infty} H^s}+ \| w \|_{L_T^{\infty} H^s} )^5 \| v- w \|_{L_T^{\infty} H^s}. \label{mes12} 
\end{align}
Moreover, it follows that 
\begin{align}
& \big\| F_{f, L} (\ha{v}) - F_{g, L} (\ha{v}) \big\|_{L_T^{\infty} l_s^2} \le C_*,  \label{mes13} \\
& \big\| G_{f, L} (\ha{v})- G_{g, L} (\ha{v}) \big\|_{L_T^{\infty} l_s^2} \le C_* \label{mes14}
\end{align}
where $C_*=C_*(v,s,|E_0(f)-E_0(g)|,|\be|,T,L)\ge 0$ and $C_*\to 0$ when $|E_0(f)-E_0(g)|\to 0$.
\end{prop}
As a corollary of Proposition~\ref{prop_main1}, we obtain the following estimates. 
\begin{cor} \label{cor_mainest}
Let $s\ge 1$. Then there exists a constant $C>0$ such that the following estimates hold 
for any $T>0$, $\vp_1 , \vp_2 \in H^s(\T)$ with $E_0(\vp_1)=E_0(\vp_2)$, $L \gg \max\{ 1, | \be E_0(\vp_1)| \}$ 
and any solution $u_1\in C([-T,T]:H^s(\T))$ (resp. $u_2\in C([-T,T]:H^s(\T))$) to 
\eqref{5KdV3} with initial data $\vp_1$ (resp. $\vp_2 $):
\EQS{
\| u_1 \|_{L_T^{\infty} H^s}   &\leq  \| \vp_1 \|_{H^s} + C L^{-1} (1+\|  u_1 \|_{L_T^{\infty} H^s} )^3 \| u_1 \|_{L_T^{\infty} H^s} \notag\\
& + C T L^3 (1+\| u_1 \|_{L_T^{\infty} H^s })^5  \| u_1 \|_{L_T^{\infty} H^s },  \label{mes31} \\
\| u_1- u_2 \|_{L_T^{\infty} H^s }  &\le \| \vp_1-\vp_2 \|_{H^s} \notag\\
& + CL^{-1} (1+\| u_1 \|_{L_T^{\infty} H^s } + \| u_2 \|_{L_T^{\infty} H^s } )^{3}\| u_1 - u_2\|_{L_T^{\infty} H^s } \nonumber \\
& + CT L^{3} (1+\| u_1 \|_{L_T^{\infty} H^s } + \| u_2 \|_{L_T^{\infty} H^s } )^5\| u_1 - u_2\|_{L_T^{\infty} H^s }. \label{mes32}
}
\end{cor}
\begin{proof}[Proof of Corollary \ref{cor_mainest}]
By proposition~\ref{prop_NF2}, 
$\ha{v}_j(t,k)=e^{-t \phi_{\vp_j}(k)} \ha{u}_j (t,k)$ satisfies (\ref{NF21}) with initial data $\ha{\vp}_j$ 
for each $k \in \Z$ and any $t \in [-T, T]$. Thus, it follows that 
\EQQ{
& \Big[ \ha{v}_1(t',k )+ F_{\vp_1,L} (\ha{v}_1) (t',k)  \Big]_0^t= \int_0^t G_{\vp_1, L} (\ha{v}_1) (t',k) \, dt', \\
& \Big[(\ha{v}_1-\ha{v}_2)(t',k)+ \big(F_{\vp_1,L}(\ha{v}_1 ) - F_{\vp_2,L}(\ha{v}_2  ) \big) (t',k)  \Big]_0^t \\
& \hspace{2.5cm} =\int_0^t \big( G_{\vp_1,L}(\ha{v}_1)-G_{\vp_2,L}(\ha{v}_2)  \big) (t',k)\, dt'
}
for any $t \in [-T, T]$, which leads that 
\EQS{
\| v_1 \|_{L^{\infty}_T H^s} \leq & \| \vp_1 \|_{H^s} + 2\| F_{\vp_1, L} (\ha{v}_1) \|_{L_T^{\infty} l_s^2}
+ T \| G_{\vp_1,L}(\ha{v}_1) \|_{L_T^{\infty} l_s^2 },  \label{6es1}\\
\|v_1-v_2\|_{L^\infty_T H^s}
& \le \|\vp_1-\vp_2\|_{H^s}+2 \|F_{\vp_1,L}(\ha{v}_1)-F_{\vp_2,L}( \ha{v}_2)\|_{L_T^{\infty} l_s^2}  \notag \\ 
& \hspace{0.3cm} +T \|G_{\vp_1,L} (\ha{v}_1)-G_{\vp_2,L}(\ha{v}_2)\|_{L_T^{\infty} l_s^2}. \label{6es2}
}
Note that $F_{\vp_1,L}=F_{\vp_2,L}$ and $G_{\vp_1,L}=G_{\vp_2,L}$ since $E_0(\vp_1)=E_0(\vp_2)$.
Applying (\ref{mes11}) and (\ref{mes12}) to (\ref{6es1}) and (\ref{6es2}), we have
\EQS{
& \| v_1 \|_{L^\infty_T H^s}  \le  \| \vp_1 \|_{H^s}
+ C L^{-1} (1 +\| v_1\|_{L_T^{\infty} H^s } )^3 \| v_1\|_{L_T^{\infty} H^s } \notag \\
& \hspace{2cm} +C T L^3(1+\| v_1\|_{L_T^{\infty }H^s } )^{5}\| v_1 \|_{L_T^{\infty} H^s}, \label{6es3} \\
& \| v_1- v_2 \|_{L_T^{\infty} H^s}   \le \| \vp_1- \vp_2 \|_{H^s}  \notag \\
& \hspace{0.5cm} 
+ CL^{-1} (1+ \| v_1\|_{L_T^{\infty} H^s} + \| v_2\|_{L_T^{\infty} H^s})^3\| v_1 -v_2 \|_{L_T^{\infty} H^s} \notag \\
& \hspace{0.5cm}
+ CT L^3  (1+ \| v_1\|_{L_T^{\infty} H^s} + \| v_2\|_{L_T^{\infty} H^s})^5 \| v_1 -v_2 \|_{L_T^{\infty} H^s}. \label{6es4}
}
By $\| u_1 \|_{L_T^{\infty} H^s}= \|  v_1 \|_{L_T^{\infty} H^s }$ and $\| u_1- u_2 \|_{L_T^{\infty} H^s} = \| v_1- v_2 \|_{L_T^{\infty}H^s}$, we obtain \eqref{mes31} and \eqref{mes32}.
\end{proof}
We also have the following corollary. 
\begin{cor} \label{cor_mainest2}
Let $s\ge 1$. Then, the following estimate holds
for any $T>0$, $\vp, \vp_n \in H^s(\T)$, $L \gg \max\{ 1, | \be E_0(\vp)|, |\be E_0(\vp_n)|  \}$ 
and any solution $u\in C([-T,T]:H^s(\T))$ (resp. $u_n\in C([-T,T]:H^s(\T))$) to 
\eqref{5KdV3} with initial data $\vp$ (resp. $\vp_n $):
\EQQ{
\| u- u_n \|_{L_T^{\infty} H^s }  &\lec \| \vp-\vp_n \|_{H^s} +(1+T)C_*\notag\\
& + L^{-1} (1+\| u \|_{L_T^{\infty} H^s } + \| u_n \|_{L_T^{\infty} H^s } )^{3}(\| u - u_n\|_{L_T^{\infty} H^s }+C_*)\\
& + T L^{3} (1+\| u \|_{L_T^{\infty} H^s } + \| u_n \|_{L_T^{\infty} H^s } )^5(\| u - u_n\|_{L_T^{\infty} H^s }+C_*)
}
where $C_*=C_*(u,s,|E_0(\varphi)-E_0(\varphi_n)|,|\be|,T,L) \ge 0$ and $C_*\to 0$ when $|E_0(\varphi)-E_0(\varphi_n)|\to 0$.
\end{cor}
\begin{proof}[Proof of Corollary \ref{cor_mainest2}]
In the same manner as the proof of Corollary \ref{cor_mainest}, we have
\EQQ{
&\|v-v_n\|_{L^\infty_T H^s}\\
& \le \|\vp-\vp_n\|_{H^s}+2\|F_{\vp,L}(\ha{v})-F_{\vp_n,L}( \ha{v}_n)\|_{L_T^{\infty} l_s^2}
+T \|G_{\vp,L} (\ha{v})-G_{\vp_n,L}(\ha{v}_n)\|_{L_T^{\infty} l_s^2}\\
& \le \|\vp-\vp_n\|_{H^s}+2(\|F_{\vp,L}(\ha{v})-F_{\vp_n,L}( \ha{v})\|_{L_T^{\infty} l_s^2}+\|F_{\vp_n,L}(\ha{v})-F_{\vp_n,L}( \ha{v}_n)\|_{L_T^{\infty} l_s^2} )\\ 
& \hspace{0.3cm} +T (\|G_{\vp,L} (\ha{v})-G_{\vp_n,L}(\ha{v})\|_{L_T^{\infty} l_s^2}+\|G_{\vp_n,L} (\ha{v})-G_{\vp_n,L}(\ha{v}_n)\|_{L_T^{\infty} l_s^2})
}
where $\ha{v}(t,k)=e^{-t \phi_{\vp}(k)} \ha{u} (t,k)$ and $\ha{v}_n(t,k)=e^{-t \phi_{\vp_n}(k)} \ha{u}_n (t,k)$.
By \eqref{mes11} and \eqref{mes12}, we have
\EQQS{
&\|F_{\vp_n,L}(\ha{v})-F_{\vp_n,L}( \ha{v}_n)\|_{L_T^{\infty} l_s^2} \lec L^{-1} (1+ \| v\|_{L_T^{\infty} H^s} + \| v_n\|_{L_T^{\infty} H^s})^3\| v -v_n \|_{L_T^{\infty} H^s}\\
&\|G_{\vp_n,L} (\ha{v})-G_{\vp_n,L}(\ha{v}_n)\|_{L_T^{\infty} l_s^2} \lec L^3  (1+ \| v\|_{L_T^{\infty} H^s} + \| v_n\|_{L_T^{\infty} H^s})^5 \| v -v_n \|_{L_T^{\infty} H^s}.
}
By \eqref{mes13} and \eqref{mes14},
$\|F_{\vp,L}(\ha{v})-F_{\vp_n,L}( \ha{v})\|_{L_T^{\infty} l_s^2}$ and 
$\|G_{\vp,L} (\ha{v})-G_{\vp_n,L}(\ha{v})\|_{L_T^{\infty} l_s^2}$ are bounded by $C_*$.
We notice that it follows that 
\begin{align*}
& \| u- u_n \|_{L_T^{\infty} H^s} \le \| v- v_n \|_{L_T^{\infty}  H^s }+ C_*\\
& \| v- v_n \|_{L_T^{\infty} H^s} \le \| u- u_n \|_{L_T^{\infty}  H^s }+ C_*
\end{align*}
by uniform $l^2_s$-continuity of $\ha{u}$ and $\ha{v}$, Lebesgue's dominated convergence theorem and
\begin{align*}
& (\ha{u}- \ha{u}_n) (t, k)=e^{t \phi_{\vp_n}(k)} (\ha{v}- \ha{v}_n)(t,k)+(e^{t \phi_{\vp}(k)}- e^{t \phi_{\vp_n} (k)}) \ha{v}(t,k), \\
& (\ha{v}- \ha{v}_n) (t, k)=e^{-t \phi_{\vp_n} (k)} (\ha{u}- \ha{u}_n)(t,k)+(e^{-t \phi_{\vp} (k)}- e^{-t \phi_{\vp_n} (k)}) \ha{u}(t,k).
\end{align*}
\end{proof}
Before we prove Proposition~\ref{prop_main1}, we prepare some lemmas. 
The proofs of them are based on the nonlinear estimates and the pointwise upper bounds of multipliers 
stated in Section 5 and 6. 

\begin{lem} \label{lem_nes0} 
Let $s> 1/2$, $f, g \in L^2(\T)$ and $L_{j, f}^{(N)}$ with $(N,j) \in J_1 \cup J_2 \cup J_3$ be as in Proposition~\ref{prop_NF2}. 
Then, for any $v, w \in C([-T, T]: H^s(\T))$ and $L \gg \max \{1, |\beta E_{0} (f) |, | \beta E_0(g)| \}$, we have
\begin{align}
&\big\| \Lambda_{f}^{(N)} ( \ti{L}_{j, f}^{(N)} \chi_{>L}^{(N)} , \ha{v}  ) \big\|_{L_T^{\infty} l_s^2} 
\lesssim L^{-1} \| v \|_{L_T^{\infty} H^s}^N, \label{ne01} \\
& \big\| \La_{f}^{(N)}(\ti{L}_{j, f}^{(N)} \chi_{>L}^{(N)}, \ha{v}) 
 - \La_{f}^{(N)} ( \ti{L}_{j, f}^{(N)} \chi_{>L}^{(N)}, \ha{w} ) \big\|_{ L_T^{\infty} l_s^2} \notag \\
& \hspace{3cm} \lesssim L^{-1} (\| v \|_{L_T^{\infty} H^s}+ \| w \|_{L_T^{\infty} H^s})^{N-1} \| v-w  \|_{L_T^{\infty} H^s}. \label{ne02}
\end{align}
Moreover, it follows that 
\begin{align}
\big\|  \La_f^{(N)} ( \ti{L}_{j, f}^{(N)} \chi_{>L}^{(N)}, \ha{v}) - \La_{g}^{(N)} ( \ti{L}_{j, g}^{(N)} \chi_{>L}^{(N)}, \ha{v} )  \big\|_{L_T^{\infty}l_s^2} 
\le C_* \label{ne03}
\end{align}
where $C_*=C_*(v,s,|E_0(f)-E_0(g)|,|\be|,T) \ge 0$ and $C_*\to 0$ when $|E_0(f)-E_0(g)|\to 0$.
\end{lem}
\begin{proof} 
By (\ref{pwb3})--(\ref{pwb2}), 
\begin{align}  \label{net10}
| \ti{L}_{j, f}^{(N)} \chi_{>L}^{(N)} | \le [ |L_{j, f}^{(N)} \chi_{>L}^{(N)} | ]_{sym}^{(N)} \lesssim 
\langle k_{\max} \rangle^{-1} \chi_{>L}^{(N)} \lesssim L^{-1}.
\end{align}
Thus, it follows that 
\begin{align} \label{go11}
\Big\| \sum_{k=k_{1,\dots, N}} |\ti{L}_{j, f}^{(N)} \chi_{>L}^{(N)} | \, \prod_{l=1}^N |\ha{v}_l (t, k_l)|  \Big\|_{L_T^{\infty} l_s^2 }
\lesssim L^{-1} \prod_{l=1}^N \| v_l \|_{L_T^{\infty} H^s }
\end{align}
for any $\{ v_l \}_{l=1}^N \in C([-T, T] :H^s(\T))$. 
(\ref{go11}) implies that (\ref{ne01}) and (\ref{ne02}) hold. 
Next, we prove (\ref{ne03}). A direct computation yields that  
\begin{align*}
& [ \La_{f}^{(N)} (\ti{L}_{j, f}^{(N)} \chi_{>L}^{(N)}, \ha{v} )- \La_{g}^{(N)} (\ti{L}_{j, g}^{(N)} \chi_{>L}^{(N)}, \ha{v})] (t,k) \\
& = \La_{f}^{(N)} ( ( \ti{L}_{j, f}^{(N)}-\ti{L}_{j,g}^{(N)} ) \chi_{>L}^{(N)}, \ha{v}  ) (t,k)
+ [ \La_f^{(N)} (\ti{L}_{j,g}^{(N)} \chi_{>L}^{(N)}, \ha{v}) - \La_{g}^{(N)} (\ti{L}_{j, g}^{(N)} \chi_{>L}^{(N)}, \ha{v} )] (t,k) \\
& =: J_{1,1}^{(N)}(\ha{v}) (t,k) + J_{1,2}^{(N)} (\ha{v}) (t,k).
\end{align*}
By (\ref{pwb21}) and (\ref{net10}), we have
\begin{align*}
& | ( \ti{L}_{j, f}^{(N)} - \ti{L}_{j,g}^{(N)}) \chi_{>L}^{(N)} | 
\le [ | (L_{j, f}^{(N)} - L_{j,g}^{(N)} ) \chi_{>L}^{(N)} | ]_{sym}^{(N)} \nonumber \\
& \lesssim |\be| |E_0(f)- E_0(g)| \, \big[ |L_{j, f}^{(N)} \chi_{>L}^{(N)}|+|L_{j, g}^{(N)} \chi_{>L}^{(N)} | \big]_{sym}^{(N)}
\lesssim L^{-1} |\be| |E_0(f) - E_0 (g)|,
\end{align*}
which leads that 
\begin{align*}
\| J_{1,1}^{(N)} (\ha{v}) \|_{L_T^{\infty} l_s^2} \lesssim  L^{-1} |\be| |E_0(f) - E_0(g)| \| v \|_{L_T^\infty H^s}^N \le C_*.
\end{align*}
By Lemma~\ref{lem_go} and (\ref{go11}), we have $\| J_{1,2}^{(N)} (\ha{v})  \|_{L_T^{\infty} l_s^2} \le C_*$.
Therefore, we obtain (\ref{ne03}). 
\end{proof}

Next, we give nonlinear estimates for all multipliers $M_{j, \vp}^{(N)}$ defined in 
Proposition~\ref{prop_NF2} except for $M_{2,\vp}^{(3)}$ and $M_{4, \vp}^{(3)}$. 
We now put 
\begin{align*}
I_1=&\{ (3,8), (3,9), (4,4), (4,7), (4,8), (4,9), (4,10) \}, ~I_2=\{  (4,11) \}, \\ 
I_3=&\{ (4,3) \}, 
 \hspace{0.3cm} I_4=\{ (4,2), (5,2) \}, 
 \hspace{0.3cm} I_5=\{ (4,1), (5,1), (6,1) \}, \\
 I_6=& \{  (3,3),  (3,5), (3,6), (3,7), (3,10), (3,11), (3,12) \}, 
\hspace{0.3cm} I_7=\{ (3,1), (4,5), (4,6)  \}.
\end{align*}
\begin{lem} \label{lem_nl10}
Let $s \ge 1$, $f , g \in L^2(\T)$ and $M_{j, f}^{(N)}$ with 
$(N,j) \in I_1 \cup I_2$ be as in Proposition~\ref{prop_NF2}. 
Then, for $v,w \in C([-T, T]: H^s(\T)) $ and $L \gg \max\{ 1, | \be E_0(f)|, |\be E_0(g)| \}$, we have
\begin{align}
& \| \La_{f}^{(N)} (\ti{M}_{j, f}^{(N)}, \ha{v}) \|_{L_T^{\infty} l_s^2} \lesssim L^2 \| v \|_{L_T^{\infty} H^s}^N, \label{nes11} \\
& \|  \La_{f}^{(N)} (\ti{M}_{j, f}^{(N)}, \ha{v}) - \La_f^{(N)} (\ti{M}_{j, f}^{(N)} , \ha{w} )  \|_{L_T^{\infty} l_s^2} \notag \\
& \hspace{3cm} \lesssim L^2 (\| v \|_{L_T^{\infty} H^s}+ \| w \|_{L_T^{\infty} H^s})^{N-1}  \| v-w  \|_{L_T^{\infty} H^s}. \label{nes12} 
\end{align}
Moreover, it follows that 
\begin{align}
\|  \La_{f}^{(N)} (\ti{M}_{j, f}^{(N)}, \ha{v}) - \La_g^{(N)} (\ti{M}_{j, g}^{(N)} , \ha{v}) \|_{L_T^{\infty} l_s^2} 
\le C_* \label{nes13} 
\end{align}
where $C_*=C_*(v,s,|E_0(f)-E_0(g)|,|\be|,T, L) \ge 0$ and $C_*\to 0$ when $|E_0(f)-E_0(g)|\to 0$.
\end{lem}

\begin{proof} 
We first show 
\begin{align}
\Big\|  \sum_{k=k_{1,\dots,N}} |M_{j, f}^{(N)}| \, \prod_{l=1}^N | \ha{v}_l(t, k_l)| \Big\|_{L_T^{\infty} l_s^2 } 
\lesssim L^2 \prod_{l=1}^N \| v_l \|_{L_T^{\infty} H^s } \label{go21}
\end{align}
for $(N, j) \in I_1 \cup I_2$ and $\{ v_l \}_{l=1}^N \subset C([-T, T]:H^s(\T))$. 
For $(N, j) \in I_1$, by (\ref{pwb1}), we get (\ref{go21}). 
Now, we prove (\ref{go21}) for $(N, j) \in I_2$. 
Since $(k_1, k_2, k_3, k_4) \in \supp \chi_{NR(2,2)}^{(4)}$ implies that $|k_1| \sim |k_{1,2,3,4}|$ and $|k_3| \sim |k_4|$, 
by (\ref{pwb13}), it follows that 
\begin{align*}
\langle k_{1,2,3,4} \rangle^{s} |M_{11, f}^{(4)}(k_1, k_2, k_3, k_4)| 
& \lesssim \langle k_{1,2,3,4} \rangle^{-1} \langle k_1 \rangle^{s-1} \langle  k_3 \rangle^{s} \langle k_4 \rangle^{s}
 \langle k_4 \rangle^{-2s+2} \\
& \lesssim \langle k_{1,2,3,4} \rangle^{-1} \langle k_1  \rangle^{s-1} \langle k_3 \rangle^{s} \langle k_4  \rangle^{s} .
\end{align*}
Thus, by the Young inequality and the H\"{o}lder inequality, we have 
\begin{align*}
& \Big\| \sum_{k=k_{1,2,3,4}} |M_{11,f}^{(4)}| \, \prod_{l=1}^4 |\ha{v}_l (t, k_l)|   \Big\|_{l_s^2} \\
& \hspace{0.5cm} \lesssim \big\| \langle \cdot \rangle^{-1} \big( \langle \cdot \rangle^{s-1} |\ha{v}_1(t)|* |\ha{v}_2(t)|*
\langle \cdot \rangle^s |\ha{v}_3(t)|* \langle \cdot \rangle^{s} |\ha{v}_4(t)|  \big)     \big\|_{l^2} \\
&\hspace{0.5cm} \lesssim \| \langle \cdot \rangle^{s-1} |\ha{v}_1(t)| \|_{l^1} \| \ha{v}_2 (t) \|_{l^1} 
\| \langle \cdot \rangle^{s} |\ha{v}_3(t)|* \langle \cdot \rangle^s |\ha{v}_4(t)| \|_{l^{\infty}}
\lesssim \prod_{l=1}^4 \| v_l(t) \|_{H^s}
\end{align*}
for any $t \in [-T, T]$. This implies (\ref{go21}). 
Therefore,  we obtain (\ref{go21}) for $(N, j) \in I_1 \cup I_2$. 
By (\ref{go21}), it follows that 
\begin{align}
\Big\|  \sum_{k=k_{1,\dots,N}} |\ti{M}_{j, f}^{(N)}| \, \prod_{l=1}^N |\ha{v}_l(t, k_l)| \Big\|_{L_T^{\infty} l_s^2 } 
\lesssim L^2 \prod_{l=1}^N \| v_l \|_{L_T^{\infty} H^s }, \label{go22}
\end{align}
which leads that (\ref{nes11}) and (\ref{nes12}) hold.  

Next, we show (\ref{nes13}). Put $K_2:=\{ (3,9), (4,4), (4,10), (4,11) \}$. 
For $(N, j) \in K_2$, it follows that 
\begin{align*}
& [\La_f^{(N)} (\ti{M}_{j, f}^{(N)} , \ha{v})- \La_g^{(N)} (\ti{M}_{j, g}^{(N)}, \ha{v} )] (t,k) \\
&= \La_f^{(N)} (\ti{M}_{j, f}^{(N)} - \ti{M}_{j, g}^{(N)}, \ha{v}) (t,k)
 + [ \La_f^{(N)} (\ti{M}_{j, g}^{(N)}, \ha{v})- \La_g^{(N)} (\ti{M}_{j, g}^{(N)}, \ha{v})] (t,k)\\
& =: J_{2,1}^{(N)} (\ha{v})(t,k)+ J_{2,2}^{(N)} (\ha{v}) (t,k).
\end{align*}
By (\ref{pwb22}), 
 \begin{align*}
|\ti{M}_{j, f}^{(N)}- \ti{M}_{j,g}^{(N)}| \le [ | M_{j, f}^{(N)} - M_{j,g}^{(N)}| ]_{sym}^{(N)} 
 \lesssim |\be| |E_0(f) - E_0(g)| \, [ |M_{j,f}^{(N)}| ]_{sym}^{(N)}. 
\end{align*}
Thus, by (\ref{go21}), we get 
\begin{align*}
\| J_{2,1}^{(N)} (\ha{v}) \|_{L_T^{\infty} l_s^2} \lesssim |\be| |E_0(f) - E_0(g)| L^2 \| v \|_{L_T^\infty H^s}^N \le C_*.
\end{align*}
By Lemma~\ref{lem_go} and (\ref{go22}), 
we have $ \|J_{2,2}^{(N)} (\ha{v}) \|_{L_T^{\infty} l_s^2} \le C_* $. 
Hence, we obtain (\ref{nes13}) with $(N, j) \in K_2$.
Now, we handle the case $(N, j) \in I_1 \cup I_2 \setminus K_2$. 
Then, it follows that $M_{j, f}^{(N)}= M_{j, g}^{(N)}$, which implies that 
\begin{align*}
[ \La_f^{(N)}( \ti{M}_{j, f}^{(N)}, \ha{v} )- \La_g^{(N)} (\ti{M}_{j, g}^{(N)}, \ha{v}) ] (t,k)
=[\La_f^{(N)} (\ti{M}_{j, g}^{(N)}, \ha{v}) -  \La_g^{(N)} (\ti{M}_{j, g}^{(N)}, \ha{v}) ] (t, k).
\end{align*}
Thus, by Lemma~\ref{lem_go} and (\ref{go22}), we obtain (\ref{nes13}). 
\end{proof}

\begin{lem} \label{lem_nl11}
Let $s \ge 1$, $f, g \in L^2(\T)$ and $M_{j, f}^{(N)}$ with 
$(N,j) \in I_3 \cup I_4 \cup I_5 $ be as in Proposition~\ref{prop_NF2}. 
Then, for any $v, w \in C([-T, T]: H^s(\T))$ and \\
$L \gg \max\{ 1, | \be E_0(f)|, |\be E_0(g)| \}$, we have 
\begin{align}
& \| \La_{f}^{(N)} (\ti{M}_{j, f}^{(N)}, \ha{v} ) \|_{L_T^{\infty} l_s^2} \lesssim  \| v \|_{L_T^{\infty} H^s}^N, \label{nle41} \\
& \|  \La_{f}^{(N)} (\ti{M}_{j, f}^{(N)}, \ha{v} ) - \La_f^{(N)} (\ti{M}_{j, f}^{(N)} , \ha{w} )  \|_{L_T^{\infty} l_s^2} \notag \\
& \hspace{1cm} \lesssim  (\| v \|_{ L_T^{\infty} H^s}+ \| w \|_{L_T^{\infty} H^s})^{N-1}  \| v-w  \|_{L_T^{\infty} H^s}. \label{nle42} 
\end{align}
Moreover, it follows that 
\begin{align} \label{nle43}
\|  \La_{f}^{(N)} (\ti{M}_{j, f}^{(N)}, \ha{v} ) - \La_g^{(N)} (\ti{M}_{j, g}^{(N)} , \ha{v} )  \|_{L_T^{\infty} l_s^2} \le C_*
\end{align}
where $C_*=C_*(v,s,|E_0(f)-E_0(g)|,|\be|,T) \ge 0$ and $C_*\to 0$ when $|E_0(f)-E_0(g)|\to 0$.
\end{lem}

\begin{proof}
(I) Firstly, we handle the case $(N, j) \in I_3=\{ (4,3) \}$. 
By the proof  of Lemma~\ref{lem_pwb1} (III), $L_{j,f}^{(N)} \Phi_{f}^{(N)}$ with $(N,j) \in J_2$ satisfies (\ref{ml31}). 
Thus, by the definition of $M_{3, f}^{(4)}$ and Lemma~\ref{lem_nl5} (2), we have
\begin{align} \label{nle413}
\Big\| \sum_{k=k_{1,2,3,4}} |\ti{M}_{3,f}^{(4)}| \, \prod_{l=1}^4 |\ha{v}_l (t, k_l)| \Big\|_{L_T^{\infty} l_s^2}
 \lesssim \prod_{l=1}^4 \| v_l \|_{L_T^{\infty} H^s}
\end{align}
for any $\{ v_l \}_{l=1}^4 \subset C([-T, T]: H^s(\T))$. This implies (\ref{nle41}) and (\ref{nle42}). 
Now, we prove (\ref{nle43}). 
For $(N, j) \in J_2$, by (\ref{ml320}) and (\ref{pwb21}), it follows that 
\begin{align*}
& |(\ti{L}_{j, f}^{(3)} -\ti{L}_{j, g}^{(3)} ) \chi_{>L}^{(3)}| \lesssim |\be| |E_0(f) - E_0(g)| \, 
\big[ |L_{j, f}^{(3)} \chi_{>L}^{(3)}|+| L_{j, g}^{(3)} \chi_{>L}^{(3)} | \big]_{sym}^{(3)} \\
& \lesssim  |\be| |E_0(f)-E_0(g)| \langle k_{1,2,3} \rangle^{-1} \langle k_{\max} \rangle^{-1} \La_1^{-1} 
(1-[ 3 \chi_{H1}^{(3)} ]_{sym}^{(3)}) \chi_{NR1}^{(3)} \chi_{>L}^{(3)}. 
\end{align*}
In a similar manner as the proof of (\ref{nl83}), we have 
\begin{align*} 
\big\| \La_f^{(4)} (\ti{M}_{3, f}^{(4)}- \ti{M}_{3, g}^{(4)}, \ha{v}) \big\|_{L_T^{\infty} l_s^2} 
& \lesssim \Big\| \sum_{k=k_{1,2,3,4}} |\ti{M}_{3, f}^{(4)} -\ti{M}_{3,g}^{(4)} | \, \prod_{l=1}^4 |\ha{v}(t, k_l)| \Big\|_{L_T^{\infty} l_s^2} \\
& \lesssim |\be| |E_0(f)- E_0(g)| \| v \|_{L_T^{\infty} H^s}^4 \le C_*.
\end{align*}
By Lemma~\ref{lem_go} and (\ref{nle413}), we have 
\begin{align*} 
\big\| \La_f^{(4)} (\ti{M}_{3, g}^{(4)}, \ha{v}) - \La_g^{(4)} (\ti{M}_{3,g}^{(4)}, \ha{v} ) \big\|_{L_T^{\infty} l_s^2} \le C_*.
\end{align*}
Therefore, we obtain (\ref{nle43}). \\
(II) Secondly, we prove (\ref{nle41})--(\ref{nle43}) with $(N,j) \in I_4$. Put 
\begin{align*}
& \mathcal{N}_{f}^{(2)} (\ha{w}_1, \ha{w}_2) (t,k): = \La_f^{(2)}(-Q^{(2)} \chi_{NR1}^{(2)}, \ha{w}_1(t), \ha{w}_2(t)) (t,k), \\
& \mathcal{N}_{f}^{(2)} (\ha{w}_1) (t,k): = \mathcal{N}_f^{(2)} (\ha{w}_1, \ha{w}_1) (t,k).
\end{align*}
Since $|(Q^{(2)} \chi_{NR1}^{(2)} )(k_1, k_2)| \lesssim \langle k_{1,2} \rangle \langle k_{\max} \rangle^2$, 
by Lemma~\ref{lem_nl2}, we have 
\begin{align} \label{nle412}
\Big\| \sum_{k=k_{1,2}} |Q^{(2)} \chi_{NR1}^{(2)} | \, \prod_{l=1}^2 |\ha{w}_l(t, k_l)| \Big\|_{L_T^{\infty} l_{s-3}^2 } 
\lesssim \| w_1 \|_{L_T^{\infty} H^s} \| w_2 \|_{L_T^{\infty} H^s}
\end{align}
for any $w_1, w_2 \in C([-T, T]: H^s (\T))$, which implies that 
\begin{align}
 \| \mathcal{N}_{f}^{(2)} (\ha{w}_1 , \ha{w}_2 )  \|_{L_T^{\infty} l_{s-3}^2}
 \lesssim \| w_1\|_{L_T^{\infty} H^s } \| w_2 \|_{L_T^{\infty} H^s}.  \label{nle414} 
\end{align}
In a similar way to the proof of Lemma~\ref{lem_go}, $\mathcal{N}_f^{(2)} (\ha{w}_1, \ha{w}_2) \in C([-T, T]: l_{s-3}^2) $. By Lemma~\ref{lem_go} and (\ref{nle412}), we have
\begin{align}
\| \mathcal{N}_{f}^{(2)} (\ha{w}_1)- \mathcal{N}_{g}^{(2)} (\ha{w}_1) \|_{L_T^{\infty} l_{s-3}^2} \le C_*(w_1,s,|E_0(f)-E_0(g)|,|\beta|,T).  \label{nle415} 
\end{align}
By the definition, it follows that 
\begin{align*}
\La_f^{(4)} (\ti{M}_{2,f}^{(4)}, \ha{v})(t,k) &= \sum_{i = 1,3,5,6 ,8 ,9, 10, 11, 13 }
\La_f^{(3)} (3 \ti{L}_{i,f}^{(3)} \chi_{>L}^{(3)}, \ha{v}(t), \ha{v}(t), \mathcal{N}_{f}^{(2)} (\ha{v})(t) )(t,k), \\
\La_f^{(5)} (\ti{M}_{2,f}^{(5)}, \ha{v})(t,k) 
&= \sum_{i=1}^4 \La_f^{(4)} (4 \ti{L}_{i,f}^{(4)} \chi_{>L}^{(4)}, \ha{v}(t), \ha{v}(t), \ha{v}(t), \mathcal{N}_{f}^{(2)} (\ha{v})(t)  )(t,k). 
\end{align*}
Thus, it suffices to show that for $(N, j) \in J_1$ and $v,w \in C([-T, T]: H^s(\T))$
\begin{align}
& \| \La_f^{(N)} ( \ti{L}_{j, f}^{(N)} \chi_{>L}^{(N)}, \ha{v} , \dots, \ha{v}, \mathcal{N}_{f}^{(2)} (\ha{v} ) ) \|_{L_T^{\infty} l_s^2} 
\lesssim \| v \|_{L_T^{\infty} H^s}^{N+1}, \label{nle51} \\
&  \| \La_f^{(N)} ( \ti{L}_{j, f}^{(N)} \chi_{>L}^{(N)}, \ha{v}, \dots, \ha{v}, \mathcal{N}_{f}^{(2)} (\ha{v} ) ) 
- \La_f^{(N)} ( \ti{L}_{j, f}^{(N)} \chi_{>L}^{(N)}, \ha{w}, \dots, \ha{w}, \mathcal{N}_{f}^{(2)} (\ha{w} ) )  \|_{L_T^{\infty} l_s^2} \nonumber \\
& \hspace{0.5cm} \lesssim ( \| v \|_{L_T^{\infty} H^s} + \| w \|_{L_T^{\infty} H^s})^{N} \| v-w \|_{L_T^{\infty} H^s} , \label{nle52} 
\end{align}
and
\begin{align}
\| \La_f^{(N)} ( \ti{L}_{j, f}^{(N)} \chi_{>L}^{(N)}, \ha{v}, \dots, \ha{v}, \mathcal{N}_{f}^{(2)} (\ha{v} ) ) 
- \La_g^{(N)} (\ti{L}_{j,g}^{(N)} \chi_{>L}^{(N)}, \ha{v}, \dots , \ha{v}, \mathcal{N}_{g}^{(2)} (\ha{v}) ) \|_{L_T^{\infty} l_s^2}
\le C_*. \label{nle53}
\end{align}
For $(N,j) \in J_1$, by (\ref{pwb1}) and (\ref{pwb21}), it follows that 
\begin{align}
& | \ti{L}_{j, f}^{(N)} \chi_{>L}^{(N)}  | \lesssim \langle k_{1, \dots, N} \rangle^{-1} \langle k_{\max} \rangle^{-2}, \label{nle416} \\
& | (\ti{L}_{j, f}^{(N)}- \ti{L}_{j, g}^{(N)}) \chi_{>L}^{(N)}  | \lesssim 
|\be| |E_0(f) - E_0(g)| \langle k_{1, \dots, N} \rangle^{-1} \langle k_{\max} \rangle^{-2}. \label{nle417} 
\end{align}
By (\ref{nl41}), (\ref{nle414}) and (\ref{nle416}), we have 
\begin{align} 
& \Big\| \sum_{k=k_{1,\dots, N}} |\ti{L}_{j, f}^{(N)} \chi_{>L}^{(N)}| \, \prod_{l=1}^{N-1} |\ha{v}_l (t,k_l)| \,
\big| \mathcal{N}_f^{(2)} (\ha{v}_N, \ha{v}_{N+1}) (t, k_N)  \big|  \Big\|_{L_T^{\infty} l_s^2} \notag \\ 
& \hspace{0.5cm} \lesssim \prod_{l=1}^{N-1} \| v_l \|_{L_T^{\infty} H^s} 
\big\| \mathcal{N}_f^{(2)} (\ha{v}_N, \ha{v}_{N+1}) \big\|_{L_T^{\infty} l_{s-3}^2}
\label{nle421} \\
& \hspace{0.5cm} \lesssim \prod_{l=1}^{N+1} \| v_l \|_{L_T^{\infty} H^s} \label{nle422}
\end{align}
for any $\{ v_l \}_{l=1}^{N+1} \subset C([-T, T]:H^s(\T))$. 
(\ref{nle422}) implies (\ref{nle51}) and (\ref{nle52}). Next, we prove (\ref{nle53}). 
By a direct computation yields that 
\begin{align*}
&[ \La_f^{(N)} ( \ti{L}_{j, f}^{(N)} \chi_{>L}^{(N)}, \ha{v}, \dots, \ha{v}, \mathcal{N}_{f}^{(2)} (\ha{v} ) ) 
-\La_{g}^{(N)} ( \ti{L}_{j, g}^{(N)} \chi_{>L}^{(N)}, \ha{v}, \dots, \ha{v}, \mathcal{N}_{g}^{(2)} (\ha{v} ) )] (t,k) \\
 & = \La_f^{(N)} ( (\ti{L}_{j, f}^{(N)}- \ti{L}_{j, g}^{(N)} ) \chi_{>L}^{(N)}, \ha{v}, \dots, \ha{v}, \mathcal{N}_{f}^{(2)} (\ha{v} ) (t,k) \\
& +  [\La_f^{(N)} (\ti{L}_{j, g}^{(N)}  \chi_{>L}^{(N)}, \ha{v}, \dots, \ha{v}, \mathcal{N}_{f}^{(2)} (\ha{v} ) )
-  \La_g^{(N)} ( \ti{L}_{j, g}^{(N)} \chi_{>L}^{(N)}, \ha{v}, \dots, \ha{v}, \mathcal{N}_{f}^{(2)} (\ha{v})] (t,k) \\
& +  [\La_g^{(N)} ( \ti{L}_{j, g}^{(N)}  \chi_{>L}^{(N)}, \ha{v}, \dots, \ha{v}, 
\mathcal{N}_{f}^{(2)} (\ha{v})- \mathcal{N}_g^{(2)} (\ha{v}) )] (t,k) \\
&=: J_{3,1}^{(N+1)} (\ha{v}) (t,k)+ J_{3,2}^{(N+1)} (\ha{v}) (t,k) + J_{3,3}^{(N+1)} (\ha{v})(t,k). 
\end{align*}
By (\ref{nl41}), (\ref{nle414}) and (\ref{nle417}), we have
\begin{align} \label{nle511}
\| J_{3,1}^{(N+1)}  (\ha{v})  \|_{L_T^{\infty} l_s^2} \lesssim |\be| |E_0(f)- E_0(g) | \, \| v \|_{L_T^{\infty} H^s}^{N+1} \le C_*.
\end{align}
By Lemma~\ref{lem_go} and (\ref{nle421}), we get
\begin{align} \label{nle512}
\| J_{3,2}^{(N+1)} (\ha{v}) \|_{L_T^{\infty} l_s^2} \le C_*.
\end{align}
By (\ref{nl41}), (\ref{nle415}) and (\ref{nle416}), we have
\begin{align}
\| J_{3,3}^{(N+1)} (\ha{v})  \|_{L_T^{\infty} l_s^2}
\lesssim \| v \|_{L_T^{\infty} H^s}^{N-1} 
\big\| \mathcal{N}_f^{(2)} (\ha{v})- \mathcal{N}_g^{(2)} (\ha{v}) \big\|_{L_T^{\infty} l_{s-3}^2} \le C_*. \label{nle513}
\end{align}
Collecting (\ref{nle511})--(\ref{nle513}), we obtain (\ref{nle53}). \\
(III) Finally, we prove (\ref{nle41})--(\ref{nle43}) with $(N, j) \in I_5$. Put 
\begin{align*}
& \mathcal{N}_{f}^{(3)} (\ha{w}_1, \ha{w}_2, \ha{w}_3) (t,k) := \La_f^{(3)} (-Q^{(3)}, \ha{w}_1(t), \ha{w}_2(t), \ha{w}_3(t) ) (t,k), \\
& \mathcal{N}_{f}^{(3)} (\ha{w}_1)(t, k):= \mathcal{N}_f^{(3)} (\ha{w}_1, \ha{w}_1 , \ha{w}_1)(t,k). 
\end{align*}
By $|Q^{(3)} (k_1, k_2,  k_3)| \lesssim \langle k_{1,2,3} \rangle $, it follows that
\begin{align} \label{nle610}
\Big\| \sum_{k=k_{1,2,3}} |Q^{(3)}| \, \prod_{l=1}^3 | \ha{w}_l (t, k_l) | \Big\|_{L_T^{\infty} l_{s-1}^2 } \lesssim \prod_{l=1}^3 \| w_l \|_{L_T^{\infty} H^s}
\end{align}
for any $\{ w_l \}_{l=1}^3 \subset C([-T, T]: H^s(\T))$, which leads that 
\begin{align} \label{nle611}
\| \mathcal{N}_f^{(3)}(\ha{w}_1, \ha{w}_2, \ha{w}_3 )   \|_{L_T^{\infty} l_{s-1}^2} \lesssim \prod_{l=1}^3 \| w_l \|_{L_T^{\infty} H^s}.
\end{align}
In a similar manner as the proof of Lemma~\ref{lem_go}, it follows that $\mathcal{N}_f^{(3)} (\ha{w}_1, \ha{w}_2, \ha{w}_3) \in C([-T, T]: l_{s-1}^2)  $. 
By Lemma~\ref{lem_go} and (\ref{nle610}), we have
\begin{align} \label{nle612}
\| \mathcal{N}_f^{(3)} (\ha{w}_1)-\mathcal{N}_{g}^{(3)} (\ha{w}_1) \|_{L_T^{\infty} l_{s-1}^2} \le C_*(w_1,s,|E_0(f)-E_0(g)|,|\beta|,T).
\end{align}
By the definition of $M_{j,f}^{(N)}$ with $(N, j) \in I_5$, 
it suffices to show that for $(N, j) \in J_1 \cup J_2 \cup J_3$ and $v, w \in C([-T, T]: H^s(\T))$ 
\begin{align}
& \| \La_f^{(N)} ( \ti{L}_{j, f}^{(N)} \chi_{>L}^{(N)} , \ha{v}, \dots, \ha{v}, \mathcal{N}_{f}^{(3)} (\ha{v}) ) \|_{L_T^{\infty} l_{s}^2} 
\lesssim \| v \|_{L_T^{\infty} H^s}^{N+2}, \label{nle61} \\
& \| \La_f^{(N)} ( \ti{L}_{j, f}^{(N)} \chi_{>L}^{(N)} , \ha{v}, \dots, \ha{v}, \mathcal{N}_{f}^{(3)} (\ha{v}) ) 
-\La_f^{(N)} ( \ti{L}_{j, f}^{(N)} \chi_{>L}^{(N)} , \ha{w}, \dots, \ha{w}, \mathcal{N}_{f}^{(3)} (\ha{w}) ) \|_{L_T^{\infty} l_s^2} \nonumber \\ 
& \hspace{0.5cm} \lesssim (\| v \|_{L_T^{\infty} H^s}+ \| w \|_{L_T^{\infty} H^s})^{N+1} \| v- w \|_{L_T^{\infty} H^s} \label{nle62} 
\end{align}
and
\begin{align} \label{nle63}
 \| \La_f^{(N)} ( \ti{L}_{j, f}^{(N)} \chi_{>L}^{(N)} , \ha{v}, \dots, \ha{v}, \mathcal{N}_{f}^{(3)} (\ha{v}) ) 
- \La_g^{(N)} ( \ti{L}_{j, g}^{(N)} \chi_{>L}^{(N)} , \ha{v}, \dots, \ha{v}, \mathcal{N}_{g}^{(3)} (\ha{v}) )  \|_{L_T^{\infty} l_s^2} \le C_*.
\end{align}
For $(N,j) \in J_1 \cup J_2 \cup J_3$, by (\ref{pwb3})--(\ref{pwb2}) and (\ref{pwb21}), we have
\begin{align} 
& |\ti{L}_{j,f}^{(N)} \chi_{>L}^{(N)}| \lesssim \langle k_{\max} \rangle^{-1}, \label{nle613} \\
& |(\ti{L}_{j,f}^{(N)} -\ti{L}_{j,g}^{(N)}) \chi_{>L}^{(N)}| \lesssim |\be| |E_0(f)-E_0(g)| \langle k_{\max} \rangle^{-1}. \label{nle614} 
\end{align}
By (\ref{nl42}), (\ref{nle611}) and (\ref{nle613}), we get 
\begin{align} \label{nle615}
& \Big\| \sum_{k=k_{1,\dots,  N}} |\ti{L}_{j, f}^{(N)} \chi_{>L}^{(N)} | \, \prod_{l=1}^{N-1} |\ha{v}_l(t, k_l)| 
\, \big| \mathcal{N}_f^{(3)} (\ha{v}_N, \ha{v}_{N+1}, \ha{v}_{N+2})  \big| \Big\|_{L_T^{\infty} l_s^2 } \notag \\
& \hspace{0.5cm} \lesssim \prod_{l=1}^{N-1} \| v_l \|_{L_T^{\infty} H^s } \big\| \mathcal{N}_f^{(3)} (\ha{v}_N, \ha{v}_{N+1}, \ha{v}_{N+2}) \big\|_{L_T^{\infty}l_{s-1}^2} \\
& \hspace{0.5cm} \lesssim \prod_{l=1}^{N+2} \| v_l \|_{L_T^{\infty} H^s} \notag
\end{align}
for any $\{ v_l \}_{l=1}^{N+2} \subset C([-T, T] :H^s(\T))$, which implies that (\ref{nle61}) and (\ref{nle62}). 
In a similar way to the proof of (\ref{nle53}), 
by (\ref{nl42}), (\ref{nle611})--(\ref{nle612}), (\ref{nle613})--(\ref{nle615}) and Lemma~\ref{lem_go}, we obtain (\ref{nle63}). 
\end{proof}

The proposition below and Proposition~\ref{prop_res1} say that the nonlinear estimates hold with $s \ge 1/2$ 
for all resonant parts and nearly resonant parts defined in Proposition~\ref{prop_NF2}. 

\begin{lem} \label{lem_nl12}
Let $s \ge 1/2$, $f, g \in L^2(\T)$ and $M_{j, f}^{(N)}$ with 
$(N,j) \in I_6 \cup I_7$ be as in Proposition~\ref{prop_NF2}. 
Then, for any $v,w \in C([-T,  T]: H^s(\T))$ and \\
$L \gg \max\{ 1, | \be E_0(f)|, |\be E_0(g)| \}$, we have 
\begin{align}
& \| \La_{f}^{(N)} (\ti{M}_{j, f}^{(N)}, \ha{v} ) \|_{L_T^{\infty} l_s^2} \lesssim  \| v \|_{L_T^{\infty} H^s}^N, \label{nle71}  \\
& \| \La_{f}^{(N)} (\ti{M}_{j, f}^{(N)}, \ha{v} ) - \La_f^{(N)} (\ti{M}_{j, f}^{(N)} , \ha{w} ) \|_{L_T^{\infty} l_s^2} \notag \\
& \hspace{3cm} \lesssim  (\| v \|_{L_T^{\infty} H^s}+ \| w \|_{L_T^{\infty} H^s})^{N-1}  \| v-w  \|_{L_T^{\infty} H^s}. \label{nle72}
\end{align}
Moreover, it follows that 
\begin{align}
 \|  \La_{f}^{(N)} (\ti{M}_{j, f}^{(N)}, \ha{v} ) - \La_g^{(N)} (\ti{M}_{j, g}^{(N)} , \ha{v} )   \|_{L_T^{\infty} l_s^2} 
\le C_*  \label{nle73} 
\end{align}
where $C_*=C_*(v,s,|E_0(f)-E_0(g)|,|\be|,T) \ge 0$ and $C_*\to 0$ when $|E_0(f)-E_0(g)|\to 0$.
\end{lem}

\begin{proof}
(I) First, we handle the case $(N, j) \in I_6$. 
By Proposition~\ref{prop_res3} and \eqref{pwb15}, 
it follows that $|M_{j, f}^{(N)}| \lesssim \langle k_1 \rangle (\chi_{R1}^{(3)} (k_1, k_2, k_3) + \chi_{R1}^{(3)} (k_1,k_3, k_2)  )$.  
Thus, by Lemma~\ref{lem_nes01},  we have 
\begin{align}
& \Big\|  \sum_{k=k_{1, \dots, N}} |M_{j, f}^{(N)}| \, \prod_{l=1}^N | \ha{v}_l (t, k_l) |  \Big\|_{L_T^{\infty} l_s^2 }
 \lesssim \prod_{l=1}^N \| v_l \|_{L_T^{\infty} H^s }, \label{go41} \\
 &  \Big\| \sum_{k=k_{1,\dots,N}} |\ti{M}_{j, f}^{(N)}| \, \prod_{l=1}^N |\ha{v}_l(t, k_l)| \Big\|_{L_T^{\infty} l_s^2 } 
\lesssim \prod_{l=1}^N  \| v_l \|_{L_T^{\infty} H^s} \label{go42}
\end{align}
for any $\{ v_l \}_{l=1}^{N} \in C([-T, T] :H^s(\T))$. 
(\ref{go42}) implies (\ref{nle71}) and (\ref{nle72}). 
By (\ref{go41}), (\ref{go42}) and Lemma~\ref{lem_go}, we get (\ref{nle73}). \\
(II) Next, we deal with the case $(N, j) \in I_7$. First, we show (\ref{go42}). \\
(IIa) Estimate of $M_{1,f}^{(3)}$: By the H\"{o}lder inequality and the continuous embedding $l^2 \hookrightarrow l^6$, we have
\begin{align*}
& \big\| \sum_{k=k_{1,2,3}}  |M_{1,f}^{(3)}| \, \prod_{l=1}^3 | \ha{v}_l (t, k_l)|   \big\|_{l_s^2}
 \lesssim \big\| \prod_{l=1}^3 \big( \langle k \rangle^{s/3+1/3}  |\ha{v}_l (t, k)|  \big)  \big\|_{l^2}  \\
& \hspace{0.5cm} 
\le \prod_{l=1}^3 \big\|  \langle k \rangle^{s/3+1/3} |\ha{v}_l(t,k)|  \big\|_{l^6} \lesssim \prod_{l=1}^3 \| v_l (t) \|_{H^{s/3+1/3}}
 \le\prod_{l=1}^3 \| v_l (t) \|_{H^s}
\end{align*}
for any $t \in [-T, T]$. Here we used $s \ge 1/2$ in the last inequality. This implies (\ref{go42}). \\
(IIb) Estimate of $\ti{M}_{5, f}^{(4)}$: 
By Proposition~\ref{prop_res2}, it follows that $|\tilde{M}_{5,f}^{(4)}| \lesssim [ \chi_{H1}^{(4)} \chi_{R1}^{(4)} ]_{sym}^{(4)}$. 
We notice that $(k_1, k_2, k_3, k_4) \in \supp \chi_{H1}^{(4)} \chi_{R1}^{(4)}$ means that $k_{1,2,3}=0$ and 
$|k_4| > 4^3 \max \{ |k_1|, |k_2| , |k_3| \}$.
Thus, by the H\"{o}lder inequality and the Sobolev inequality, we have 
\begin{align*} 
& \big\|  \sum_{k=k_{1,2,3,4}} \chi_{H1}^{(4)} \chi_{R1}^{(4)} \, \prod_{l=1}^4 |\ha{v}_l (t, k_l) |  \big\|_{l_s^2} 
 \lesssim \sum_{k_{1,2,3}=0} |\ha{v}_1 (t, k_1)| |\ha{v}_2 (t,k_2)| |\ha{v}_3 (t,k_3)| \, \|  v_4(t) \|_{H^s} \notag \\
& \hspace{1cm} 
\lesssim \prod_{l=1}^3 \|  \mathcal{F}^{-1} [ |\ha{v}_l (t) | ]  \|_{L^3} \| v_4(t) \|_{H^s} \lesssim \prod_{l=1}^4 \| v_l(t) \|_{H^s}
\end{align*} 
for any $t \in [-T, T]$ and $s \ge 1/6$. This leads that
\begin{align*}
\Big\| \sum_{k=k_{1,2,3,4}} [  \chi_{H1}^{(4)} \chi_{R1}^{(4)} ]_{sym}^{(4)} \, \prod_{l=1}^4 |\ha{v}_l(t, k_l) | \Big\|_{L_T^{\infty} l_s^2} 
\lesssim \prod_{l=1}^4 \| v_l \|_{L_T^{\infty} H^s}. 
\end{align*}
Thus, by $|\ti{M}_{5, f}^{(4)}| \lesssim [\chi_{H1}^{(4)} \chi_{R1}^{(4)} ]_{sym}^{(4)}$, we obtain (\ref{go42}). \\
(IIc) Estimate of $M_{6, f}^{(4)}$: Suppose that $|k_1| \le |k_2| \le |k_3|$ holds. 
Then, by \eqref{pwb12}, $|M_{6,f}^{(4)}| \lesssim \langle k_1  \rangle^{-1}  \langle k_2 \rangle^{1/8} \langle k_3 \rangle^{1/8} \chi_{H1}^{(4)} \chi_{R5}^{(4)} $. 
Since $(k_1,k_2, k_3, k_4) \in \supp \chi_{H1}^{(4)} \chi_{R5}^{(4)} $ leads that $|k_{2,3}| \lesssim |k_1| \le |k_2| \sim |k_3|$
 and $k_{\max}=|k_4|$, we have 
$$
|M_{6, f}^{(4)}| \lesssim \langle k_1 \rangle^{-1/3} \langle k_{2,3} \rangle^{-7/6} \langle k_2 \rangle^{3/8} 
\langle k_3 \rangle^{3/8} \chi_{H1}^{(4)}. 
$$
Thus, by the Young inequality and the H\"{o}lder inequality, we have
\begin{align*}
& \big\|  \sum_{k=k_{1,2,3,4}} |M_{6, f}^{(4)}| \, \prod_{l=1}^4 |\ha{v}_l (t, k_l) |  \big\|_{l_s^2} \\
&\hspace{0.5cm} \lesssim 
\big\| \langle \cdot \rangle^{-1/3} |\ha{v}_1 (t)| * ( \langle \cdot \rangle^{-7/6} 
 ( \langle \cdot \rangle^{3/8} |\ha{v}_2(t)|* \langle \cdot \rangle^{3/8} |\ha{v}_3(t)| ) ) 
* \langle \cdot  \rangle^{s} |\ha{v}_4(t)|  \big\|_{l^2}   \\
& \hspace{0.5cm} \lesssim \| \langle \cdot \rangle^{-1/3} |\ha{v}_1(t)| \|_{l^1} 
\| \langle \cdot \rangle^{3/8} |\ha{v}_2(t)| * \langle  \cdot \rangle^{3/8} |\ha{v}_3(t)|  \|_{l^{\infty}}
 \| \langle \cdot  \rangle^{s} |\ha{v}_4(t)| \|_{l^2}\\
& \hspace{0.5cm} \lesssim \prod_{l=1}^4  \| v_l(t) \|_{H^s}
\end{align*}
for any $t \in [-T, T]$ and $s \ge 3/8$.  This implies (\ref{go42}). 

By (IIa)--(IIc), we obtain (\ref{go42}) for $(N, j) \in I_7$. 
(\ref{go42}) leads (\ref{nle71}) and (\ref{nle72}). 
Since $M_{j,f}^{(N)}=M_{j, g}^{(N)}$ for $(N,j) \in I_7$, by Lemma~\ref{lem_go} and (\ref{go42}), 
we get (\ref{nle73}). 
\end{proof}

Now, we prove Proposition~\ref{prop_main1}. 
\begin{proof}[Proof of Proposition~\ref{prop_main1}] 
By the definition of $F_{f, L}$ and Lemma~\ref{lem_nes0}, we have (\ref{mes11}) and (\ref{mes13}).  
For $(N,j) \in J_1 \cup J_2 \cup J_3$, we can easily check
\begin{align} \label{mes20}
|L_{j, f}^{(N)} \Phi_{f}^{(N)} \chi_{ \le L}^{(N)} | \lesssim L^3. 
\end{align} 
By (\ref{mes20}), we have 
\begin{align} \label{mes21}
\Big\| \sum_{k=k_{1, \dots, N}} |\ti{L}_{j, f}^{(N)} \Phi_{f}^{(N)} \chi_{\le L}^{(N)}| \, \prod_{l=1}^N |\ha{v}_l(t, k_l)| \Big\|_{L_T^{\infty} l_s^2 } 
\lesssim L^{3} \prod_{l=1}^N \| v_l \|_{L_T^{\infty} H^s}
\end{align}
for any $\{ v_l \}_{l=1}^N \subset C([-T, T]: H^s(\T))$. 
Thus, by the definition of $G_{f, L}$, Lemmas~\ref{lem_nl10}--\ref{lem_nl12} and Proposition \ref{prop_res1}, we have (\ref{mes12}). 
Next, we prove (\ref{mes14}). 
For $(N, j) \in J_1 \cup J_2 \cup J_3$, by (\ref{pwb31}), (\ref{pwb32}) and (\ref{mes20}), it follows that 
\begin{align}
|\ti{L}_{j, f}^{(N)} \Phi_{f}^{(N)} \chi_{\le L}^{(N)}- \ti{L}_{j, g}^{(N)} \Phi_g^{(N)} \chi_{\le L}^{(N)}  | 
& \lesssim |\be| |E_0(f)- E_0(g)| [ |L_{j, f}^{(N)} \Phi_f^{(N)} \chi_{\le L}^{(N)}  | ]_{sym}^{(N)} \notag \\
& \lesssim  |\be| |E_0(f)- E_0(g)| L^3 .  \label{mes23}
\end{align}
By (\ref{mes21}), (\ref{mes23}) and Lemma~\ref{lem_go}, we have 
\begin{align*}
\big\|  \La_f^{(N)} ( \ti{L}_{j, f}^{(N)} \Phi_f^{(N)} \chi_{\le L}^{(N)}, \ha{v} )
- \La_g^{(N)} ( \ti{L}_{j, g}^{(N)} \Phi_g^{(N)} \chi_{\le L}^{(N)}, \ha{v} )  \big\|_{L_T^{\infty} l_s^2} \le C_*.
\end{align*}
Hence, by the definition of $G_{f, L}$, Lemmas~\ref{lem_nl10}--\ref{lem_nl12} and Proposition \ref{prop_res1}, we obtain (\ref{mes14}). 
\end{proof}

\section{Proof of Theorem~\ref{thm_main}}

In this section, we give the proof of Theorem~\ref{thm_main}.
When $u\in C([-T,T]:H^1(\T))$ is a solution to \eqref{5KdV1} or \eqref{5KdV2} with \eqref{ini}, we have the conservation law \eqref{cons1}.
Thus, the unconditional local well-posedness for \eqref{5KdV1} with \eqref{ini} is equivalent to that for \eqref{5KdV2} with \eqref{ini}.
Moreover, Propositions \ref{prop_equiv0} and \ref{prop_equiv} below mean that
the unconditional local well-posedness for \eqref{5KdV2} with \eqref{ini} is equivalent to that for \eqref{5KdV3} with \eqref{ini}.
Here, we define the translation operators $\mT_v$ and $\mT_v^{-1}$ by 
\begin{align*}
\mT_v u(t,x):=u \Big(t,x+\int_0^t [K(v)](t')\, dt'\Big), \  \mT_v^{-1} u(t,x):=u \Big(t,x-\int_0^t [K(v)](t')\, dt'\Big)
\end{align*} 
for $v \in C([-T,T]:L^2(\T))$.

\begin{prop} \label{prop_equiv0}
Let $s \ge 0$. Then, a map $\mT: u \mapsto \mT_u u$ is a homeomorphism on $C([-T, T]:H^s(\T))$. 
\end{prop}

Before we prove Proposition \ref{prop_equiv0}, we prepare the following lemma. 
\begin{lem}\label{lem_trans1}
Let $s\in \R$. 
Assume that $v, v_n \in C([-T,T]: L^2(\T))$ satisfy $v_n \to v$ in $C([-T,T]:L^2(\T))$ as $n\to \infty$. Then, we have the followings:

(i) $\mT_v$, $\mT_{v}^{-1}$ are maps from $C([-T,T]:H^s(\T))$ to itself 
and 
$$\|[\mT_v u](t)\|_{H^s}=\|u(t)\|_{H^s}, \hspace{0.5cm} \| \mT_{v}^{-1} u(t) \|_{H^s}=\| u (t) \|_{H^s}$$ 
for $u \in C([-T, T]: H^s(\T))$ and $t \in [-T, T]$. 

(ii) Let $u\in C([-T,T]:L^2 (\T))$. Then, $\mT_u=\mT_{\mT_v u}= \mT_{\mT_{v}^{-1} u}$ and $\mT_u^{-1}= \mT_{\mT_{v} u}^{-1}$ hold.

(iii) Let $u\in C([-T,T]:H^s(\T))$. Then, it follows that 
$\mT_{v_n} u \to \mT_{v} u$, $\mT_{v_n}^{-1} u \to \mT_v^{-1} u$ 
in $C([-T,T]:H^s(\T))$ as $n\to \infty$.
\end{lem}

\begin{proof}
Let $u\in C([-T,T]:H^s(\T))$.
Since $\mT_v$, $\mT_v^{-1}$ are translation operators with respect to $x$, we have $\|[\mT_v u](t)\|_{H^s}=\|u(t)\|_{H^s}$ and 
$\| [ \mT_{v}^{-1} u ] (t)\|_{H^s} =\| u (t) \|_{H^s} $.
When $t_1\to t_2$,
we have
\EQQ{
&\|[\mT_vu](t_1)-[\mT_vu](t_2)\|_{H^s}\\
\le&\Big\|u\Big(t_1,x+\int_0^{t_1}[K(v)](t')\,dt'\Big)-u\Big(t_2,x+\int_0^{t_1}[K(v)](t')\,dt'\Big)\Big\|_{H^s}\\
+&\Big\|u\Big(t_2,x+\int_0^{t_1}[K(v)](t')\,dt'\Big)-u\Big(t_2,x+\int_0^{t_2}[K(v)](t')\,dt'\Big)\Big\|_{H^s}\\
=&\|u(t_1,x)-u(t_2,x)\|_{H^s}
+\Big\|u(t_2,x)-u\Big(t_2,x+\int_{t_1}^{t_2}[K(v)](t')\,dt'\Big)\Big\|_{H^s}\to 0.
}
Thus, $\mT_v u$ is in $C([-T,T]:H^s (\T))$. In the same manner, $\mT_{v}^{-1} u $ is in $C([-T, T]:H^s(\T))$. 
Hence, we obtain (i). 

If $u \in C([-T, T]:L^2(\T)) $, then $K(u)=K(\mT_v u)=K (\mT_v^{-1} u)$ holds. Thus, by the definition of the translation operator, we obtain (ii).

Finally, we show (iii).
When $n\to \infty$, it follows that
\EQQ{
&\|\mT_{v_n}u-\mT_{v}u\|_{L^\infty_TH^s}\\
=&\Big\|u\Big(t,x+\int_0^{t}[K(v_n)](t')\,dt'\Big)-u\Big(t,x+\int_0^{t}[K(v)](t')\,dt'\Big)\Big\|_{L^\infty_TH^s}\\
=&\Big\|u\Big(t,x+\int_0^{t}[K(v_n)](t')-[K(v)](t')\,dt'\Big)-u(t,x)\Big\|_{L^\infty_TH^s}\to 0
}
because $\|[K(v_n)](t)-[K(v)](t)\|_{L^\infty_T} \to 0$ as $n\to \infty$. 
In the same manner, it follows that $\mT_{v_n}^{-1} u \to \mT_{v}^{-1} u$ in $C([-T, T]: H^s(\T))$ 
as $n \to \infty$. 
\end{proof}

Now, we prove Proposition~\ref{prop_equiv0}. 
\begin{proof}[Proof of Proposition~\ref{prop_equiv0}] 
By (i) of Lemma~\ref{lem_trans1}, $\mT$ is well-defined as a map on \\ $C([-T, T] :H^s(\T))$. 
Now, we put $\mT^{-1} u := \mT_{u}^{-1} u$ for $u \in C([-T, T] :H^s(\T) )$. 
Then, $\mT^{-1}$ is the inverse map of $\mT$. 
In fact, by (i) and (ii) of Lemma~\ref{lem_trans1}, it follows that 
\begin{align*}
& \mT^{-1}(\mT u) = \mT_{\mT_u u}^{-1} (\mT_u u)= \mT_u^{-1} (\mT_u u)=u, \\
& \mT(\mT^{-1} u)= \mT_{\mT_u^{-1} u} (\mT_u^{-1} u)= \mT_u (\mT_u^{-1} u)=u 
\end{align*}
for any $u \in C([-T, T]: H^s(\T))$. 
Moreover, these properties imply the bijection of $\mT$ and $\mT^{-1}$. 


Next, we show that the map $\mT$ is continuous. 
Assume that $u_n, u \in C([-T, T]:H^s(\T))$ satisfy $u_n \to u$ in $C([-T, T]:H^s(\T) )$. 
Then, by (iii) of Lemma~\ref{lem_trans1}, we have 
\begin{align*}
\mT u_n - \mT u= \mT_{u_n} u_n - \mT_{u} u =\mT_{u_n}(u_n-u )+ (\mT_{u_n}-\mT_u) u \to 0
\end{align*}
in $C([-T, T] :H^s(\T))$ as $n \to \infty$. In the same manner, it follows that $\mT^{-1}$ is continuous. 
\end{proof}

\begin{prop}\label{prop_equiv}
Let $s \ge 1$.

(i) If $u$ is a solution to \eqref{5KdV2} with \eqref{ini} and $u\in C([-T,T]: H^s(\T))$,
then $\mT_{u} u$ is a solution to \eqref{5KdV3} with \eqref{ini} and $\mT_{u} u \in C([-T,T]: H^s(\T))$.

(ii) If $u$ is a solution to \eqref{5KdV3} with \eqref{ini} and $u\in C([-T,T]: H^s(\T))$,
then $\mT_{u}^{-1} u$ is a solution to \eqref{5KdV2} with \eqref{ini} and $\mT_{u}^{-1} u \in C([-T,T]: H^s(\T))$. 
\end{prop}

\begin{proof}
First, we show (i).
Assume that $u\in C([-T,T]:H^s(\T))$ is a solution to \eqref{5KdV2} with \eqref{ini}. 
Obviously, $\mT_u J_i(u)=J_i(\mT_u u)$ for $i=1,2$ and $3$ because $K(u)$ does not depend on $x$.
Since $\mT_u (\p_t u+K(u)\p_x u)= \p_t \mT_u (u)$, we conclude $\mT_u u$ satisfies \eqref{5KdV3}.
Moreover, by (i) of Lemma \ref{lem_trans1}, we have $\mT_u u \in C([-T,T]:H^s(\T))$
In the same manner, we also have (ii). 
\end{proof}

By Propositions~\ref{prop_equiv0} and \ref{prop_equiv}, Theorem \ref{thm_main} is equivalent to Proposition \ref{prop_LWP} below and we only need to show it.
\begin{prop} \label{prop_LWP}
Let $s \geq 1$ and $\vp \in H^s (\T)$. Then, there exist $T=T(\| \vp \|_{H^s}) >0$ and a unique solution $u \in C([-T,T]: H^{s} (\mathbb{T}))$ to 
\eqref{5KdV3} with \eqref{ini}. 
Moreover the solution map, $H^s(\mathbb{T}) \ni \varphi \mapsto u \in C([-T,T]: H^s (\mathbb{T}))$ is continuous.  
\end{prop}
\begin{rem}
As we proved in Proposition \ref{prop_NF2}, for any solution $u \in C([-T,T]:H^s(\T))$ to \eqref{5KdV3}, $\ha{v}=e^{-t\phi_\vp(k)}\ha{u}$ satisfies \eqref{NF21}.
By the standard argument with Proposition \ref{prop_main1}, 
we can prove the existence of solutions of \eqref{NF21}.
However, this argument is useless to show the existence of solutions to \eqref{5KdV3} since we do not know whether $u:=U_\vp(t) v$  satisfy \eqref{5KdV3} or not when $\ha{v}$ satisfy \eqref{NF21}.
To avoid this difficulty, we use the existence of the solution to \eqref{5KdV1}--\eqref{ini} for sufficiently smooth initial data.
\end{rem}

In \cite{T},  the second author has proved the local well-posedness of fifth order dispersive equations for sufficiently smooth data 
by the modified energy method. 
We can easily check that the nonlinear term of (\ref{5KdV1}) is non-parabolic resonance type. 
Therefore, we have the following result by Theorem 1.1 in \cite{T}. 
\begin{prop}\label{prop_existence}
Let $m \in \N$ be sufficiently large.

Then, \eqref{5KdV1}--\eqref{ini} is locally well-posed in $H^m(\T)$ on $[-T,T]$. 
The existence time $T>0$ depends only on $\|\vp\|_{H^m}$.
\end{prop}
\begin{rem}\label{rem_time}
Since $u(T), u(-T)\in H^{m}(\T)$, we can extend the solution on $[-T-T',T+T']$
where $T'=T'(\|u(T)\|_{H^{m}},\|u(-T)\|_{H^{m}})>0$.
Iterating this process, the solution can be extended on $(-T^{max},T^{max})$ 
where $T^{max}$ satisfies
$$\liminf_{t\to T^{max}} \|u(t)\|_{H^{m}} =\infty, \ \liminf_{t\to -T^{max}} \|u(t)\|_{H^{m}} =\infty \ \ \text{or} \ \  T^{max}=\infty.$$
\end{rem}

\begin{proof}[Proof of Proposition~\ref{prop_LWP}]
Firstly, we show the existence of the solution. 
For any $\vp \in H^s(\T)$, we can take $\{ \vp_n  \}_{n=1}^{\infty} \subset H^{\infty} (\T)$ such that 
$\vp_n \to \vp$ in $H^s(\T)$, $\| \vp_n \|_{H^s} \le \| \vp \|_{H^s}$ and $E_0(\vp_n)=E_0(\vp)$, for instance, 
take $\vp_n= \mathcal{F}_{k}^{-1} [ e^{-|k|/n} \ha{\vp}(k)  ] $. 
By Proposition~\ref{prop_existence} and Remark ~\ref{rem_time}, 
we have the existence time $T_n^{max}>0$ and the solution 
$w_n \in C((-T_{n}^{max}, T_n^{max}):H^{\infty}(\T))$ to (\ref{5KdV1}) with initial data $\vp_n$. 
By the conservation law \eqref{cons1} and (i) of Proposition~\ref{prop_equiv}, $u_n:=\mT w_n$ satisfies (\ref{5KdV3}). 
Note that, by the definition of $\mT$ and the conservation law (\ref{cons1}), it follows that 
\begin{align*}
E_0( u_n) (t)= E_0(w_n)(t)=E_0(\vp_n)=E_0(\vp)
\end{align*} 
for $t \in (-T_n^{\max}, T_n^{\max})$ and $u_n$ satisfies 
\begin{align} \label{5thKdV}
\p_t u_n +\p_x^5 u_n +\be E_0(\vp) \p_x^3 u_n  =J_1(u_n)+ J_2 (u_n)+ J_3(u_n).
\end{align}
Thus, by (\ref{mes31}), there exists a constant $C \ge 1$ such that the following estimate holds for any 
$0< \delta < T_n^{max}$ and $L \gg \max \{ 1, |\be|  \| \vp \|_{H^s} \}$:
\EQS{ \label{es71}
\| u_n \|_{L_{\de}^{\infty}H^s} \leq &  \| \vp_n \|_{H^s}+ C L^{-1}  (1+ \| u_n \|_{L_{\de}^{\infty} H^{s} } )^3 \| u_n \|_{L_{\delta}^{\infty} H^s} \nonumber \\
& + C \delta L^3 (1+ \| u_n  \|_{L_{\delta}^{\infty} H^s })^5 \| u_n \|_{L_{\delta}^{\infty} H^s}. 
}
Here, we take sufficiently large $L$ and small $T_0>0$ such that
\begin{align} \label{co71}
CL^{-1} (1 + 4\| \vp  \|_{H^s})^3 \le \frac{1}{4}, \ \ C T_0 L^{3} (1+4 \| \vp \|_{H^s})^5 \le \frac{1}{4}.
\end{align}
If $0 < \delta < T:=\min\{T_n^{\max}, T_0\}$ and $\| u_n \|_{L_{\de}^{\infty} H^{s} } \le 4 \| \vp \|_{H^{s}}$, 
then, by (\ref{es71}) and (\ref{co71}), we have 
\EQ{ \label{es72}
\| u_n \|_{L_{\de}^{\infty} H^{s}} \le 2 \| \vp_n \|_{H^{s}} \le 2 \| \vp \|_{H^{s}}. 
}
Therefore, by continuity argument, 
we obtain (\ref{es72}) for $\de=T$ without the assumption $\| u_n \|_{L_{\de}^{\infty} H^{s}} \leq 4 \| \vp \|_{H^s} $. 
By Remark~\ref{rem_time}, it follows that $T< T_n^{max}$. Thus, $T=T_0$. 
Moreover, by (\ref{mes32}) with $u_1=u_m$ and $u_2=u_n$ and $E_0(\vp_m)=E_0(\vp_n)=E_0(\vp)$, we have 
\begin{align} \label{es73}
\| u_m- u_n \|_{L_T^{\infty} H^s} & \le \| \vp_m - \vp_n \|_{ H^s} \notag \\
& + CL^{-1} (1+ \| u_m \|_{L_T^{\infty} H^s}+ \|  u_n \|_{L_T^{\infty} H^s })^{3} \| u_m - u_n \|_{L_T^{\infty} H^s} \notag \\
& + C T L^3  (1+ \| u_m \|_{L_T^{\infty} H^s} + 
\| u_n \|_{L_T^{\infty} H^s})^5  \|  u_m- u_n \|_{L_T^{\infty} H^s}.
\end{align}
Applying (\ref{co71}) and (\ref{es72}) with $\de=T$ to (\ref{es73}), we have  
\begin{align*}
\| u_m- u_n \|_{L_T^{\infty} H^s} \le 2 \| \vp_m -\vp_n  \|_{H^s}  \to 0
\end{align*}
as $m,n \to \infty$. 
Since $u_n$ satisfies \eqref{5thKdV} and \eqref{es72} with $\de=T$, we also have
\[
\|\p_t (u_n-  u_m)\|_{L_T^{\infty} H^{s-5}} \lesssim (1+|\be| \| \vp \|_{H^s} + \| \vp \|_{H^s}^2)  \|u_n-u_m\|_{L_{T}^{\infty} H^{s}} \to 0
\]
as $m, n \to \infty$.
Therefore, there exists $u \in C([-T, T]):H^s(\T)) \cap C^1([-T, T]:H^{s-5}(\T))$ such that $u_n \to u$ strongly in $C([-T, T]:H^s(\T))\cap C^1([-T, T]:H^{s-5}(\T))$.
Then $u$ satisfies \eqref{5KdV3} in the sense of $C([-T, T]:H^{s-5}(\T))$. 

Next, we prove the uniqueness.
Assume that $u_1, u_2 \in C([-T, T] :H^s(\T))$ satisfy \eqref{5KdV3} with \eqref{ini} for the same initial data $\vp\in H^s(\T)$.
Then, by (\ref{mes32}), we have 
\begin{align*}
&\| u_1- u_2 \|_{L_\de^{\infty} H^s} \lec \| u_1- u_2 \|_{L_\de^{\infty} H^s}\\
& \times \left(L^{-1} (1+ \| u_1 \|_{L_\de^{\infty} H^s} + \| u_2 \|_{L_\de^{\infty} H^s})^3+
\de L^3  (1+ \| u_1 \|_{L_\de^{\infty} H^s} + \| u_2 \|_{L_\de^{\infty} H^s})^5\right).
\end{align*}
Therefore, we obtain $u_1(t)=u_2(t)$ on $[-\de,\de]$ by taking sufficiently small $\de>0$ and large $L$.
Note that $\de$ and $L$ depend only on $\| u_j \|_{L_T^{\infty} H^s}$ for $j=1,2$ because
$\| u_j \|_{L_\de^{\infty} H^s} \le \| u_j \|_{L_T^{\infty} H^s}$ for $j=1,2$.
Iterating this argument, we obtain $u_1(t)=u_2(t)$ on $[-T,T]$. 

Finally, we can easily show the proof of the continuity of the solution map by Corollary \ref{cor_mainest2}.
\end{proof}


\end{document}